\documentclass[aip,jmp]{revtex4-1}
\usepackage{graphicx}
\usepackage{dcolumn}
\usepackage{bm}
\usepackage{amssymb}
\usepackage{amsmath,amsthm}
\usepackage{mathptmx}
\usepackage{tikz}
\usetikzlibrary{arrows}
\usetikzlibrary{decorations.markings}

\makeatletter

\newcommand{\Rmnum}[1]{\expandafter\@slowromancap\romannumeral #1@}
\makeatother
\newfam\msbfam
\font\twelmsb=msbm10 at 12pt 
\font\sevenmsb=msbm10 at 7pt \font\fivemsb=msbm10 at 5pt
\textfont\msbfam=\twelmsb
\scriptfont\msbfam=\sevenmsb \scriptscriptfont\msbfam=\fivemsb

\newtheorem{theorem}{Theorem}[section]
\newtheorem{example}[theorem]{Example}
\newtheorem{corollary}[theorem]{Corollary}
\newtheorem{definition}[theorem]{Definition}
\newtheorem{proposition}[theorem]{Proposition}
\newtheorem{lemma}[theorem]{Lemma}
\newtheorem{remark}[theorem]{Remark}

\def\N{\mathbb{N}}
\def\Q{\mathbb{Q}}
\def\Z{\mathbb{Z}}
\def\C{\mathbb{C}}
\def\F{\mathcal{F} }
\def\Q{\mathbb{Q} }
\def\uu{\mathbf{u} }
\def\xx{\mathbf{x} }
\def\g{\mathfrak{g} }
\def\a{\alpha }
\def\hg{\widehat{\mathfrak g}}
\def\CC{\mathcal{C}}

\def\A{\mathcal{A} }
\def\Si{\Sigma}
\def\M{\mathcal{M}}
\allowdisplaybreaks
\numberwithin{equation}{section}

\begin{document}
\preprint{}

\title{Cluster algebra structure on the finite dimensional representations of $U_q(\widehat{A_{3}})$ for $\ell=2$}

\author{Yan-Min Yang, Zhu-Jun Zheng}
\thanks {}\email{zhengzj@scut.edu.cn}
\affiliation{Department of Mathematics, South China University of
 Technology, Guangzhou 510641, P. R. China.}

\begin{abstract}
In this paper, we prove one case of the conjecture given by Hernandez and Leclerc\cite{HL0}. Specifically, we give a cluster algebra structure on the Grothendieck ring of a full subcategory of the finite dimensional representations of a simply-laced quantum affine algebra $U_q(\widehat{\g})$. In the procedure, we also give a specific description of compatible subsets of type $E_{6}$. As a conclusion, for every exchange relation of cluster algebra there exists a exact sequence of the full subcategory corresponding to it.
\end{abstract}


\maketitle

\section{Introduction}
Cluster algebras were introduced by Fomin and Zelevinsky  \cite{FZ0} in 2002 in the context of total positivity and canonical bases in Lie theory. Since then cluster algebras have been shown to be related to various fields in mathematics, including those indicated in FIGURE 1.

There are two different categorical approaches for cluster algebras. One is an additive category (cluster category \cite{BMRRT}) which is defined as an orbit category of the derived category of finite dimensional representations of a quiver $\Gamma$ under an action of an automorphism. Another one is a monoidal category \cite{HL0} which is a full subcategory of finite dimensional representations of a simply-laced quantum affine algebra. The Grothendieck ring of this monoidal category is a cluster algebra. We use a table (TABLE 1) to illustrate the relationships between the two approaches.

\begin{figure}[H]
\begin{center}
\includegraphics {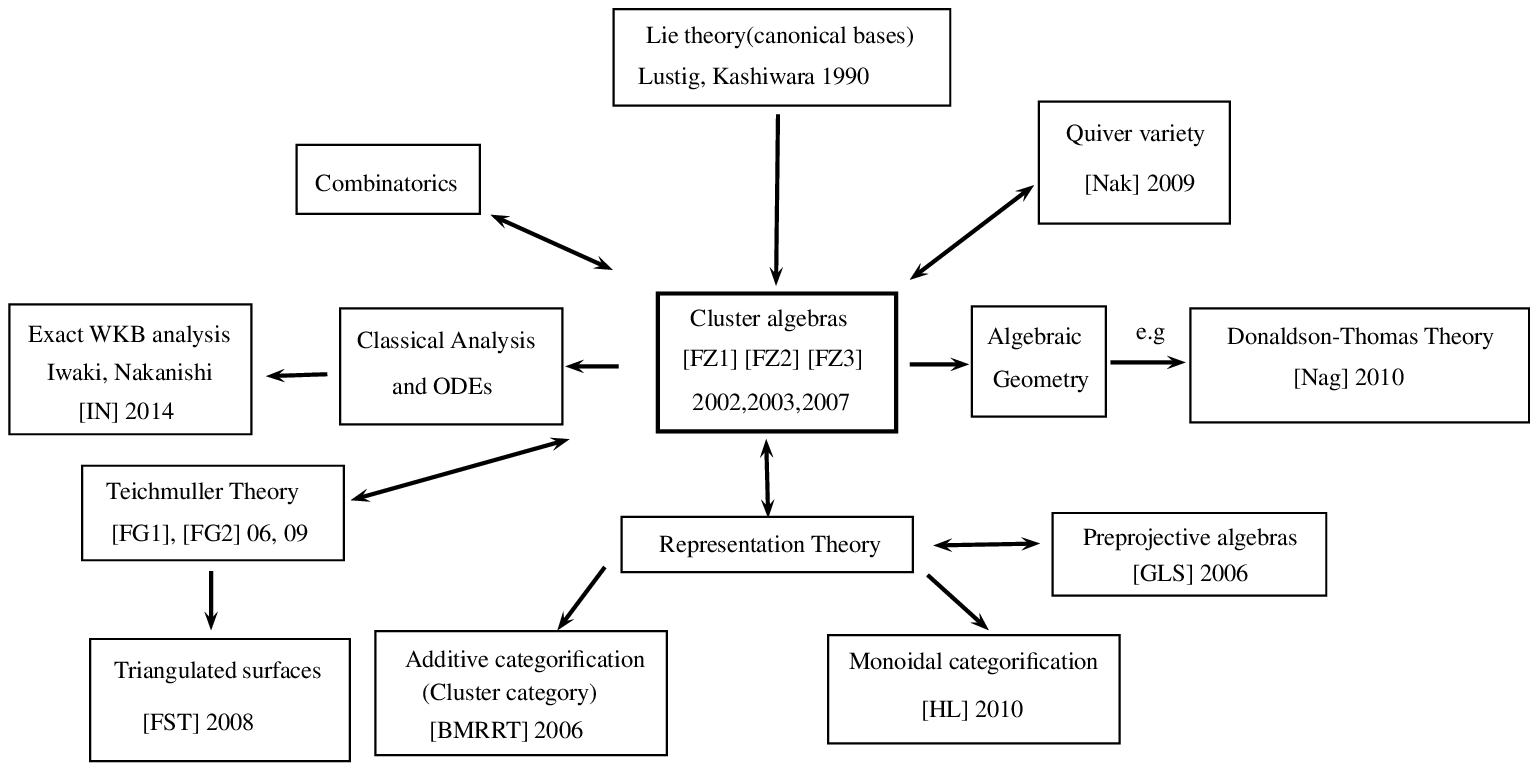}
\caption{relative fields of cluster algebras}
\end{center}
\end{figure}

\begin{table}
\begin{center}
    \begin{tabular}{||c|c|c||}
      \hline
      Cluster algebra & Additive categorification & Monoidal categorification  \\
      \hline
+ & $open$ &	$\oplus$ \\

$\times$ &$\oplus$&$\otimes$	\\

cluster variables &rigid indecomposable objects&	real prime simple objects	\\

clusters &cluster tilting objects&	real simple sets	\\

exchange relations &exchange triangles 	&$0 \rightarrow S \rightarrow M_{i}\otimes M_{i}^{*} \rightarrow S^{'} \rightarrow 0$	\\

dual canonical bases &	$open$ &simple objects	\\

$open$ & $open$	&prime simple objects	\\
\hline
 \end{tabular}
  {Table I}
  \end{center}
\end{table}

In this paper, we focus on the second categorification. Hernandez and Leclerc \cite{HL0} have introduced certain monoidal subcategories $\CC_{\ell}$ of the category $\CC$ of finite dimensional representations of a simply-laced quantum affine algebras, and they conjectured that each $\CC_{\ell}$ is a monoidal categorification of a cluster algebra. In specific, the Grothendieck ring $R_\ell$ of $\CC_\ell$ is isomorphic to a cluster algebra $\A$, and the cluster variables of $\A$ are the classes of all the real prime simple objects of $\CC_\ell$ and the cluster monomials of $\A$ are the classes of all the real simple objects of $\CC_\ell$.
Depending on representation theory, Hernandez and Leclerc  \cite{HL0,HL1} proved this conjecture for $\CC_{1}$ of type $A_{n}$ and $D_{n}$. Almost at the same time, Nakajima \cite{Nak0} using the theory of perverse sheaves and quiver varieties proved the conjecture for $\CC_{1}$ of type $ADE$. And then, Qin \cite{Qin} generalized Nakajima's approach and obtained monoidal categorifications of the cluster algebras associated to any acyclic quiver.

As an application, Hernandez and Leclerc \cite{HL2} described a cluster algebra algorithm for calculating $q$-characters of Kirillov-Reshetikhin representations for any untwisted quantum affine algebra. For simply-laced quantum affine algebra, their formula extends Nakajima's one for $q$-characters of standard representations.

We intend to prove the conjecture \cite{HL0} for $\CC_{2}$ of type $A_{3}$. In the case $\ell=1$, for every simply laced $\g$ the cluster type of Grothendieck ring $R_1$ of $\CC_1$ coincide with the root system of $\g$. But this does not work for $\ell\geq2$. In our case, the cluster type of the ring $R_2$ is of type $E_6$.
In order to write out all cluster monomials, we need to judge the compatibility of any two cluster variables. But the initial quiver $\Gamma_2$ defined in the conjecture is not good enough. So we set a mutation equivalent quiver $\Gamma$ as initial quiver which is an bipartite quiver with underlying graph $E_6$. Two cluster variables are compatible if and only if the corresponding two almost positive roots are compatible, and we give a specific description of compatible subsets of $E_6$ in Appendix~\ref{compatible}. Under the fact that cluster algebras $\A_\Gamma$ and $\A_{\Gamma_2}$ are isomorphic, the conjecture is equivalent to that $\CC_2$ is a monoidal categorification of $\A_\Gamma$. We also denote the map from $\A_\Gamma$ to $R$ by $\iota$. For the image of all cluster variables of $\iota$, we give explicit formulas for the truncated $q$-characters. Then the primary follows. Let $S_1,\  \ldots,\ S_N$ be simple objects of $\CC$, Hernandez \cite{H} proved that the tensor product $S_1\otimes \ldots \otimes S_N$ is simple if and only if $S_i\otimes S_j$ is simple for $i<j$. Thus, we only need to consider any two cluster variables in a cluster.

 If the conjecture is true, we have some conclusions. Since there are $45$ cluster variables in the cluster algebra of type $E_6$ with $3$ frozen variables, the category $\CC_2$ has $45$ real prime simple objects. And the class in $R_2$ of any simple object of $\CC_2$ is a cluster monomial. Besides, for every exchange relation of cluster algebra $\A_\Gamma$, there exists an exact sequence of $\CC_2$ corresponding to it. The set of such exact sequences includes the T-systems.

As shown in Table 1 \cite{HL0}, the cluster algebras $\A_\ell$ defined in the conjecture  \cite{HL0} are of finite type if and only if $\ell=1$ or $n\ell=4,\ 6,\ 8$, where $A_n$ is the type of $\g$. In this paper, we consider the case of $A_3$ for $\ell=2$. The remaining cases can be proved by same way. When cluster algebra $\A_\ell$ is of infinite type, $\A_\ell$ has infinitely many cluster variables and there exists simple objects of $\CC_\ell$ which are not real, our method does not work.

The paper is organized as follows. In Section ~\ref{s2}, we review background on cluster algebras and finite dimensional
representations of quantum affine algebras and give some useful propositions. In Section ~\ref{s3}, we state the main theorem,
Theorem~\ref{thm} firstly. And then we give an equivalent Proposition~\ref{equithm} and prove the equivalence between
Theorem~\ref{thm} and Proposition~\ref{equithm}. In Section ~\ref{s4}, we prove Proposition~\ref{equithm}. We give a polynomial
ring structure on cluster algebra $\A_\Gamma$, the explicit formulas of cluster variables are given in Appendix A.
And then, we give explicit formulas for the truncated $q$-characters of all real prime objects of $\CC_2$. The
compatible subsets of $E_6$ is given in Appendix B.

The explicit formulas of cluster variables and the description of compatible roots are mechanical and lengthy, we put them in Appendices for easily reading.

\section{Definitions and Notations}\label{s2}
In this section, we give a simplified introduction to the theory of cluster algebras \cite{FZ0,FZ1,FZ2} and finite dimensional representations of quantum affine algebras \cite{CP0,CH}.
\subsection{Cluster algebras}\label{s2.1}

Given an $m \times n$ integer matrix $\widetilde{B} = (b_{ij})$,
where the principle part $B$ of  $\widetilde{B}$ is an $n \times n$ antisymmetric square matrix. For $k\in [1,n]$, define
$\mu_k(\widetilde{B}) = (b_{ij}')$ by
$$ b'_{ij} =\left\{\begin{array}{ll}
-b_{ij} & \text{if $i=k$ or $j=k$,} \\[.05in]
b_{ij} + \textup{sgn}(b_{ik}) \ [b_{ik}b_{kj}]_+
 & \text{otherwise,} \end{array}   \right.
$$
where $[x]_+=\textup{max}(x,0)$.
 $\mu_k(\widetilde{B})$ is called the mutation of the matrix $\widetilde{B}$
in direction $k$.

To an $m \times n$ antisymmetric integer matrix there corresponds a quiver $\Gamma$ with vertex set
$\{1,\ldots,m\}$ and with $b_{ij}$ arrows from $i$ to $j$ if $b_{ij}>0$. In particular, $\Gamma$ has no loops nor 2-cycles. Obviously, this corre\-spondence yields a bijection. Under this bijection, given a $\mu_k(\widetilde{B})$ for some $k \in [1,n]$, then there exists an unique quiver corresponding to it, we denote the quiver by $\mu_k(\Gamma)$.

A initial seed is a pair $(\xx,\widetilde{B})$, where  $\xx =(x_1,\ldots,x_m)$ is a free generating set of the field $\F = \Q(x_1,\ldots,x_m)$. For $k \in [1,n]$,  define
\begin{equation}\label{1}
 x'_k x_k=\prod_{b_{ik}> 0} x_i^{b_{ik}} + \prod_{b_{ik}< 0} x_i^{-b_{ik}}.
\end{equation}
The equation~\ref{1} is called a exchange relation.  $\mu_k(\xx)=\{\xx-x_k\}\cup x'_k$. Then the pair
$(\mu_k(\xx),\mu_k(\widetilde{B}))$ is called the mutation of $(\xx,\widetilde{B})$ in direction $k$.
In particular, $\mu_{k}(\mu_k(\xx),\mu_k(\widetilde{B}))=(\xx,\widetilde{B})$.
The $m$-tuple in each seed is called a cluster, and the elements of each cluster are called cluster variables.
Specially, the last $m-n$ elements in each cluster are called frozen variables.

The cluster algebra
$\mathcal{A}(\widetilde{B})$  of rank $n$ is defined as
the subring of $\F$ generated by all the cluster variables of
all seeds.

\begin{remark}
\begin{itemize}
  \item[{\rm (i)}] The initial seed also can be defined by pair $(\xx,\Gamma)$, where $\Gamma$ is the quiver corresponding to $\widetilde{B}$. For $k \in [1,n]$, the mutation of $(\xx,\Gamma)$ in direction $k$ is defined similarly and the corresponding cluster algebra can be written as $\mathcal{A}(\Gamma)$.
  \item[{\rm (ii)}] Two quivers $\Gamma$ and $\Gamma'$ are called mutation equivalent if each of them can be obtained from the other by a sequence of mutations. Denote by $\Gamma\simeq \Gamma'$.
\end{itemize}
\end{remark}

A cluster monomial is a monomial in the cluster variables
which all belong to a same cluster.

Two cluster variables $x$ and $x'$ are said to be compatible if they belong to a same cluster. Otherwise there are some exchange relation such that $xx'=m_1+m_2$ for some cluster monomials $m_1$ and $m_2$.

\begin{theorem}{\cite{FZ0,FZ1}}\label{FZ01}
\begin{itemize}
  \item[{\rm (i)}]  Laurent phenomenon ---  All denominators of all cluster variables are monomials.
In other words, the cluster variables are Laurent polynomials.
  \item[{\rm (ii)}] Classification of cluster algebras --- A cluster algebra is of finite type, $i.e.$ with finitely many different cluster variables if and only if there exists a seed $(\mathbf{x^{'}}, \widetilde{B}^{'})$ such that the quiver attached to the principal part of $\widetilde{B}^{'}$ is a orientation of a simply laced Dynkin diagram $\Gamma$(type $ADE$).

      In this case, $\Gamma$ is unique and there exists a bijection between the set of cluster variables (except for frozen ones) and the set of almost positive roots $\Phi_{\geq-1}$ (the positive roots together with the negatives of the simple roots) of $\Gamma$. If we denote the simple roots by $\alpha_{1},\ldots, \alpha_{n} $, then the bijection is defined by $x_{i} \mapsto -\alpha_{i}, \forall i \in [1,n]$, $\frac{\Box}{\prod x_{i}^{d_{i}}} \mapsto \sum d_{i}\alpha_{i}$ for some $i \in [1,n], d_{i}\in \Z_{\geq0}$.
\end{itemize}
\end{theorem}

\begin{definition}Definition2.1\cite{HL0}\label{defmoncat}
Let $\mathcal{A}$ be a cluster algebra and let $\mathcal{M}$ be an abelian monoidal category.
We say that $\mathcal{M}$ is a monoidal categorification of $\mathcal{A}$
if the Grothendieck ring of $\mathcal{M}$ is isomorphic to $\mathcal{A}$,
and if
\begin{itemize}
 \item[{\rm (i)}]
the cluster monomials of $\mathcal{A}$ are the classes of all the
real simple objects of $\mathcal{M}$;
\item[{\rm (ii)}]
the cluster variables of $\mathcal{A}$ (including the frozen ones)
are the classes of all the real prime simple objects of $\mathcal{M}$.
\end{itemize}
\end{definition}
A simple object $S$ of $\mathcal{M}$ is real if $S\otimes S$ is simple.
A non-trivial simple object $S$ of $\mathcal{M}$ is prime
if there exists no non-trivial factorization $S\cong S_1 \otimes S_2$.

If a cluster algebra has a monoidal categorification, then the famous positivity conjecture is true.

\subsection{Finite dimensional representations of quantum affine algebras}\label{s2.2}

Let $\g$ be a simple Lie algebra of type $ADE$. We denote by $I$ the vertex set of its Dynkin diagram,
by $A=[a_{ij}]_{i,j\in I}$ the Cartan matrix of $\g$, and by $\a_i\ (i\in I)$ the simple roots of $\g$.
Let $U_q(\hg)$ denote the corresponding quantum affine algebra,
with parameter $q\in\C^*$ not a root of unity.

Let $\CC$ be the category of finite-dimensional $U_q(\hg)$-representations
(of type 1). Since $U_q(\hg)$ is a Hopf algebra, $\CC$ is an abelian monoidal category.

\begin{theorem}\cite{CP0,CP2}\label{propertysim}
\begin{itemize}
 \item[{\rm (i)}]
Every simple objects of $\CC$ is a highest weight representation.
 \item[{\rm (ii)}]
Let $V$ be a simple objects of $\CC$. Then, $V$ can be parameterized by $I-$tuples $\pi_{V}=(\pi_{i,V}(u);i\in I)$ of polynomials $\pi_{i,V}(u)\in \mathbb{C}[u]$ with constant term $1$. The $I-$tuples $\pi_{V}$ are called the Drinfeld polynomials of $V$.
 \item[{\rm (iii)}]
 If $V_{1}$, $V_{2}$ and $V$ are simple objects of $\CC$, and $\pi_{i,V}(u)=\pi_{i,V_{1}(u)}\pi_{i,V_{2}(u)} (i\in I)$. Then, $V$ is a subquotient of $V_{1}\otimes V_{2}$.
\end{itemize}
\end{theorem}

For each $i\in I$ and $a \in \C^*$, a fundamental representation $V_{i,a}$ of $U_q(\hg)$ is the simple object with Drinfeld polynomials $ \pi_{j,V_{i,a}}(u) =\left\{\begin{array}{ll}
1-au & \text{if $j=i$ ,} \\[.05in]
1 & \text{if $j\neq i.$} \end{array}   \right.$

For each $i\in I$, $a \in \C^*$ and $k\in \mathbb{N}^{*}$, a Kirillov-Reshetikhin representation $W_{k,a}^{(i)}$ of $U_q(\hg)$ is the simple object with Drinfeld polynomials $ \pi_{j,W_{k,a}^{(i)}}(u) =\left\{\begin{array}{ll}
(1-au)(1-aq^{2}u)\ldots(1-aq^{2k-2}u) & \text{if $j=i$ ,} \\[.05in]
1 & \text{if $j\neq i.$} \end{array}   \right.$

In particular, $W_{1,a}^{(i)}=V_{i,a}$. And $W_{0,a}^{(i)}$ is the trivial representation for every $i$ and $a$.

The $T$-system is the following system of equations:
$$[W^{(i)}_{k,a}][W^{(i)}_{k,aq^2}] = [W^{(i)}_{k+1,a}][W^{(i)}_{k-1,aq^2}]
+ \prod_{j:a_{i,j}=-1} [W^{(j)}_{k,aq}].$$

Let $R$ denote the Grothendieck ring of $\CC$, then it was proved by Frenkel and Reshetikhin that $R$ is a commutative ring which is isomorphic to $\mathbb{Z}\big[[V_{i,a}]\big]_{i\in I,a\in \mathbb{C}^{*}}$.

Since the Dynkin diagram of $\g$ is bipartite, we denote the two parts of $I$ by $I_{0}$ and $I_{1}$.

For $i\in I$, set $ \xi_{i} =\left\{\begin{array}{ll}
0 & \text{if $i\in I_{0}$ ,} \\[.05in]
1 & \text{if $i\in I_{1}.$} \end{array}   \right.$ \quad \quad $\varepsilon_{i}=(-1)^{\xi_{i}}$.

For $\ell\in \Z_{\geq0}$, we define a full subcategory $\CC_{\ell}$ of $\CC$ as follows.
 For any $V$ of $\CC_{\ell}$, the roots of the Drinfeld polynomials of
every composition factor of $V$ belong to $\{q^{-2k-\xi_i}\mid 0\le k \le \ell,\ i\in I \}$.

Denote the Grothendieck ring of $\CC_{\mathbb{\ell}}$ by $R_{\mathbb{\ell}}$, then $R_{\mathbb{\ell}}=\mathbb{Z}\big[[V_{i,q^{2k+\xi_{i}}}]_{i\in I,0\leq k\leq \ell}\big]$.

It is shown  \cite{HL0} that the description of the simple objects of $\CC$ can be reduced to the description of
the simple objects of $\CC_\ell$.

\subsection{$q$--characters\cite{FR,HL0}}\label{s2.3}

To study the finite dimensional representations of $U_q(\hg)$ Frenkel and Reshetikhin \cite{FR} introduced $q$-characters which encodes the (pseudo)-eigenvalues of some commuting elements $\phi_{i,\pm m}^{\pm}(m\geq0)$ of the Cartan subalgebra $U_q(\hat{\mathfrak{h}})\subset U_q(\hg)$:
for $V$ a object of $\CC$, we have
$V=\bigoplus_{\gamma\in\C^{I\times\Z}}V_\gamma$,
where for $\gamma=(\gamma_{i,\pm m}^{(\pm)})_{i\in I,m\geq0}$, $V_\gamma$ is a simultaneous generalized eigenspace ($l$-weight space):
\[
V_{\gamma}=\{x\in V\mid\exists p\in\N,\forall i\in I,\forall m\geq 0,(\phi_{i,\pm m}^{(\pm)}-\gamma_{i,\pm m}^{(\pm)})^p\cdot x=0\}.
\]
The morphism of $q$-characters is an injective ring homomorphism:
\[
\begin{array}{rcl}
\chi_q:R &\rightarrow&\mathcal{Y}=\Z[Y_{i,a}^{\pm}]_{i\in I,a\in\C^*}\\
V&\mapsto & \chi_q(V)=\underset{\gamma\in \C^{I\times \Z}}{\sum}\text{dim}(V_{\gamma})m_{\gamma}
\end{array}
\]
where $m_\gamma\in \mathcal{Y}$ depends of $\gamma$.

For $i\in I, a\in \mathbb{C}^{*}$, define $A_{i,a} = Y_{i,aq}Y_{i,aq^{-1}}\prod_{j\not = i}Y_{j,a}^{a_{ij}}$.
Denote by $\mathcal{M}$ the set of Laurent monomials in $Y_{i,a}$.
Define a partial order on $\M$:
$m \le m' \Longleftrightarrow  \frac{m'}{m}$ is a monomial in the $A_{i,a}$.

A dominant monomial of $\mathcal{M}$ is a  monomial in $Y_{i,a}$ with positive powers only. We will denote the set of dominant monomials by $\mathcal{M}_+$.

Let $V$ be a a simple object of $\CC$ with Drinfeld polynomials:
$ \pi_{i,V}(u) = \prod_{k=1}^{n_i}(1-ua_k^{(i)}),\quad i\in I,$
then the highest weight monomial of $\chi_q(V)$ is
$m_V=\prod_{i\in I}\prod_{k=1}^{n_i}Y_{i,a_k^{(i)}}.$

Conversely, for any dominant monomial $m\in \mathcal{M}_+$,
there exist a unique $I-$tuples of Drinfeld polynomials corre\-sponding to it. We denote the simple object with the highest weight monomial $m$ by $L(m)$.

Frenkel and Mukhin \cite{FM} proved that
\[
 \chi_q(L(m)) = m(1 + \sum_p M_p),
\]
where each $M_p$ is a monomial in variables $A_{i,a}^{-1}$.

A $q$-segment is defined as the string of complex numbers:
$\Si(k,a)=\{a,\ aq^2,\ \cdots, aq^{2k-2}\}.$

We make a little change to the description of the following definition given by Chari and Pressley firstly.
For two $q$-segments $\Si(k_i,a_i)$ and $\Si(k_j,a_j)$, if $\Si(k_i,a_i) \bigcap \Si(k_j,a_j)=\emptyset$, we say that the two $q$-segments are in special position and in general position otherwise; if $\Si(k_i,a_i) \bigcap \Si(k_j,a_j)\in \{\Si(k_i,a_i), \Si(k_j,a_j)\}$, we say the two $q$-segments are in general position.
It is easy to check that every finite multi-set $\{b_{1}, \ldots, b_{s}\}$ of elements of $\mathbb{C}^{*}$ can be written uniquely as a union of segments $\Si(k_i,a_i)$ in such a way that every pair $(\Si(k_i,a_i),\,\Si(k_j,a_j))$ is in general position.

\begin{example}
Let $\g=s\ell_{2}$, then $I=\{1\}$. Drop the index $i\in I$, then
$\mathcal{Y} = \Z[Y_a^\pm ; a\in\C^*]$, $A_a = Y_{aq}Y_{aq^{-1}}$.

Since $W_{1,a}$ can be decomposed as a sum of two 1-dimensional common eigenspaces, $$\chi_q(W_{1,a})=Y_a+Y_{aq^2}^{-1}=Y_a(1+A_{aq}^{-1}).$$
From a equation of $T$--systems: $[W_{1,a}][W_{1,aq^2}] = [W_{2,a}][W_{0,aq^2}]+ 1=[W_{2,a}]+1$, we can get
$$\chi_q(W_{2,a})=Y_aY_{aq^2}+Y_aY_{aq^4}^{-1}+Y_{aq^2}^{-1}Y_{aq^4}^{-1}=Y_aY_{aq^2}(1+A_{aq^{3}}^{-1}(1+A_{aq}^{-1})).
$$
Recursively, for $\forall k\in\N$, set $m_{k,a}=Y_{a}Y_{aq^2}\cdots Y_{aq^{2k-2}}$, then
\[ \chi_q(W_{k,a})=m_{k,a}\left(1+A_{aq^{2k-1}}^{-1}\left(1+A_{aq^{2(k-1)-1}}^{-1}
\left(1+\cdots\left(1+A_{aq^3}^{-1}(1+A_{aq}^{-1})\right)\cdots \right)\right)\right).\]
\end{example}
The authors \cite{CP1} proved that the simple object $V$ of $\CC$ for $\g=s\ell_{2}$ with Drinfeld polynomials: $\pi_{V}(u)=\prod_{m=1}^s(1-ub_m)$
is isomorphic to $\bigotimes_i W_{k_i,a_i}$. Hence we can calculate $q$-characters for all simple objects.

\begin{proposition} Proposition 5.3\cite{HL0}
Let $V$ and $W$ be two objects of $\CC$. If $\chi_q(V)$ and
$\chi_q(W)$ have the same dominant monomials with the same multiplicities,
then $\chi_q(V)=\chi_q(W)$.
\end{proposition}

\subsection{FM--algorithm\cite{FM,HL0}}\label{s2.4}

For any $m\in \mathcal{M}_{+}$, Frenkel and Mukhin define a polynomial $FM(m)$ which equal to $\chi_q(L(m))$ under certain conditions.

Given $i\in I$, we say that $m\in \mathcal{M}$ is $i$-dominant if every
variable $Y_{i,a}$ occurs in $m$ with non-negative
power, and in this case we write $m\in\M_{i,+}$.
\begin{itemize}
 \item[{\rm (i)}] $\forall m=\prod\limits_{k=1}^{n_i}Y_{i,a_k^{(i)}}\prod\limits_{j\in I\backslash\{i\}}\prod\limits_{k=1}^{n_j}Y_{j,a_k^{(j)}}^{b_k^{(j)}}\in \M_{i,+}$, define $\varphi_{i}(m)$ as follows.
  Let $\overline{m}=\prod\limits_{k=1}^{n_i}Y_{a_k^{(i)}}$. Then there is a unique simple object $L(\overline{m})$ of $U_{q}(\hat{s\ell_{2}})$. We can calculate $\chi_q(L(\overline{m}))=\overline{m}(1 + \sum_p \overline{M_{p}})$ exactly.
Then, set $\varphi_i(m)= m(1 + \sum_p M_p)$ where each $M_p$
is obtained from the corresponding $\overline{M_{p}}$ by replacing
each variable $A_a^{-1}$ by $A_{i,a}^{-1}$.
 \item[{\rm (ii)}] Suppose now that $m\in\M_+$, define  $D_m\subseteq\M$ as follows.
  For every $i\in I$, we have $\varphi_i(m)$. Label the monomials occurring in $\varphi_i(m)$
  by $m_0=m,\ m_1,\ \ldots,\ m_r$. For each $m_j\ (1\leq j\leq r)$, there exists some $i\in I$ such that $m_j\in \M_{i,+}$. Thus we have $\varphi_i(m_j)$. Same manner, label the monomials occurring in $\varphi_i(m_j)$ by some $m_k,\ m_{k+1},\ \ldots,\ m_{k+t}$. Repeat this procedure, and denote the set of all such monomials $m_i$ by $D_m$. Obviously, if $m_r\leq m_s$, then $s\leq r$. In particular, for every $m_i\in D_m$, we have $m_i\leq m$.
 \item[{\rm (iii)}] We define the sequences $(s(m_r))_{r\geq 0}$ and $(s_i(m_r))_{r\geq 0}\ (i\in I)$
as follows.
  The initial condition is $s(m_0)=1$ and $s_i(m_0)=0$ for all $i\in I$.
For $t\geq 1$, set\\
$s_i(m_t)=\sum_{\substack{r<t,\\ m_r\in\M_{i,+}}}(s(m_{r})-s_i(m_{r}))[\varphi_i(m_{r}) : m_t], (i\in I), \quad s(m_t) = \max\{s_i(m_t)\mid i\in I\},$
where $[\varphi_i(m_{r}) : m_t]$ denotes the coefficient of $m_t$ in $\varphi_i(m_r)$.
 \item[{\rm (iv)}] We then define
$ FM(m) = \sum_{r\geq 0}s(m_r) m_r.$
\end{itemize}

Let $m\in\M_+$. We say that $L(m)$ is minuscule (special) if $m$ is the only dominant monomial of $\chi_q(L(m))$.
It was proved that \cite{Nak1} all Kirillov-Reshetikhin modules are minuscule. Frenkel and Mukhin proved that
if $L(m)$ is minuscule then $\chi_q(L(m))=FM(m)$.

For any subset $J$ of $I$, we define $\varphi_J(m)$ as a generalization of $\varphi_i(m)$.
A monomial $m\in\M$ is called $J$-dominant if every variable $Y_{j,a}\ (j\in J,\ a\in\C^*)$ occurs in $m$ with non-negative power, and we write $m\in\M_{J,+}$.

For $m=\prod\limits_{j\in J}\prod\limits_{k=1}^{n_j}Y_{j,a_k^{(j)}}\prod\limits_{i\in I-J}\prod\limits_{k=1}^{n_i}Y_{i,a_k^{(i)}}^{b_k^{(i)}}\in \M$, we denote by $\overline{m}=\prod\limits_{j\in J}\prod\limits_{k=1}^{n_j}Y_{j,a_k^{(j)}}$.
If $m\in\M_{J,+}$, then $\overline{m}$ can be regarded as a dominant monomial for the
subalgebra $U_q(\hg_J)$ of $U_q(\hg)$.
Write $\chi_q(L(\overline{m}))=\overline{m}(1 + \sum_p \overline{M}_p)$,
then set $\varphi_J(m) = m(1 + \sum_p M_p)$ where each $M_p$
is obtained from the corresponding $\overline{M}_p$ by replacing
each variable $\overline{A_{j,a}^{-1}}$ by $A_{j,a}^{-1}$.

\begin{proposition}{Proposition 5.10\cite{HL0}}\label{qchadecom}
Let $V$ be an object of $\CC$ and let $J$ be an arbitrary subset of $I$.
Then there is a unique decomposition of $\chi_q(V)$ as a finite sum
\[
 \chi_q(V) = \sum_{\substack{m\in\M_{J,+}\\ \lambda_m\in \Z_{\geq0}}} \lambda_m \varphi_J(m).
\]
\end{proposition}

Let $m\in\M_+$, $mM$ be a monomial
of $\chi_q(L(m))$, where $M$ is a monomial in the
$A_{i,a}^{-1}\ (i\in I)$.
If $M$ contains no variable $A_{j,a}^{-1}$ with $j\in J$
then $mM\in \M_{J,+}$
and $\varphi_J(mM)$ is contained in $\chi_q(L(m))$.

Let us introduce for $P, \widetilde{P} \in \Z[Y_{i,r}^\pm; i\in I, r\in\Z]$ the notation
\[
 P \le \widetilde{P} \quad \Longleftrightarrow \quad \widetilde{P} - P \in \N[Y_{i,r}^\pm; i\in I, r\in\Z].
\]

We define a useful polynomial $N(m)\leq \chi_q(L(m))$ for $m\in \M_+$. We will see later that $N(m)=\chi_q(L(m))$ for some simple objects $L(m)$.
\begin{itemize}
  \item[{\rm (i)}]  $\forall i\in I$, we have $\varphi_i(m)=m(1+\sum M_{i_p})$, where $M_{i_p}$ are monomials in $A_{i,a}^{-1}$, denote all such $mM_{i_p}$ by $m_{i_p}$;
  \item[{\rm (ii)}]  For every $m_{i_p}$, $m_{i_p}$ is $j$-dominant for $j\in I\backslash\{i\}$. We have $\varphi_j(m_{i_p})$, denote the monomials in $\varphi_j(m_{i_p})$ by $m_{i_{p_k}}$;
  \item[{\rm (iii)}]  Repeat this procedure until every monomial contain all $A_{i,a}^{-1}$ with $i\in I$. Put all such distinguished monomials together and denote it by $N(m)$. Clearly, $N(m)\leq \chi_q(L(m))$.
\end{itemize}

\begin{example}
Take $\g$ of type $A_2$ and $m=Y_{1,1}Y_{2,q^3}$.
We have
\[
\begin{array}{lcllcl}
\varphi_1(m) = m(1+ A_{1,q}^{-1}),
&\varphi_2(m) = m(1 + A_{2,q^4}^{-1}),
\end{array}
\]
The monomial $mA_{1,q}^{-1}$ is 2-dominant, $mA_{2,q^4}^{-1}$ is 1-dominant, so
\[
\begin{array}{lcllcl}
\varphi_2(mA_{1,q}^{-1}) = mA_{1,q}^{-1}(1+A_{2,q^4}^{-1}+A_{2,q^4}^{-1}A_{2,q^2}^{-1}),
&\varphi_1(mA_{2,q^4}^{-1}) = mA_{2,q^4}^{-1}(1 + A_{1,q}^{-1})(1+A_{1,q^5}^{-1}),
\end{array}
\]
Then
\[
\begin{array}{lcl}
N(m)= m(1+A_{1,q}^{-1}+A_{2,q^4}^{-1}+A_{1,q}^{-1}A_{2,q^4}^{-1}+A_{1,q}^{-1}A_{2,q^4}^{-1}A_{2,q^2}^{-1}+A_{2,q^4}^{-1}A_{1,q^5}^{-1}
+A_{2,q^4}^{-1}A_{1,q}^{-1}A_{1,q^5}^{-1}).
\end{array}
\]
\end{example}

\subsection{Truncation}\label{s2.5}

To simplify the calculation, we will define some truncations.

By $R_{\mathbb{\ell}}=\mathbb{Z}\big[[V_{i,q^{2k+\xi_{i}}}]_{i\in I,0\leq k\leq \ell}\big]$, we know that the monomials occur in the $q$-characters of objects of $\CC_\ell$ only involve in $Y_{i,q^r}\ (i\in I,\, r\in\Z)$.
So we write $Y_{i,r}$ and $A_{i,r}$ instead
of $Y_{i,q^r}$ and $A_{i,q^r}$ respectively.

Let $V$ be an object of $\CC_2$, $ \chi_q(V) =m(1+\sum_p M_p)$,
where $m$ is a  dominant monomial in the variables $Y_{i,\xi_i}, Y_{i,\xi_i+2}, Y_{i,\xi_i+4}\ (i\in I)$,
and the $M_p$ are certain monomials in the $A_{i,r}^{-1}=Y_{i,r-1}Y_{i,r+1}\prod_{j\neq i}Y_{j,r}^{a_{ij}}$.
Thus it can be checked that for $r\geq5$, there can not exist a monomial $mM_p$ such that $mM_p$  is domi\-nant.

For this reason, we can define the truncated $q$-character of $V$ to be the polynomial obtained from $\chi_q(V)$ by keeping the monomials $M_p$ which do not contain any $A_{i,r}^{-1}$ with $r\ge 5$.
We denote this truncated polynomial by $\chi_q(V)_{\le 4}$.

\begin{proposition}
The map $V \mapsto \chi_q(V)_{\le 4}$ is an injective homomorphism
from the Grothendieck ring $R_2$ of $\CC_2$ to $\mathcal{Y}$.
\begin{proof}
The way to prove this proposition is similar with the method of Proposition 6.1\cite{HL0}, so we omit it.
\end{proof}
\end{proposition}

Similarly, we attach a  polynomial $FM(m)_{\leq 4}$ which is obtained from $FM(m)$ by keeping only the monomials which do not contain any $A_{i,r}^{-1}$ with $r\ge 5$.
In detail,
$\forall m\in \M_{i,+}$, replace $\chi_q(\bar{m})$ by $\chi_q(\bar{m})_{\leq 4}$, $\varphi_i(m)_{\leq 4}$ is the polynomial obtained from $\chi_q(\bar{m})_{\leq 4}$ by replacing
each variable $A_a^{-1}$ by $A_{i,a}^{-1}$. The set $D_m$ and the sequences $(s(m_r))_{r\geq 0}$ and $(s_i(m_r))_{r\geq 0}\ (i\in I)$ are defined as the same. Then, $FM(m)_{\leq 4}$ is well-defined.

\begin{proposition}
 \ If $L(m)$ is minuscule then $\chi_q(L(m))_{\leq4}=FM(m)_{\leq4}$.
\end{proposition}
\begin{proof}
$L(m)$ is minuscule, then $\chi_q(L(m))=FM(m)$. That is to say, $\chi_q(L(m))$ and $FM(m)$ have the same dominant monomials with the same multiplicities. Besides, both $\chi_q(L(m))_{\leq4}$ and $FM(m)_{\leq4}$ contain all dominant monomials. Hence, $\chi_q(L(m))_{\leq4}$ and $FM(m)_{\leq4}$  have the same dominant monomials with the same multiplicities. Then the proposition follows.
\end{proof}

Lastly, we define $N(m)_{\leq 4}$ as the polynomial which is obtained from $N(m)$ by keeping only the monomials which do not contain any $A_{i,r}^{-1}$ with $r\ge 5$.


\section{Main Results}\label{s3}

We prove the conjecture  \cite{HL0} for $\CC_{2}$ of type $A_{3}$. Now, we give some useful notations.





We choose $I_{0}=\{1, 3\}$. The quiver $\Gamma_{2}$ is defined by
\[
\begin{matrix} (1,1)& \leftarrow & (1,2) & \rightarrow & (1,3)\cr
        \downarrow && \uparrow && \cr
 (2,1)& \rightarrow & (2,2)& \leftarrow & (2,3) \cr
         \uparrow && \downarrow && \cr
(3,1)& \leftarrow & (3,2)& \rightarrow & (3,3)
\end{matrix}.
\]
$\A_2 = \A(\Gamma_{2})$ be the cluster algebra attached to the initial seed $(\uu, \Gamma_{2})$, where
$\uu=(x_{(i,k)}\mid i\in I,\ 1\le k \le 3)$, $x_{(i,3)}(i\in I)$ are the frozen variables.


To simplify, we write $W_{k,r}^{(i)}$ instead of $W_{k,q^r}^{(i)}$. Then
\[
R_2 = \Z\left[[W_{1,0}^{(1)}], [W_{1,2}^{(1)}], [W_{1,4}^{(1)}], [W_{1,1}^{(2)}], [W_{1,3}^{(2)}], [W_{1,5}^{(2)}], [W_{1,0}^{(3)}], [W_{1,2}^{(3)}], [W_{1,4}^{(3)}]\right].
\]

\begin{theorem}{($\mathbf{Main \ result}$)}\label{thm}
The map
\[
\begin{array}{ccc}
  x_{(1,1)} \mapsto [W^{(1)}_{1,\,2}], & x_{(1,2)} \mapsto [W^{(1)}_{2,\,2}], &
  x_{(1,3)} \mapsto [W^{(1)}_{3,\,0}], \\
  x_{(2,1)} \mapsto [W^{(2)}_{1,\,3}], & x_{(2,2)} \mapsto [W^{(2)}_{2,\,1}], &
  x_{(2,3)} \mapsto [W^{(2)}_{3,\,1}], \\
  x_{(3,1)} \mapsto [W^{(3)}_{1,\,2}], & x_{(3,2)} \mapsto [W^{(3)}_{2,\,2}], &
  x_{(3,3)} \mapsto [W^{(3)}_{3,\,0}]
\end{array}
\]
extends to a ring isomorphism $\iota$ from the cluster algebra $\A_2$
to the Grothendieck ring $R_2$ of $\CC_2$.
If we identify $\A_{2}$ with $R_2$ via $\iota$, $\CC_2$
becomes a monoidal categorification of $\A_{2}$.
\end{theorem}

In order to prove Theorem~\ref{thm} we need to determine whether two cluster variables are compatible. When a quiver is bipartite of type ADE, two cluster variables are compatible if and only if the corresponding two roots are compatible. So we take a mutation equivalent quiver $\Gamma$ as initial quiver.

\begin{proposition}\label{equithm}
If $\A$ is the cluster algebra of rank $6$ with initial seed $(\xx,\Gamma)$, where $\xx=(x_0, \ldots, x_5, x_{6}, x_{7}, x_{8})$, $x_{6}, x_{7}, x_{8}$ frozen variables, $\Gamma$  defined by Figure2.
\begin{figure}[H]
\begin{center}
\includegraphics {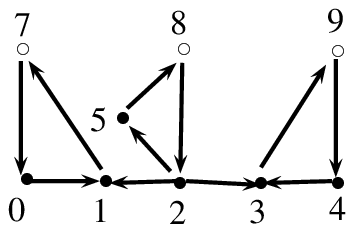}
\caption{mutation equivalent quiver}
\end{center}
\end{figure}
Then the map     $\iota$: $\A \rightarrow R_{2}$,
\[
\begin{array}{lll}
  x_{0} \mapsto [L(Y_{1,0})], & x_{1} \mapsto [L(Y_{1,4}Y_{2,1})], &
  x_{2} \mapsto [L(Y_{1,4}Y_{2,1}Y_{3,4})], \\
  x_{4} \mapsto [L(Y_{3,0})], & x_{3} \mapsto [L(Y_{2,1}Y_{3,4})], & x_{5} \mapsto [L(Y_{2,5})], \\
  x_{6} \mapsto [L(Y_{2,1}Y_{2,3}Y_{2,5})], & x_{7} \mapsto [L(Y_{1,0}Y_{1,2}Y_{1,4})], & x_{8} \mapsto [L(Y_{3,0}Y_{3,2}Y_{3,4})]
\end{array}
\]
can be extended to a ring isomorphism. If we identify $\A$ with $R_2$ via $\iota$, $\CC_2$ becomes a monoidal categorification of $\A$.
\end{proposition}

\begin{remark}
\begin{itemize}
  \item[{\rm (i)}] If $h$ is the isomrphism from $\A$ to $\A_2$, then $\iota\circ h$ is a map from $\A$ to $R_2$. But we still denote the map from $\A$ to $R_2$ by $\iota$ without confusing.
  \item[{\rm (ii)}] In order to distinguish the label of the vertices of Dynkin diagram $E_6$ from $A_3$, we denote by $\widetilde{I}=[0,5]$ the vertex set of Dynkin diagram $E_6$, and chose $\widetilde{I}_0=\{1,3,5\}$.
\end{itemize}

\end{remark}

\begin{proof}
The proof will be given in next section.
\end{proof}

\begin{theorem}
Theorem~\ref{thm} and Proposition~\ref{equithm} are equivalent.
\end{theorem}
\begin{proof}
Theorem~\ref{thm} $\Rightarrow$ Proposition~\ref{equithm}.

Let a finite mutation sequence $\mu=\mu_{(3,2)}\mu_{(1,2)}\mu_{(2,2)}\mu_{(2,1)}\mu_{(3,1)}\mu_{(1,1)}\mu_{(2,1)}$, then $\mu(\Gamma_{2})=\Gamma$ if we ignore the label of vertices. And
\[
\begin{array}{ccc}
  \mu(x_{(1,1)})=x_1, & \mu(x_{(1,2)})=x_0, & \mu(x_{(1,3)})=x_6, \\
  \mu(x_{(2,1)})=x_2, &  \mu(x_{(2,2)})=x_5, &  \mu(x_{(2,3)})=x_7, \\
  \mu(x_{(3,1)})=x_3, & \mu(x_{(3,2)})=x_4, & \mu(x_{(3,3)})=x_8.
\end{array}
\]

Hence, cluster algebras $\A_{2}$ and $\A$ are isomorphism.

In the next section, we give the accurate truncated $q$-characters of all real prime simple objects. Then the result follows.

Proposition~\ref{equithm} $\Rightarrow$ Theorem~\ref{thm}. \ \
The proof is similar.
\end{proof}

\section{The Proof of Main Results}\label{s4}
We prove the proposition by $3$ steps. Firstly, we give a polynomial ring structure on cluster algebra $\A$, moreover, give the ring isomorphism from $\A$ to $R_{2}$. Secondly, we calculate the truncated $q$-characters of all real prime simple objects of $\CC_{2}$. Lastly, we finish the proof of the main result.

Let $Q$ be the root lattice of type $E_{6}$. Let $\Phi\subset Q$ be the root system, the vertices of the Dynkin diagram are labelled as Proposition~\ref{equithm}. Denote by $\Phi_{\ge -1}$ the subset of almost positive roots.

\subsection{Ring isomorphism}\label{s4.1}

For every cluster variable $\frac{\Box}{\prod x_{i}^{d_{i}}}$, there exists a unique almost positive root $\sum d_{i}\alpha_{i}$, where $\alpha_{i} \ (0\leq i\leq 5)$ is the simple root corresponding to  $i \ (0\leq i\leq 5)$. So we denote $x\big[\sum d_{i}\alpha_{i}\big]$ by the variable $\frac{\Box}{\prod x_{i}^{d_{i}}}$.

\begin{theorem}
The cluster algebra $\A$ is equal to the polynomial ring $\Z[a,a',b,c,c',d,d',e,f]$, where $a= x_0, a'= x_4, b= x_5, c=x\big[\alpha_{0}+\alpha_{1}\big], c'=x\big[\alpha_{3}+\alpha_{4}\big],d=x\big[\alpha_{0}+\alpha_{1}+\alpha_{2}+\alpha_{5}\big],  d'= x\big[\alpha_{2}+\alpha_{3}+\alpha_{4}+\alpha_{5}\big], e= x\big[\alpha_{0}+\alpha_{1}+\alpha_{2}+\alpha_{3}+\alpha_{4}+\alpha_{5}\big], f= x\big[\alpha_{2}+\alpha_{5}\big]$.
\end{theorem}
\begin{proof}
The detailed expressions of cluster variables are shown in Appendix A.
Up to the symmetries $0 \leftrightarrow 4$, $1 \leftrightarrow 3$, $7 \leftrightarrow 8$ of $\widetilde{I}$, we reduce to the following 24 variables.

\noindent
$ x_6=bef+dc'+cd'-e-cfc'-bdd'$, \quad $x_7=ac'd+d'-ae-c'f$, \quad  $x\big[-\alpha_{1}\big]=fc'-d'$,\\
$x\big[-\alpha_{2}\big]=e-dc'+cfc'-cd'$, \qquad \qquad $ x\big[\alpha_{0}\big]=dc'-e$, \qquad \qquad $x\big[\alpha_{1}\big]=ac-1$,\\
$x\big[\alpha_{2}\big]=bf-1$, \qquad \qquad \qquad\quad $x\big[\alpha_{5}\big]=ef-dd'$, \qquad \quad $x\big[\alpha_{1}+\alpha_{2}\big]=1-ac+abd-bf$,\\
$x\big[\alpha_{2}+\alpha_{3}+\alpha_{5}\big]=a'd'-f$, \qquad\qquad\qquad \qquad\quad $x\big[\alpha_2+\alpha_3+\alpha_{4}\big]=bd'-c'$,\\
$x\big[\alpha_1+\alpha_2+\alpha_{3}+\alpha_5\big]=f-ad+a'(ae-d')$, \qquad $x\big[\alpha_{0}+\alpha_{1}+\alpha_{2}+\alpha_{3}\big]=c-bd+a'(be-cc')$,\\
$x\big[\alpha_{0}+\alpha_{1}+\alpha_{2}+\alpha_{3}+\alpha_{4}\big]=be-cc'$, \qquad \qquad\quad
$x\big[\alpha_1+\alpha_2+\alpha_{3}+\alpha_4+\alpha_5\big]ae-d'$,\\
$ x\big[\alpha_{1}+\alpha_{2}+\alpha_{3}\big]=-1+a(c-bd)+bf+a'(abe+c'-acc'-bd')$,\\
$x\big[\alpha_{1}+2\alpha_{2}+\alpha_{3}+\alpha_{5}\big]=f[-1+a(c-bd)+bf]+a'[ae+(-ad+f)c'-(ac-abd+bf)d']$,\\
$x\big[\alpha_{0}+2\alpha_{1}+2\alpha_{2}+\alpha_{3}+\alpha_{5}\big]=d[-1+a(c-bd)+bf]+a'[e+abde-bef+c(-ad+f)c'-cd']$,\\
$x\big[\alpha_0+\alpha_1+2\alpha_2+\alpha_{3}+\alpha_5\big]=(c-bd)f+a'(e-dc'-cd'+bdd')$,\\
$x\big[\alpha_0+\alpha_1+2\alpha_2+\alpha_{3}+\alpha_4+\alpha_5\big]=e-dc'+(-c+bd)d'$,\\
$x\big[\alpha_{0}+2\alpha_{1}+2\alpha_{2}+\alpha_{3}+\alpha_{4}+\alpha_{5}\big]=e+abde-bef+c(-ad+f)c'-cd'$,\\
$x\big[\alpha_0+2\alpha_1+2\alpha_2+2\alpha_{3}+\alpha_4+\alpha_5\big]=(cf-d)c'+(c-bd)(ae-d')+ea'(abe+c'-acc'-bd')$,\\
$x\big[\alpha_0+2\alpha_1+3\alpha_2+2\alpha_{3}+\alpha_4+\alpha_5\big]=e-2 b e f+a b c e f-a b^2 d e f+b^2 e f^2+(-c f (-2+a c+b f)+\\
 d(-1+a b c f))c'-cd'+bdd'+a'(be-cc')[a e + (-a d + f)c'-(a c - a b d + b f)d']$,\\
$x\big[\alpha_{0}+2\alpha_{1}+3\alpha_{2}+2\alpha_{3}+\alpha_{4}+2\alpha_{5}\big]=-e (a d - f) (-1 + b f) +[-d + c f + e (1 + a b d - b f)a']d'-\\
ca'd'^2-(ad-f)c'(d-cf+ca'd')$.

As shown, all cluster variables can be represented by some polynomial. And furthermore, it is can be checked  that these $9$ generators are algebraically independent.

Hence, $\A=\Z[a,a',b,c,c',d,d',e,f]$.
\end{proof}

\begin{proposition}\label{ringiso}
The assignment
\[
\begin{array}{ccc}
  x\big[-\alpha_{0}\big] \mapsto [L(Y_{1,0})], & x\big[\alpha_0+\alpha_1+\alpha_2+\alpha_5\big] \mapsto [L(Y_{1,2})], & x\big[\alpha_{3}+\alpha_{4}\big] \mapsto [L(Y_{1,4})],\\
 x\big[-\alpha_{4}\big] \mapsto [L(Y_{3,0})], & x\big[\alpha_2+\alpha_3+\alpha_4+\alpha_5\big] \mapsto [L(Y_{3,2})], & x\big[\alpha_{0}+\alpha_{1}\big] \mapsto [L(Y_{3,4})], \\
  x\big[\alpha_2+\alpha_5\big] \mapsto [L(Y_{2,1})], & x\big[\alpha_{0}+\alpha_{1}+\alpha_{2}+\alpha_{3}+\alpha_{4}+\alpha_{5}\big] \mapsto [L(Y_{2,3})],&x\big[-\alpha_{5}\big] \mapsto [L(Y_{2,5})],
\end{array}
\]
extends to a ring isomorphism $\iota$ from $\A$ to $R_2$.
\end{proposition}
\begin{proof}
Both cluster algebra $\A$ and Grothendieck ring $R_2$ are polynomial rings and an one-to-one correspondence among generators is defined as in the proposition, so they are isomorphic.
\end{proof}

\subsection{Truncated $q$-characters of real prime objects}\label{s4.2}

For each $\beta\in \Phi_{\geq-1}$, we attach a simple object $S(\beta)$ of $\CC_2$.
\[
\begin{array}{rclrclrcl}
  S(-\alpha_{0})&=&L(Y_{1,0}), & S(\alpha_0+\alpha_1+\alpha_2+\alpha_5)&=&L(Y_{1,2}), & S(\alpha_{3}+\alpha_{4})&=&L(Y_{1,4}),\\
 S(-\alpha_{4})&=&L(Y_{3,0}), &S(\alpha_2+\alpha_3+\alpha_4+\alpha_5)&=&L(Y_{3,2}),
&S(\alpha_{0}+\alpha_{1})&=&L(Y_{3,4}),\\
S(\alpha_2+\alpha_5)&=&L(Y_{2,1}),&S(\alpha_{0}+\alpha_{1}+\alpha_{2}+\alpha_{3}+\alpha_{4}+\alpha_{5})&=&L(Y_{2,3}),
&S(-\alpha_{5})&=&L(Y_{2,5}),
\end{array}
\]
\[
\begin{array}{rclrclrcl}
S(-\alpha_1)&=& L(Y_{1,4}Y_{2,1}),&S(-\alpha_2)&=& L(Y_{1,4}Y_{2,1}Y_{3,4}),
&S(-\alpha_3)&=& L(Y_{2,1}Y_{3,4}),\\
S(\alpha_0)&=& L(Y_{1,2}Y_{1,4}),
&S(\alpha_4)&=& L(Y_{3,2}Y_{3,4}),
&S(\alpha_5)&=&L(Y_{2,1}Y_{2,3}),\\
S(\alpha_1)&=&L(Y_{1,0}Y_{3,4}),
&S(\alpha_2)&=&L(Y_{2,1}Y_{2,5}),
&S(\alpha_3)&=&L(Y_{1,4}Y_{3,0}),
\end{array}
\]

For any positive roots $\alpha=\sum_{i\in I}a_i\alpha_i\in \Phi_{>0}-\{\Pi,\ \a_0+\a_1,\ \a_3+\a_4,\ \a_2+\a_5\}$, $\alpha$ can be uniquely written as: 
\begin{itemize}
  \item[{\rm (i)}] if $a_1 = a_2 \geq a_3$, $\alpha=a_3(\alpha_{0}+\alpha_{1}+\alpha_{2}+\alpha_{3}+\alpha_{4}+\alpha_{5})+(a_1-a_3)(\alpha_0+\alpha_1+\alpha_2+\alpha_5)+\sum_{i=0,4,5}b_i(-\alpha_i)$ for some $b_i\in \N$,
  \item[{\rm (ii)}] if $a_1 \leq a_2 = a_3$, $\alpha=a_1(\alpha_{0}+\alpha_{1}+\alpha_{2}+\alpha_{3}+\alpha_{4}+\alpha_{5})+(a_3-a_1)(\alpha_2+\alpha_3+\alpha_4+\alpha_5)+\sum_{i=0,4,5}b_i(-\alpha_i)$ for some $b_i\in \N$,
  \item[{\rm (iii)}] if $a_1 = a_3 < a_2$, $\alpha=\lfloor \frac{a_2}{2}\rfloor(\alpha_0+\alpha_1+\alpha_2+\alpha_5)+\lfloor \frac{a_2}{2}\rfloor(\alpha_2+\alpha_3+\alpha_4+\alpha_5)+(a_1-\lfloor \frac{a_2}{2}\rfloor)(\alpha_{0}+\alpha_{1}+\alpha_{2}+\alpha_{3}+\alpha_{4}+\alpha_{5})+\sum_{i=0,4,5}b_i(-\alpha_i)$ for some $b_i\in \N$, where $\lfloor \frac{a_2}{2}\rfloor$ denotes the biggest integer $\leq \frac{a_2}{2}$.
\end{itemize}

Correspondingly, we define $S(\a)$:
\begin{itemize}
  \item[{\rm (i)}] if $a_1 = a_2 \geq a_3$,
      $S(\a)=Y_{1,0}^{b_0}Y_{1,2}^{a_1-a_3}Y_{2,3}^{a_3}Y_{2,5}^{b_5}Y_{3,0}^{b_4}Y_{3,2}^{0}$,
  \item[{\rm (ii)}] if $a_1 \leq a_2 = a_3$,
   $S(\a)=Y_{1,0}^{b_0}Y_{1,2}^{0}Y_{2,3}^{a_1}Y_{2,5}^{b_5}Y_{3,0}^{b_4}Y_{3,2}^{a_3-a_1}$,
  \item[{\rm (iii)}] if $a_1 = a_3 < a_2$,
   $S(\a)=Y_{1,0}^{b_0}Y_{1,2}^{\lfloor \frac{a_2}{2}\rfloor}Y_{2,3}^{a_1-\lfloor
   \frac{a_2}{2}\rfloor}Y_{2,5}^{b_5}Y_{3,0}^{b_4}Y_{3,2}^{\lfloor \frac{a_2}{2}\rfloor}$.
\end{itemize}

For $i\in I_0=\{1,3\}$, we denote $F_i=L(Y_{i,0}Y_{i,2}Y_{i,4})$; for $i=2$, we denote $F_2=L(Y_{2,1}Y_{2,3}Y_{2,5})$.

\begin{lemma}\label{primeqcha}
For $\forall\beta\in \Phi_{\ge -1}$, we can calculate the truncated $q$-character $\chi_q(S(\beta))_{\leq4}$ exactly.
\end{lemma}

The $18$ Kirillov-Reshetikhin modules of $\CC_2$ are minuscule, their $q$-characters can be obtained by FM-algorithm.

Up to the symmetries $1 \leftrightarrow 3$ of $I$, we list the truncated $q$-characters of the following $12$ Kirillov-Reshetikhin modules.

\noindent
$\chi_q(L(Y_{2,3}))_{\leq4}=Y_{2,3}(1+A_{2,4}^{-1})$,\quad $\chi_q(L(Y_{1,2}))_{\leq4}=Y_{1,2} (1+A_{1,3}^{-1}+A_{1,3}^{-1} A_{2,4}^{-1})$,\\
$\chi_q(L(Y_{1,4}))_{\leq4}=Y_{1,4}$, \quad $\chi_q(L(Y_{1,2}Y_{1,4}))_{\leq4}=Y_{1,2}Y_{1,4}$,\quad $\chi_q(L(Y_{1,0}Y_{1,2}Y_{1,4}))_{\leq4}=Y_{1,0}Y_{1,2}Y_{1,4}$,\\
$\chi_q(L(Y_{2,5}))_{\leq4}=Y_{2,5}$,\quad $\chi_q(L(Y_{2,3}Y_{2,5}))_{\leq4}=Y_{2,3}Y_{2,5}$,\quad $\chi_q(L(Y_{2,1}Y_{2,3}Y_{2,5}))_{\leq4}=Y_{2,1}Y_{2,3}Y_{2,5}$,\\
$\chi_q(L(Y_{2,1}Y_{2,3}))_{\leq4}=Y_{2,1}Y_{2,3}(1+A_{2,4}^{-1}+A_{2,2}^{-1} A_{2,4}^{-1})$,\\
$\chi_q(L(Y_{1,0}))_{\leq4}=Y_{1,0} (1+A_{1,1}^{-1}+A_{1,1}^{-1} A_{2,2}^{-1}+A_{1,1}^{-1} A_{2,2}^{-1} A_{3,3}^{-1})$,\\
$\chi_q(L(Y_{2,1}))_{\leq4}=Y_{2,1}(1+A_{2,2}^{-1}+A_{2,2}^{-1} A_{3,3}^{-1}+A_{2,2}^{-1} A_{1,3}^{-1}+A_{2,2}^{-1} A_{3,3}^{-1}A_{1,3}^{-1}+A_{2,2}^{-1} A_{3,3}^{-1}A_{1,3}^{-1} A_{2,4}^{-1})$,\\
$\chi_q(L(Y_{1,0}Y_{1,2}))_{\leq4}=Y_{1,0}Y_{1,2}(1+A_{1,3}^{-1}+A_{1,1}^{-1} A_{1,3}^{-1}+A_{1,3}^{-1} A_{2,4}^{-1}+A_{1,1}^{-1} A_{1,3}^{-1} A_{2,4}^{-1}+A_{1,1}^{-1} A_{1,3}^{-1} A_{2,4}^{-1} A_{2,2}^{-1})$.

By Theorem~\ref{propertysim} and Proposition~\ref{qchadecom}, we know that the following $8$ objects are also minuscule.
\noindent
$\chi_q(L(Y_{2,5}Y_{1,2}))_{\leq4}=Y_{2,5}Y_{1,2}(1+A_{1,3}^{-1})$,\quad
$\chi_q(L(Y_{1,4}Y_{2,1}))_{\leq4}=Y_{1,4}Y_{2,1}(1+A_{2,2}^{-1}+A_{2,2}^{-1} A_{3,3}^{-1})$,\\
$\chi_q(L(Y_{1,0}Y_{2,3}Y_{2,5}))_{\leq4}=Y_{1,0}Y_{2,3}Y_{2,5}(1+A_{1,1}^{-1})$,\quad
$\chi_q(L(Y_{1,4}Y_{2,1}Y_{3,4}))_{\leq4}=Y_{1,4}Y_{2,1}Y_{3,4}(1+A_{2,2}^{-1})$,\\
$\chi_q(L(Y_{1,0}Y_{3,4}))_{\leq4}=Y_{1,0}Y_{3,4}(1+A_{1,1}^{-1}+A_{2,2}^{-1} A_{1,1}^{-1})$,\\
$\chi_q(L(Y_{1,0}Y_{1,2}Y_{2,5}))_{\leq4}=Y_{1,0}Y_{1,2}Y_{2,5}(1+A_{1,3}^{-1}+A_{1,1}^{-1}A_{1,3}^{-1})$,\\
$\chi_q(L(Y_{2,3}Y_{1,0}))_{\leq4}=Y_{2,3}Y_{1,0}(1+A_{1,1}^{-1}+A_{2,4}^{-1}+A_{2,4}^{-1}A_{1,1}^{-1}+A_{2,4}^{-1}A_{1,1}^{-1} A_{2,2}^{-1})$,\\
$\chi_q(L(Y_{2,1}Y_{2,5}))_{\leq4}=Y_{2,1}Y_{2,5}(1+A_{2,2}^{-1}+A_{2,2}^{-1} A_{3,3}^{-1}+A_{2,2}^{-1} A_{1,3}^{-1}+A_{2,2}^{-1} A_{3,3}^{-1}A_{1,3}^{-1})$.

Though the remaining real prime simple objects are not minuscule, we can calculate them accurately.

Let us introduce for $P, Q \in \Z[Y_{i,r}^\pm; i\in I, r\in\Z]$ the notation
$$P\cap Q =\{ the \ monomials  \ both \ contained \ in \ P \ and \ Q\}.$$

Consider $L(Y_{1,0}Y_{2,3}Y_{3,0})$,

\noindent
$N(Y_{1,0}Y_{2,3}Y_{3,0})_{\leq4}=Y_{1,0} Y_{3,0} Y_{2,3} (1+A_{1,1}^{-1}+A_{2,4}^{-1}+A_{3,1}^{-1}+A_{1,1}^{-1} A_{2,4}^{-1}+A_{1,1}^{-1} A_{2,2}^{-1} A_{2,4}^{-1}+A_{1,1}^{-1} A_{3,1}^{-1}+A_{1,1}^{-1} A_{2,2}^{-1} A_{3,1}^{-1}
+A_{2,4}^{-1} A_{3,1}^{-1}+A_{1,1}^{-1} A_{2,4}^{-1} A_{3,1}^{-1}+A_{2,2}^{-1} A_{2,4}^{-1} A_{3,1}^{-1}+2 A_{1,1}^{-1} A_{2,2}^{-1} A_{2,4}^{-1} A_{3,1}^{-1}+A_{1,1}^{-1} A_{2,2}^{-2} A_{2,4}^{-1} A_{3,1}^{-1})\leq \chi_q(L(Y_{1,0}Y_{2,3}Y_{3,0})).$

By Theorem~\ref{propertysim}, we have\\
$\chi_q(L(Y_{1,0}Y_{2,3}Y_{3,0}))_{\leq4}\leq \chi_q(L(Y_{i,0}Y_{2,3}))_{\leq4}\chi_q(L(Y_{j,0}))_{\leq4}$, where $ i=1,j=3 \ or\  i=3,j=1$.

But $(\chi_q(L(Y_{1,0}Y_{2,3}))_{\leq4}\chi_q(L(Y_{3,0}))_{\leq4}-N(Y_{1,0}Y_{2,3}Y_{3,0})_{\leq4})
\cap
(\chi_q(L(Y_{3,0}Y_{2,3}))_{\leq4}\chi_q(L(Y_{1,0}))_{\leq4}-\\ N(Y_{1,0}Y_{2,3}Y_{3,0})_{\leq4})
 = \emptyset.$

Hence,  $\chi_q(L(Y_{1,0}Y_{2,3}Y_{3,0}))_{\leq4} \ \ =\ \  N(Y_{1,0}Y_{2,3}Y_{3,0})_{\leq4}$.

In same way, we have

$\chi_q(L(Y_{1,2}Y_{3,2}Y_{2,5}))_{\leq4}= N(Y_{1,2}Y_{3,2}Y_{2,5})_{\leq4}=Y_{1,2}Y_{3,2}Y_{2,5}(1+A_{3,3}^{-1}+A_{1,3}^{-1}+A_{1,3}^{-1} A_{3,3}^{-1}+A_{1,3}^{-1} A_{3,3}^{-1} A_{2,4}^{-1})$,

$\chi_q(L(Y_{1,0}Y_{1,2}Y_{2,3}Y_{2,5}))_{\leq4}= N(Y_{1,0}Y_{1,2}Y_{2,3}Y_{2,5})_{\leq4}=Y_{1,0}Y_{1,2}Y_{2,3}Y_{2,5}(1+A_{1,3}^{-1}+A_{1,3}^{-1} A_{1,1}^{-1}+A_{1,3}^{-1} A_{2,4}^{-1}+A_{1,3}^{-1} A_{1,1}^{-1}A_{2,4}^{-1})$,

$\chi_q(L(Y_{1,0}Y_{1,2}Y_{2,3}Y_{2,5}Y_{3,0}))_{\leq4} = N(Y_{1,0}Y_{1,2}Y_{2,3}Y_{2,5}Y_{3,0})_{\leq4}=Y_{1,0}Y_{1,2}Y_{2,3}Y_{2,5}Y_{3,0}(1+A_{1,3}^{-1}+A_{3,1}^{-1}+A_{1,1}^{-1} A_{1,3}^{-1}+
A_{1,3}^{-1} A_{3,1}^{-1}+A_{1,1}^{-1} A_{1,3}^{-1} A_{3,1}^{-1}+A_{1,3}^{-1} A_{2,4}^{-1}+A_{1,3}^{-1} A_{2,4}^{-1} A_{3,1}^{-1}+A_{1,1}^{-1} A_{1,3}^{-1} A_{2,4}^{-1}+A_{1,1}^{-1} A_{1,3}^{-1} A_{2,4}^{-1} A_{3,1}^{-1}+A_{1,1}^{-1} A_{1,3}^{-1} A_{2,2}^{-1} A_{2,4}^{-1} A_{3,1}^{-1})$,

$\chi_q(L(Y_{1,0}Y_{2,3}Y_{2,5}Y_{3,0}))_{\leq4}= N(Y_{1,0}Y_{2,3}Y_{2,5}Y_{3,0})_{\leq4}=Y_{1,0}Y_{2,3}Y_{2,5}Y_{3,0}(1+A_{1,1}^{-1}1+A_{3,1}^{-1}+A_{3,1}^{-1} A_{1,1}^{-1}+A_{3,1}^{-1} A_{2,2}^{-1} A_{1,1}^{-1}) $,

$\chi_q(L(Y_{1,0}Y_{1,2}Y_{2,5}Y_{3,2}))_{\leq4} = N(Y_{1,0}Y_{1,2}Y_{2,5}Y_{3,2})_{\leq4}=Y_{1,0}Y_{1,2}Y_{2,5}Y_{3,2}(1+A_{1,3}^{-1}+A_{3,3}^{-1}+A_{1,1}^{-1} A_{1,3}^{-1}+A_{1,3}^{-1} A_{3,3}^{-1}+
A_{1,1}^{-1} A_{1,3}^{-1} A_{3,3}^{-1}+A_{3,3}^{-1} A_{2,4}^{-1} A_{1,3}^{-1}+A_{1,1}^{-1} A_{3,3}^{-1} A_{2,4}^{-1} A_{1,3}^{-1})$,

$\chi_q(L(Y_{1,0}Y_{1,2}Y_{2,5}Y_{3,0}Y_{3,2}))_{\leq4} = N(Y_{1,0}Y_{1,2}Y_{2,5}Y_{3,0}Y_{3,2})_{\leq4}=Y_{1,0}Y_{1,2}Y_{2,5}Y_{3,0}Y_{3,2}(1+A_{1,3}^{-1}+A_{3,3}^{-1}+A_{1,1}^{-1} A_{1,3}^{-1}
+A_{3,1}^{-1} A_{3,3}^{-1}+A_{1,3}^{-1} A_{3,3}^{-1}+A_{3,1}^{-1} A_{1,3}^{-1} A_{3,3}^{-1}+A_{1,1}^{-1} A_{1,3}^{-1} A_{3,3}^{-1}+A_{1,1}^{-1} A_{1,3}^{-1} A_{3,1}^{-1} A_{3,3}^{-1}+A_{3,3}^{-1} A_{2,4}^{-1} A_{1,3}^{-1}+A_{3,1}^{-1} A_{3,3}^{-1} A_{2,4}^{-1} A_{1,3}^{-1}
+A_{1,1}^{-1} A_{3,3}^{-1} A_{2,4}^{-1} A_{1,3}^{-1}+A_{1,1}^{-1} A_{1,3}^{-1} A_{3,1}^{-1} A_{3,3}^{-1} A_{2,4}^{-1}+A_{1,1}^{-1} A_{1,3}^{-1} A_{3,1}^{-1} A_{3,3}^{-1} A_{2,2}^{-1} A_{2,4}^{-1})$,

$\chi_q(L(Y_{1,0}Y_{1,2}Y_{2,3}Y_{2,5}Y_{3,0}Y_{3,2}))_{\leq4} = N(Y_{1,0}Y_{1,2}Y_{2,3}Y_{2,5}Y_{3,0}Y_{3,2})_{\leq4}=Y_{1,0}Y_{1,2}Y_{2,3}Y_{2,5}Y_{3,0}Y_{3,2}(1+A_{1,3}^{-1}+
A_{3,3}^{-1}
+A_{1,1}^{-1} A_{1,3}^{-1}+A_{3,1}^{-1} A_{3,3}^{-1}+A_{2,4}^{-1} A_{1,3}^{-1}+A_{1,3}^{-1} A_{3,3}^{-1}+A_{3,1}^{-1} A_{1,3}^{-1} A_{3,3}^{-1}+A_{1,1}^{-1} A_{2,4}^{-1} A_{1,3}^{-1}+A_{1,1}^{-1} A_{1,3}^{-1} A_{3,3}^{-1}+A_{1,1}^{-1} A_{1,3}^{-1} A_{3,1}^{-1} A_{3,3}^{-1}
+A_{2,4}^{-1} A_{3,3}^{-1}+A_{3,3}^{-1} A_{2,4}^{-1} A_{3,1}^{-1}+2 A_{3,3}^{-1} A_{2,4}^{-1} A_{1,3}^{-1}+A_{3,3}^{-1} A_{2,4}^{-2} A_{1,3}^{-1}+2 A_{3,1}^{-1} A_{3,3}^{-1} A_{2,4}^{-1} A_{1,3}^{-1}+A_{3,1}^{-1} A_{3,3}^{-1} A_{2,4}^{-2} A_{1,3}^{-1}+
2 A_{1,1}^{-1} A_{3,3}^{-1} A_{2,4}^{-1} A_{1,3}^{-1}+A_{1,1}^{-1} A_{3,3}^{-1} A_{2,4}^{-2} A_{1,3}^{-1}+2 A_{1,1}^{-1} A_{1,3}^{-1} A_{3,1}^{-1} A_{3,3}^{-1} A_{2,4}^{-1}+A_{1,1}^{-1} A_{1,3}^{-1} A_{3,1}^{-1} A_{3,3}^{-1} A_{2,2}^{-1} A_{2,4}^{-1}
+A_{1,1}^{-1} A_{1,3}^{-1} A_{3,1}^{-1} A_{3,3}^{-1} A_{2,4}^{-2}+A_{1,1}^{-1} A_{1,3}^{-1} A_{3,1}^{-1} A_{3,3}^{-1} A_{2,2}^{-1} A_{2,4}^{-2}),$

$\chi_q(L(Y_{1,0} Y_{2,3}^2 Y_{2,5} Y_{3,0}))_{\leq4}= N(Y_{1,0} Y_{2,3}^2 Y_{2,5} Y_{3,0})_{\leq4}=Y_{1,0} Y_{2,3}^2 Y_{2,5} Y_{3,0} (1+A_{3,1}^{-1}+A_{1,1}^{-1}+A_{2,4}^{-1}+A_{1,1}^{-1} A_{2,4}^{-1}
+A_{3,1}^{-1} A_{1,1}^{-1}+A_{3,1}^{-1} A_{2,4}^{-1}+A_{1,1}^{-1} A_{2,4}^{-1} A_{3,1}^{-1}+A_{3,1}^{-1} A_{2,2}^{-1} A_{2,4}^{-1} A_{1,1}^{-1})$,

$\chi_q(L(Y_{1,0} Y_{1,2} Y_{2,3} Y_{2,5}^2 Y_{3,0} Y_{3,2}))_{\leq4}=  N(Y_{1,0} Y_{1,2} Y_{2,3} Y_{2,5}^2 Y_{3,0} Y_{3,2})_{\leq4}=Y_{1,0} Y_{1,2} Y_{2,3} Y_{2,5}^2 Y_{3,0} Y_{3,2} (1+A_{1,3}^{-1}+A_{3,3}^{-1}
+A_{1,1}^{-1} A_{1,3}^{-1}+A_{3,1}^{-1} A_{3,3}^{-1}+A_{1,3}^{-1} A_{3,3}^{-1}+A_{3,1}^{-1} A_{1,3}^{-1} A_{3,3}^{-1}+A_{1,1}^{-1} A_{1,3}^{-1} A_{3,3}^{-1}+A_{1,1}^{-1} A_{1,3}^{-1} A_{3,1}^{-1} A_{3,3}^{-1}+A_{3,3}^{-1} A_{2,4}^{-1} A_{1,3}^{-1}
+A_{3,1}^{-1} A_{3,3}^{-1} A_{2,4}^{-1} A_{1,3}^{-1}+A_{1,1}^{-1} A_{3,3}^{-1} A_{2,4}^{-1} A_{1,3}^{-1}+A_{1,1}^{-1} A_{1,3}^{-1} A_{3,1}^{-1} A_{3,3}^{-1} A_{2,4}^{-1}).$


\begin{theorem}\label{isomor}

We have
\[
   \iota(x\big[\beta\big]) = [S(\beta)],\quad \iota(x_6) = [F_2],\quad \iota(x_7) = [F_1],\quad  \iota(x_8) = [F_3],\qquad (\beta\in \Phi_{\ge -1}).
\]
\end{theorem}
\begin{proof}
The result  follows from lemma~\ref{primeqcha}.
\end{proof}

\subsection{Compatible subsets}

In this subsection, we will give a specific description of compatible subsets of type $E_{6}$.

For any $\alpha,\ \beta\in \Phi_{\geq-1}$, Fomin and Zelevinsky \cite{FZ2} associate a nonnegative integer $(\alpha||\beta)$, known as the compatibility degree.

Let $Q$ denote the root lattice, $Q^\vee$ the root lattice for the dual root system. The vertices of the Dynkin diagram are labelled as Proposition~\ref{equithm}. We have a partition $\widetilde{I}=\widetilde{I}_0 \sqcup \widetilde{I}_1$, and we choose $\widetilde{I}_0=\{1,3,5\}$. For $\alpha \in Q$, we denote by $[\alpha : \alpha_i]$
the coefficient of $\alpha_i$ in the expansion of $\alpha$ in the
basis $\{\a_i\ : i\in \widetilde{I}\}$.

Let $\tau_+$ and $\tau_-$ denote the automorphisms of $Q$ given by
\[
[\tau_\varepsilon  \alpha: \alpha_i] =\left\{\begin{array}{ll}
- [\alpha: \alpha_i] - \sum_{j \neq i}  a_{ij} \max ([\alpha: \alpha_j], 0) & \text{if $\varepsilon (i) = \varepsilon$;} \\[.05in]
[\alpha: \alpha_i]
 & \text{otherwise.} \end{array}   \right.
\]
Both $\tau_+$ and $\tau_-$ are involutions and preserve $\Phi_{\ge -1}$.

An bilinear pairing
\[
\begin{array}{ccc}
Q^\vee \times Q &\to& \Z \\
(\xi,\gamma) &\mapsto & \{\xi, \gamma\}
\end{array}
\]
defined by
$\{\xi, \gamma\} = \sum_{i \in \widetilde{I}} \varepsilon (i) [\xi : \alpha_i^\vee]
[\gamma: \alpha_i] .$

\begin{proposition}Proposition 3.1\cite{FZ2}
The compatibility degree $(\alpha \| \beta)$ is given by
\[
(\alpha \| \beta) = \max(\{\alpha^\vee, \tau_+ \beta\},
\{\tau_+ \alpha^\vee, \beta\}, 0)\,.
\]
\end{proposition}

Two roots $\alpha,\beta\in\Phi_{\geq -1}$ are compatible if $(\alpha \| \beta) = 0$.









For $\a\in \Phi_{\geq-1}$, denote $C_\a$ be the set of almost positive roots which are compatible with $\a$.
We list all such set $C_\a$ in Appendix B.

\subsection{Simple tensor products}

\begin{proposition}Theorem 1.1\cite{H}\label{tensor}
Let $S_1,\cdots, S_N$ be simple objects of $\CC$. The tensor product $S_1\otimes \cdots \otimes S_N$ is simple if and only if $S_i\otimes S_j$ is simple for any $i < j$.
\end{proposition}

\subsection{Proof of Theorem~\ref{monoicat}}

Let $\gamma\in Q$, a cluster expansion of $\gamma$ is a way to express $\gamma$ as
$\gamma=\sum_{\a\in \Phi_{\geq-1}}m_\a\a,$
where all $m_\a\in \Z_{\geq0}$, and  $m_\a m_\beta=0$ whenever $(\a\|\beta)\neq0$.

\begin{proposition}Theorem 3.11\cite{FZ2}\label{FZ3}
Every element of the root lattice has a unique cluster expansion.
\end{proposition}

For any $\a\in \Phi_{\ge -1}$, we have attached a simple object $S(\a)=L(m_\a)$ for some highest weight monomial $m_\a$. Set $Y^{\a}=m_\a$, then for $\gamma = \sum_{\a\in\Phi_{\ge -1}} n_\a \a$, define
$Y^{\gamma}=\prod_{\a\in\Phi_{\ge -1}}(Y^{\a})^{n_\a}.$

Let $V$ be a simple object in $\CC_2$. Its highest weight monomial of the form
$m = \prod_{i\in I} Y_{i,\xi_i}^{a_i} Y_{i,\xi_i+2}^{b_i}Y_{i,\xi_i+4}^{c_i}$
for some $a_i, b_i, c_i\in \Z_{\geq0}$.
Write $m$ as
\[
\begin{array}{ccl}
 m&=&\prod_{i\in I} (Y_{i,\xi_i}Y_{i,\xi_i+2}Y_{i,\xi_i+4})^{\min(a_i,b_i,c_i)}Y^\gamma \\
  &=&\prod_{i\in I} (Y_{i,\xi_i}Y_{i,\xi_i+2}Y_{i,\xi_i+4})^{\min(a_i,b_i,c_i)}
\prod_{\a\in\Phi_{\ge -1}} (Y^\a)^{n_\a}
\end{array}
\]
for some unique cluster expansion $\gamma=\sum_{\a\in\Phi_{\ge -1}} n_\a \a\in Q$, and the factorization of $m$ is unique.

\begin{theorem}\label{monoicat}
$\CC_2$ is a monoidal categorification of $\A$. That is, the cluster monomials of $\A$ are the classes of all the real simple objects of $\M$, and the cluster variables of $\A$ (including the frozen ones)
are the classes of all the real prime simple objects of $\M$.
\end{theorem}

\begin{proof}
In order to prove this theorem, we then need to prove that
\begin{equation}\label{equation}
 L(m) \cong  \bigotimes_{i\in I} F_i^{\otimes\min(a_i,b_i,c_i)}
\bigotimes_{\a\in\Phi_{\ge -1}} S(\a)^{\otimes n_\a}.
\end{equation}

Because of Proposition~\ref{tensor}, 
 the proof of equation ~\ref{equation}  reduces to check that
\begin{itemize}
\item[(i)] $F_i\otimes S(\a)$ and $F_i\otimes F_j$ are simple for each $\a\in\Phi_{\ge -1}$ and $i,j\in I$;
\item[(ii)] if $(\a\|\beta)=0$ ,then $S(\a)\otimes S(\beta)$ is simple;
\item[(iii)] for every $\a\in\Phi_{\geq-1}$, $S(\a)$ is prime;
\item[(iv)] for every $i\in I$, $F_i$ is prime.
\end{itemize}

Consider (iii) and (iv) firstly.

Note that the simple objects $L(Y_{i,\xi_{i}+2k})(i\in I, \ 0\leq k\leq2)$ are clearly prime.

Indeed, for $\a=\a_i(i\in \widetilde{I})$ or $-\a_1$ or $-\a_3$, $S(\a)=L(Y_{i,l}Y_{j,k})$ for some $i,j,k,l$, if $S(\a)$ is not prime we only have $S(\a)\cong L(Y_{i,l})\otimes L(Y_{j,k})$, and $\chi_q(S(\a))<\chi_q(L(Y_{i,l}))\chi_q(L(Y_{j,k}))$.

For $S(-\a_2)=L(Y_{1,4}Y_{2,1}Y_{3,4})$, if $S(-\a_2)$ is not prime we have $3$ possible factorizations, but all factorizations can not be established by calculating their truncated $q$-characters.

Since we have given explicit formulas of $\chi_q(F_i)_{\leq4}(i\in I)$ and $\chi_q(S(\a))_{\leq4}$ for all $\a\in \Phi_{\geq-1}$. We can check that every $S(\a)$ is prime. Thus (iii) and (iv) follows.

There are second method. Recently, the authors \cite{CMY} they explore the relation between self extensions of simple objects of $\CC$ and the property of a simple object being prime. By Theorem 2 \cite{CMY}, we can quickly check that $S(\a)(\a\in\Phi_{\geq-1})$ and $F_i(i\in I)$ are prime.

Now we dispose of (i).

For every $i,j\in I$,
$\chi_q(F_i \otimes F_j)_{\leq4}=\chi_q(F_i)_{\leq4} \chi_q( F_j)_{\leq4}=Y_{i,\xi_i}Y_{i,\xi_i+2}Y_{i,\xi_i+4}Y_{j,\xi_j}Y_{j,\xi_j+2}Y_{j,\xi_j+4},$
hence $F_i \otimes F_j$ is simple.

$N(Y_{i,\xi_i}Y_{i,\xi_i+2}Y_{i,\xi_i+4}Y^{\a})_{\leq4}=Y_{i,\xi_i}Y_{i,\xi_i+2}Y_{i,\xi_i+4}N(Y^{\a})_{\leq4}=
\chi_q(F_i)_{\leq4}\chi_q(L(Y^{\a})_{\leq4},$
hence $F_i \otimes S(\a)$ is simple.

We dispose of (ii) finally.

We have described explicitly the pair of compatible roots of $E_6$ in Appendix B. Up to the symmetries
\smallskip
$0 \leftrightarrow 4$, $1 \leftrightarrow 3$ of $\widetilde{I}$, we are thus reduced to those $C_\a$, where $\a\in C=\{\pm\a_i(i=0,1,2,5),\a_1+\a_2,\a_2+\a_5,\a_3+
\smallskip
\a_4,\a_1+\a_2+\a_3,\a_0+\a_1+\a_2,\a_1+\a_2+\a_5,\a_1+\a_2+\a_3+\a_4,
\a_0+\a_1+\a_2+\a_5,\a_1+\a_2+\a_3+\a_5,\a_1+
\smallskip
2\a_2+\a_3+\a_5,\a_1+\a_2+\a_3+\a_4+\a_5,\a_0+2\a_1+2\a_2+\a_3+\a_5,
\a_0+\a_1+2\a_2+\a_3+\a_5,\a_0+\a_1+\a_2+
\smallskip
\a_3+\a_4,\a_0+2\a_1+2\a_2+\a_3+\a_4+\a_5,\a_0+2\a_1+2\a_2+2\a_3+\a_4+\a_5,
\a_0+\a_1+\a_2+\a_3+\a_4+\a_5,\a_0+\a_1+2\a_2+\a_3+\a_4+\a_5,\a_0+2\a_1+3\a_2+2\a_3+\a_4+\a_5,\a_0+2\a_1+3\a_2+2\a_3+\a_4+2\a_5\}.$

For every $\a\in C,\ \beta\in C_\a$, by Theorem~\ref{propertysim}, we know that $\chi_q(L(Y^{\a}Y^{\beta}))\leq\chi_q(S(\a)\otimes S(\beta))$. In order to prove that $S(\a)\otimes S(\beta)\cong L(Y^{\a}Y^{\beta})$. We only need to check that all the dominant monomials of $\chi_q(S(\a)\otimes S(\beta))$ belong to $\chi_q(L(Y^{\a}Y^{\beta}))$.

For the most of the cases, we can check that each dominant monomial occurs in some $\varphi_J(Y^{\a}Y^{\beta}M)\leq\chi_q(L(Y^{\a}Y^{\beta}))$, where $J\subseteq I$, $M$ is a $J$-dominant monomial. For such cases, we give a brief descriptions.

\medskip
(1) $\a=-\a_0$, $Y^{\a}=Y_{1,0}$. If $\beta\in C_\a\backslash\{-\a_2,-\a_3,\a_1,\a_4,\a_5,\a_1+\a_2,\a_1+\a_2+\a_3,\a_1+\a_2+\a_5,\a_1+\a_2+\a_3+\a_4,\a_1+\a_2+\a_3+\a_5,
\a_1+2\a_2+\a_3+\a_5,\a_1+\a_2+\a_3+\a_4+\a_5,\a_1+2\a_2+\a_3+\a_4+\a_5,\a_1+2\a_2+2\a_3+\a_4+\a_5\}$, then $\chi_q(S(\a)\otimes S(\beta))_{\leq4}$ contains a unique dominant monomial and $S(\a)\bigotimes S(\beta)$ is simple.

If $\beta=-\a_2$, $Y^{\beta}=Y_{1,4}Y_{2,1}Y_{3,4}$, denote $m=Y^{\a}Y^{\beta}$. The only dominant monomials in $\chi_q(S(\a)\otimes S(\beta))_{\leq4}$ are $m$ and $mA_{1,1}^{-1}A_{2,2}^{-1}A_{3,3}^{-1}$.

Since $mA_{1,1}^{-1}=Y_{1,2}^{-1}Y_{1,4}Y_{2,1}^{2}Y_{3,4}$ is $\{2,3\}$-dominant, consider $\varphi_{\{2,3\}}(mA_{1,1}^{-1})$, $\overline{mA_{1,1}^{-1}}=Y_{2,1}(Y_{2,1}Y_{3,4})$, by Proposition~\ref{qchadecom}, we have
$$\varphi_{\{2,3\}}(mA_{1,1}^{-1})_{\leq4}=mA_{1,1}^{-1}(1+A_{2,2}^{-1})(1+A_{2,2}^{-1}+A_{2,2}^{-1}A_{3,3}^{-1})
\leq\chi_q(L(m))_{\leq4}.$$


If $\beta=-\a_3$, $Y^{\beta}=Y_{2,1}Y_{3,4}$, denote $m=Y^{\a}Y^{\beta}$. The only dominant monomials in $\chi_q(S(\a)\otimes S(\beta))_{\leq4}$ are $m$ and $mA_{1,1}^{-1}A_{2,2}^{-1}A_{3,3}^{-1}$. The case is similar to the above.




If $\beta=\a_1$, $Y^{\beta}=Y_{1,0}Y_{3,4}$, denote $m=Y^{\a}Y^{\beta}$. The only dominant monomials in $\chi_q(S(\a)\otimes S(\beta))_{\leq4}$ are $m$ and $mA_{1,1}^{-1}A_{2,2}^{-1}A_{3,3}^{-1}$.

It is can be checked that $N(m)_{\leq4}+mA_{1,1}^{-1}A_{2,2}^{-1}A_{3,3}^{-1}=\chi_q(S(\a)\bigotimes S(\beta))_{\leq4}$. If $mA_{1,1}^{-1}A_{2,2}^{-1}A_{3,3}^{-1}$ is not contained in $\chi_q(L(m))_{\leq4}$, $L(m)$ thus be minuscule and $\chi_q(L(m))=FM(m)$. But, $mA_{1,1}^{-2}A_{2,2}^{-1}A_{3,3}^{-1}\leq N(m)$ is not contained in $FM(m)$.

If $\beta=\a_4$, $Y^{\beta}=Y_{3,2}Y_{3,4}$, denote $m=Y^{\a}Y^{\beta}$. The only dominant monomials in $\chi_q(S(\a)\otimes S(\beta))_{\leq4}$ are $m$ and $mA_{1,1}^{-1}A_{2,2}^{-1}A_{3,3}^{-1}$.
\[
\begin{array}{lll}
  \varphi_1(m)_{\leq4}=m(1+A_{1,1}^{-1}), &
  \varphi_2(mA_{1,1}^{-1})_{\leq4}=mA_{1,1}^{-1}(1+A_{2,2}^{-1}),&
  \varphi_3(mA_{1,1}^{-1}A_{2,2}^{-1})_{\leq4}=mA_{1,1}^{-1}A_{2,2}^{-1}(1+A_{3,3}^{-1}).
\end{array}
\]
Hence $mA_{1,1}^{-1}A_{2,2}^{-1}A_{3,3}^{-1}$ is contained in $\chi_q(L(m))_{\leq4}.$

If $\beta=\a_5$, $Y^{\beta}=Y_{2,1}Y_{2,3}$, denote $m=Y^{\a}Y^{\beta}$. The only dominant monomials in $\chi_q(S(\a)\otimes S(\beta))_{\leq4}$ are $m$ and $mA_{1,1}^{-1}A_{2,2}^{-1}$.
\[
\begin{array}{ll}
  \varphi_1(m)_{\leq4}=m(1+A_{1,1}^{-1}), &  \varphi_2(mA_{1,1}^{-1})_{\leq4}=mA_{1,1}^{-1}(1+A_{2,2}^{-1})(1+A_{2,4}^{-1}+A_{2,4}^{-1}A_{2,2}^{-1}).
\end{array}
\]
Hence $mA_{1,1}^{-1}A_{2,2}^{-1}$ is contained in $\chi_q(L(m))_{\leq4}.$

If $\beta=\a_1+\a_2$, $Y^{\beta}=Y_{1,0}Y_{1,2}Y_{2,5}$, denote $m=Y^{\a}Y^{\beta}$. The only dominant monomials in $\chi_q(S(\a)\otimes S(\beta))_{\leq4}$ are $m$ and $mA_{1,1}^{-1}$.

$mA_{1,1}^{-1}$ is contained in $\varphi_1(m)$, hence occurs in $\chi_q(L(m))_{\leq4}.$

If $\beta=\a_1+\a_2+\a_3$, $Y^{\beta}=Y_{1,0}Y_{2,3}Y_{2,5}Y_{3,0}$, denote $m=Y^{\a}Y^{\beta}$. The only dominant monomials in $\chi_q(S(\a)\otimes S(\beta))_{\leq4}$ are $m$, $mA_{1,1}^{-1}A_{2,2}^{-1}$ and $2mA_{1,1}^{-1}A_{2,2}^{-1}A_{3,1}^{-1}.$

Since $\varphi_{\{1,2\}}(m)_{\leq4}=\varphi_{\{1,2\}}(Y_{1,0})_{\leq4}\varphi_{\{1,2\}}(Y_{1,0}Y_{2,3}Y_{2,5}Y_{3,0})_{\leq4}
=m(1+A_{1,1}^{-1}+A_{1,1}^{-1}A_{2,2}^{-1})(1+A_{1,1}^{-1})$. So, we have $mA_{1,1}^{-1}A_{2,2}^{-1}+2mA_{1,1}^{-1}\leq \chi_{q}(L(m)).$
\[
\begin{array}{cc}
  \varphi_3(2mA_{1,1}^{-1})_{\leq4}=2mA_{1,1}^{-1}(1+A_{3,1}^{-1}), &
  \varphi_2(2mA_{1,1}^{-1}A_{3,1}^{-1})_{\leq4}=2mA_{1,1}^{-1}A_{3,1}^{-1}(1+A_{2,2}^{-1}),
\end{array}
\]
hence $2mA_{1,1}^{-1}A_{2,2}^{-1}A_{3,1}^{-1}$ occurs in $\chi_q(L(m))_{\leq4}.$

If $\beta=\a_1+\a_2+\a_5$, $Y^{\beta}=Y_{1,0}Y_{1,2}$, denote $m=Y^{\a}Y^{\beta}$. The only dominant monomials in $\chi_q(S(\a)\otimes S(\beta))_{\leq4}$ are $m$ and $mA_{1,1}^{-1}$.

$mA_{1,1}^{-1}$ is contained in $\varphi_1(m)$, hence occurs in $\chi_q(L(m))_{\leq4}.$

If $\beta=\a_1+\a_2+\a_3+\a_4$, $Y^{\beta}=Y_{1,0}Y_{2,3}Y_{2,5}$, denote $m=Y^{\a}Y^{\beta}$. The only dominant monomials in $\chi_q(S(\a)\otimes S(\beta))_{\leq4}$ are $m$ and $mA_{1,1}^{-1}A_{2,2}^{-1}$.
\[
mA_{1,1}^{-1}A_{2,2}^{-1}\leq \varphi_{\{1,2\}}(m)_{\leq4}=\varphi_{\{1,2\}}(Y_{1,0})_{\leq4}\varphi_{\{1,2\}}(Y_{1,0}Y_{2,3}Y_{2,5})_{\leq4}
=m(1+A_{1,1}^{-1}+A_{1,1}^{-1}A_{2,2}^{-1})(1+A_{1,1}^{-1}).
\]

If $\beta=\a_1+\a_2+\a_3+\a_5$, $Y^{\beta}=Y_{1,0}Y_{2,3}Y_{3,0}$, denote $m=Y^{\a}Y^{\beta}$. The only dominant monomials in $\chi_q(S(\a)\otimes S(\beta))_{\leq4}$ are $m$, $mA_{1,1}^{-1}A_{2,2}^{-1}$ and $2mA_{1,1}^{-1}A_{2,2}^{-1}A_{3,1}^{-1}$.

Since $\varphi_{\{1,2\}}(m)_{\leq4}
=m(1+A_{1,1}^{-1}+A_{1,1}^{-1}A_{2,2}^{-1})(1+A_{1,1}^{-1}+A_{2,4}^{-1}+A_{2,4}^{-1}A_{1,1}^{-1}+A_{2,4}^{-1}A_{1,1}^{-1} A_{2,2}^{-1})$, we have \\ $mA_{1,1}^{-1}A_{2,2}^{-1}+2mA_{1,1}^{-1}\leq \chi_{q}(L(m)).$
\[
\begin{array}{cc}
  \varphi_3(2mA_{1,1}^{-1})_{\leq4}=2mA_{1,1}^{-1}(1+A_{3,1}^{-1}), &
  \varphi_2(2mA_{1,1}^{-1}A_{3,1}^{-1})_{\leq4}=2mA_{1,1}^{-1}A_{3,1}^{-1}(1+A_{2,2}^{-1})(1+A_{2,4}^{-1}+A_{2,4}^{-1}A_{2,2}^{-1}),
\end{array}
\]
hence $2mA_{1,1}^{-1}A_{2,2}^{-1}A_{3,1}^{-1}$ occurs in $\chi_q(L(m))_{\leq4}.$

If $\beta=\a_1+2\a_2+\a_3+\a_5$, $Y^{\beta}=Y_{1,0}Y_{1,2}Y_{2,5}Y_{3,0}Y_{3,2}$, denote $m=Y^{\a}Y^{\beta}$. The only dominant monomials in $\chi_q(S(\a)\otimes S(\beta))_{\leq4}$ are $m$, $mA_{1,1}^{-1}$, $mA_{1,3}^{-1}A_{3,3}^{-1}A_{2,4}^{-1}$, $mA_{1,1}^{-1}A_{1,3}^{-1}A_{3,3}^{-1}A_{2,2}^{-1}A_{2,4}^{-1}$ and $2mA_{1,1}^{-1}A_{1,3}^{-1}A_{3,1}^{-1}A_{3,3}^{-1}A_{2,2}^{-1}A_{2,4}^{-1}$.
\[
\begin{array}{ll}
  \varphi_1(m)_{\leq4}=m(1+A_{1,1}^{-1})(1+A_{1,3}^{-1}+A_{1,3}^{-1}A_{1,1}^{-1}), &
  \varphi_3(m)_{\leq4}=m(1+A_{3,3}^{-1}+A_{3,3}^{-1}A_{3,1}^{-1}),
\end{array}
\]
\[
\begin{array}{ll}
\medskip
\varphi_3(2mA_{1,3}^{-1}A_{1,1}^{-1})_{\leq4}=2mA_{1,3}^{-1}A_{1,1}^{-1}(1+A_{3,3}^{-1}+A_{3,3}^{-1}A_{3,1}^{-1}),\\
\varphi_2(2mA_{1,3}^{-1}A_{1,1}^{-1}A_{3,3}^{-1}A_{3,1}^{-1})_{\leq4}=2mA_{1,3}^{-1}A_{1,1}^{-1}A_{3,3}^{-1}A_{3,1}^{-1}
  (1+A_{2,4}^{-1}+A_{2,4}^{-1}A_{2,2}^{-1}).
\end{array}
\]

Consider $\varphi_{\{1,2\}}(mA_{3,3}^{-1})_{\leq4}$, $\overline{mA_{3,3}^{-1}}=(Y_{1,0}Y_{1,2})(Y_{1,0}Y_{2,3}Y_{2,5})$, so
\[
\varphi_{\{1,2\}}(mA_{3,3}^{-1})_{\leq4}=mA_{3,3}^{-1}(1+A_{1,3}^{-1}+A_{1,1}^{-1} A_{1,3}^{-1}+A_{1,3}^{-1} A_{2,4}^{-1}+A_{1,1}^{-1} A_{1,3}^{-1} A_{2,4}^{-1}+A_{1,1}^{-1} A_{1,3}^{-1} A_{2,4}^{-1} A_{2,2}^{-1})(1+A_{1,1}^{-1}).
\]

Hence all dominant monomials occur in $\chi_q(L(m))_{\leq4}.$

If $\beta=\a_1+\a_2+\a_3+\a_4+\a_5$, $Y^{\beta}=Y_{1,0}Y_{2,3}$, denote $m=Y^{\a}Y^{\beta}$. The only dominant monomials in $\chi_q(S(\a)\otimes S(\beta))_{\leq4}$ are $m$ and $mA_{1,1}^{-1}A_{2,2}^{-1}$.

Consider $\varphi_{\{1,2\}}(m)$, $\overline{m}=(Y_{1,0})(Y_{1,0}Y_{2,3})$, so
$$\varphi_{\{1,2\}}(m)=m(1+A_{1,1}^{-1}+A_{1,1}^{-1}A_{2,2}^{-1})(1+A_{1,1}^{-1}+A_{2,4}^{-1}+A_{2,4}^{-1}A_{1,1}^{-1}+A_{2,4}^{-1}A_{1,1}^{-1} A_{2,2}^{-1}),$$
hence $mA_{1,1}^{-1}A_{2,2}^{-1}$ occurs in $\chi_q(L(m))_{\leq4}.$

If $\beta=\a_1+2\a_2+\a_3+\a_4+\a_5$, $Y^{\beta}=Y_{1,0}Y_{1,2}Y_{2,5}Y_{3,2}$, denote $m=Y^{\a}Y^{\beta}$. The only dominant monomials in $\chi_q(S(\a)\otimes S(\beta))_{\leq4}$ are $m$, $mA_{1,1}^{-1}$, $mA_{1,3}^{-1}A_{3,3}^{-1}A_{2,4}^{-1}$ and $mA_{1,1}^{-1}A_{1,3}^{-1}A_{3,3}^{-1}A_{2,2}^{-1}A_{2,4}^{-1}$.

$\varphi_1(m)_{\leq4}=m(1+A_{1,1}^{-1})(1+A_{1,3}^{-1}+A_{1,3}^{-1}A_{1,1}^{-1})$, \ $\varphi_3(m)_{\leq4}=m(1+A_{3,3}^{-1})$.

Consider $\varphi_{\{1,2\}}(mA_{3,3}^{-1})$, $\overline{mA_{3,3}^{-1}}=(Y_{1,0}Y_{1,2})(Y_{1,0}Y_{2,3}Y_{2,5})$ which has been calculated above.

Hence $mA_{1,3}^{-1}A_{3,3}^{-1}A_{2,4}^{-1}$ and $mA_{1,1}^{-1}A_{1,3}^{-1}A_{3,3}^{-1}A_{2,2}^{-1}A_{2,4}^{-1}$ occur in $\varphi_{\{1,2\}}(mA_{3,3}^{-1})$.

If $\beta=\a_1+2\a_2+2\a_3+\a_4+\a_5$, $Y^{\beta}=Y_{1,0}Y_{2,3}Y_{2,5}Y_{3,0}Y_{3,2}$, denote $m=Y^{\a}Y^{\beta}$. The only dominant monomials in $\chi_q(S(\a)\otimes S(\beta))_{\leq4}$ are $m$, $mA_{1,1}^{-1}A_{2,2}^{-1}$, $mA_{3,3}^{-1}A_{2,4}^{-1}$, $mA_{1,1}^{-1}A_{2,2}^{-1}A_{2,4}^{-1}A_{3,3}^{-1}$ and $2mA_{1,1}^{-1}A_{3,1}^{-1}A_{3,3}^{-1}A_{2,2}^{-1}A_{2,4}^{-1}$.

It is can be checked that
\[
\begin{array}{ll}
\medskip
  mA_{3,3}^{-1}A_{2,4}^{-1}\leq\varphi_2(mA_{3,3}^{-1}), & mA_{1,1}^{-1}A_{2,2}^{-1}A_{2,4}^{-1}A_{3,3}^{-1}\leq\varphi_{\{1,2\}}(mA_{3,3}^{-1})_{\leq4} , \\
  mA_{1,1}^{-1}A_{2,2}^{-1}\leq\varphi_{\{1,2\}}(m)_{\leq4}, &
2mA_{1,1}^{-1}A_{3,1}^{-1}A_{3,3}^{-1}A_{2,2}^{-1}A_{2,4}^{-1}\leq \varphi_2(2mA_{1,1}^{-1}A_{3,1}^{-1}A_{3,3}^{-1})_{\leq4}.
\end{array}
\]

\medskip
(2) $\a=-\a_1$, $Y^{\a}=Y_{1,4}Y_{2,1}$. If $\beta\in C_\a\backslash\{-\a_2,-\a_3,-\a_4,\a_2,\a_4,\a_5,\a_2+\a_5\}$, then $\chi_q(S(\a)\otimes S(\beta))_{\leq4}$ contains a unique dominant monomial and $S(\a)\otimes S(\beta)$ is simple.

If $\beta=-\a_2$, $Y^{\beta}=Y_{1,4}Y_{2,1}Y_{3,4}$, denote $m=Y^{\a}Y^{\beta}$. The only dominant monomials in $\chi_q(S(\a)\otimes S(\beta))_{\leq4}$ are $m$ and $mA_{2,2}^{-1}A_{3,3}^{-1}.$
\[
mA_{2,2}^{-1}A_{3,3}^{-1}\leq\varphi_{\{2,3\}}(m)_{\leq4}\leq \chi_q(L(m)).
\]

If $\beta=-\a_3$, $Y^{\beta}=Y_{2,1}Y_{3,4}$, denote $m=Y^{\a}Y^{\beta}$. The only dominant monomials in $\chi_q(S(\a)\otimes S(\beta))_{\leq4}$ are $m$, $mA_{2,2}^{-1}A_{3,3}^{-1}$, $mA_{2,2}^{-1}A_{1,3}^{-1}$ and $mA_{1,3}^{-1}A_{2,2}^{-2}A_{3,3}^{-1}$.
\[
\begin{array}{ccc}
  mA_{2,2}^{-1}A_{1,3}^{-1}\leq\varphi_{\{1,2\}}(m)_{\leq4}, &
  mA_{2,2}^{-1}A_{3,3}^{-1}\leq\varphi_{\{2,3\}}(m)_{\leq4}, &
  mA_{1,3}^{-1}A_{2,2}^{-2}A_{3,3}^{-1}\leq\varphi_3(mA_{1,3}^{-1}A_{2,2}^{-2})_{\leq4}.
\end{array}
\]

If $\beta=-\a_4$, $Y^{\beta}=Y_{3,0}$, denote $m=Y^{\a}Y^{\beta}$. The only dominant monomials in $\chi_q(S(\a)\otimes S(\beta))_{\leq4}$ are $m$ and $mA_{1,3}^{-1}A_{2,2}^{-1}A_{3,1}^{-1}$.
\[
mA_{1,3}^{-1}A_{2,2}^{-1}A_{3,1}^{-1}\leq\varphi_{\{1,2\}}(mA_{3,1}^{-1})_{\leq4}\leq \chi_q(L(m)).
\]

If $\beta=\a_2$, $Y^{\beta}=Y_{2,1}Y_{2,5}$, denote $m=Y^{\a}Y^{\beta}$. The only dominant monomials in $\chi_q(S(\a)\otimes S(\beta))_{\leq4}$ are $m$ and $mA_{1,3}^{-1}A_{2,2}^{-1}$.
\[
mA_{1,3}^{-1}A_{2,2}^{-1}\leq\varphi_{\{1,2\}}(m)_{\leq4}\leq \chi_q(L(m)).
\]

If $\beta=\a_4$, $Y^{\beta}=Y_{3,2}Y_{3,4}$, denote $m=Y^{\a}Y^{\beta}$. The only dominant monomials in $\chi_q(S(\a)\otimes S(\beta))_{\leq4}$ are $m$ and $mA_{3,3}^{-1}A_{2,2}^{-1}$.
\[
mA_{3,3}^{-1}A_{2,2}^{-1}\leq\varphi_{3}(mA_{2,2}^{-1})_{\leq4}\leq \chi_q(L(m)).
\]

If $\beta=\a_5$, $Y^{\beta}=Y_{2,1}Y_{2,3}$, denote $m=Y^{\a}Y^{\beta}$. The only dominant monomials in $\chi_q(S(\a)\otimes S(\beta))_{\leq4}$ are $m$ and $mA_{2,2}^{-1}$.
\[
mA_{2,2}^{-1}\leq\varphi_{2}(m)_{\leq4}\leq \chi_q(L(m)).
\]

If $\beta=\a_2+\a_5$, $Y^{\beta}=Y_{2,1}$, denote $m=Y^{\a}Y^{\beta}$. The only dominant monomials in $\chi_q(S(\a)\otimes S(\beta))_{\leq4}$ are $m$ and $mA_{2,2}^{-1}A_{1,3}^{-1}$.
\[
mA_{2,2}^{-1}A_{1,3}^{-1}\leq\varphi_{\{1,2\}}(m)_{\leq4}\leq \chi_q(L(m)).
\]

\medskip
(3) $\a=-\a_2$, $Y^{\a}=Y_{1,4}Y_{2,1}Y_{3,4}$. If $\beta\in C_\a\backslash\{-\a_0,-\a_1,-\a_3,\a_5\}$, then $\chi_q(S(\a)\otimes S(\beta))_{\leq4}$ contains a unique dominant monomial and $S(\a)\otimes S(\beta)$ is simple.

For $\beta=-\a_0,-\a_1$, we have done.

If $\beta=-\a_3$, $Y^{\beta}=Y_{2,1}Y_{3,4}$, denote $m=Y^{\a}Y^{\beta}$. The only dominant monomials in $\chi_q(S(\a)\otimes S(\beta))_{\leq4}$ are $m$ and $mA_{2,2}^{-1}A_{1,3}^{-1}$.

We have $mA_{2,2}^{-1}A_{1,3}^{-1}\leq\varphi_{\{1,2\}}(m)_{\leq4}\leq \chi_q(L(m)).$

If $\beta=\a_5$, $Y^{\beta}=Y_{2,1}Y_{2,3}$, denote $m=Y^{\a}Y^{\beta}$. The only dominant monomials in $\chi_q(S(\a)\otimes S(\beta))_{\leq4}$ are $m$ and $mA_{2,2}^{-1}$.

We have $mA_{2,2}^{-1}\leq\varphi_{2}(m)_{\leq4}\leq \chi_q(L(m)).$

\medskip
(4) $\a=-\a_5$, $Y^{\a}=Y_{2,5}$. If $\beta\in C_\a\backslash\{-\a_0,-\a_1,-\a_2,-\a_3,-\a_4,\a_1+\a_2+\a_3\}$, then $\chi_q(S(\a)\otimes S(\beta))_{\leq4}$ contains a unique dominant monomial and $S(\a)\otimes S(\beta)$ is simple.

For $\beta=-\a_0,-\a_1,-\a_2,-\a_3,-\a_4$, we have done.

If $\beta=\a_1+\a_2+\a_3$, $Y^{\beta}=Y_{1,0}Y_{2,3}Y_{2,5}Y_{3,0}$, denote $m=Y^{\a}Y^{\beta}$. The only dominant monomials in
$\chi_q(S(\a)\otimes S(\beta))_{\leq4}$ are $m$ and $mA_{1,1}^{-1}A_{2,2}^{-1}A_{3,1}^{-1}$.
\[
\begin{array}{ll}
  \varphi_3(mA_{1,1}^{-1})_{\leq4}=mA_{1,1}^{-1}(1+A_{3,1}^{-1}), &
  \varphi_2(mA_{1,1}^{-1}A_{3,1}^{-1})_{\leq4}=mA_{1,1}^{-1}A_{3,1}^{-1}(1+A_{2,2}^{-1})\leq\chi_q(L(m)).
\end{array}
\]

\medskip
(5) $\a=\a_0$, $Y^{\a}=Y_{1,2}Y_{1,4}$. If $\beta\in C_\a\backslash\{-\a_1,-\a_2,-\a_3,-\a_4,-\a_5,\a_2,\a_2+\a_5,\a_0+\a_1+\a_2,\a_0+\a_1+\a_2+\a_5,\a_0+\a_1+2\a_2+\a_3+\a_5,
\a_0+\a_1+2\a_2+2\a_3+\a_4+\a_5,\a_0+\a_1+2\a_2+\a_3+\a_4+\a_5\}$, then $\chi_q(S(\a)\otimes S(\beta))_{\leq4}$ contains a unique dominant monomial and $S(\a)\otimes S(\beta)$ is simple.

For $\beta=-\a_1,-\a_2,-\a_3,-\a_4,-\a_5$, we have done.

If $\beta=\a_2$, $Y^{\beta}=Y_{2,1}Y_{2,5}$, denote $m=Y^{\a}Y^{\beta}$. The only dominant monomials in $\chi_q(S(\a)\otimes S(\beta))_{\leq4}$ are $m$ and $mA_{2,2}^{-1}A_{1,3}^{-1}$.
\[
\begin{array}{cc}
  \varphi_2(m)_{\leq4}=m(1+A_{2,2}^{-1}), &
  \varphi_1(mA_{2,2}^{-1})_{\leq4}=mA_{2,2}^{-1}(1+A_{1,3}^{-1})\leq\chi_q(L(m)).
\end{array}
\]

If $\beta=\a_2+\a_5$, $Y^{\beta}=Y_{2,1}$, denote $m=Y^{\a}Y^{\beta}$. The only dominant monomials in $\chi_q(S(\a)\otimes S(\beta))_{\leq4}$ are $m$ and $mA_{2,2}^{-1}A_{1,3}^{-1}$.
\[
\begin{array}{cc}
  \varphi_2(m)_{\leq4}=m(1+A_{2,2}^{-1}), &
  \varphi_1(mA_{2,2}^{-1})_{\leq4}=mA_{2,2}^{-1}(1+A_{1,3}^{-1})\leq\chi_q(L(m)).
\end{array}
\]

If $\beta=\a_0+\a_1+\a_2$, $Y^{\beta}=Y_{1,2}Y_{2,5}$, denote $m=Y^{\a}Y^{\beta}$. The only dominant monomials in $\chi_q(S(\a)\otimes S(\beta))_{\leq4}$ are $m$ and $mA_{1,3}^{-1}$.
\[
\begin{array}{c}
  \varphi_1(m)_{\leq4}=m(1+A_{1,3}^{-1})\leq\chi_q(L(m)).
\end{array}
\]

If $\beta=\a_0+\a_1+\a_2+\a_5$, $Y^{\beta}=Y_{1,2}$, denote $m=Y^{\a}Y^{\beta}$. The only dominant monomials in $\chi_q(S(\a)\otimes S(\beta))_{\leq4}$ are $m$ and $mA_{1,3}^{-1}$.
\[
\begin{array}{c}
  \varphi_1(m)_{\leq4}=m(1+A_{1,3}^{-1})\leq\chi_q(L(m)).
\end{array}
\]

If $\beta=\a_0+\a_1+2\a_2+\a_3+\a_5$, $Y^{\beta}=Y_{1,2}Y_{2,5}Y_{3,0}Y_{3,2}$, denote $m=Y^{\a}Y^{\beta}$. The only dominant monomials in $\chi_q(S(\a)\otimes S(\beta))_{\leq4}$ are $m$, $mA_{1,3}^{-1}$ and $mA_{1,3}^{-1}A_{2,4}^{-1}A_{3,3}^{-1}$.
\[
\begin{array}{ll}
  \varphi_1(m)_{\leq4}=m(1+A_{1,3}^{-1}),&
  \varphi_2(mA_{1,3}^{-1}A_{3,3}^{-1})_{\leq4}=mA_{1,3}^{-1}A_{3,3}^{-1}(1+A_{2,4}^{-1}).
\end{array}
\]

If $\beta=\a_0+\a_1+2\a_2+2\a_3+\a_4+\a_5$, $Y^{\beta}=Y_{2,3}Y_{2,5}Y_{3,0}Y_{3,2}$, denote $m=Y^{\a}Y^{\beta}$. The only dominant monomials in $\chi_q(S(\a)\otimes S(\beta))_{\leq4}$ are $m$ and $mA_{3,3}^{-1}A_{2,4}^{-1}$.
\[
\begin{array}{cc}
  \varphi_3(m)_{\leq4}=m(1+A_{3,3}^{-1}+A_{3,3}^{-1}A_{3,1}^{-1}),&
  \varphi_2(mA_{3,3}^{-1})_{\leq4}=mA_{3,3}^{-1}(1+A_{2,4}^{-1})\leq\chi_q(L(m)).
\end{array}
\]

If $\beta=\a_0+\a_1+2\a_2+\a_3+\a_4+\a_5$, $Y^{\beta}=Y_{1,2}Y_{2,5}Y_{3,2}$, denote $m=Y^{\a}Y^{\beta}$. The only dominant monomials in $\chi_q(S(\a)\otimes S(\beta))_{\leq4}$ are $m$, $mA_{1,3}^{-1}$ and $mA_{1,3}^{-1}A_{2,4}^{-1}A_{3,3}^{-1}$.
\[
\begin{array}{ccc}
  \varphi_1(m)_{\leq4}=m(1+A_{1,3}^{-1}),&
  \varphi_3(mA_{1,3}^{-1})_{\leq4}=mA_{1,3}^{-1}(1+A_{3,3}^{-1}),&
  \varphi_2(mA_{1,3}^{-1}A_{3,3}^{-1})_{\leq4}=mA_{1,3}^{-1}A_{3,3}^{-1}(1+A_{2,4}^{-1}).
\end{array}
\]

\medskip
(6) $\a=\a_1$, $Y^{\a}=Y_{1,0}Y_{3,4}$. If $\beta\in C_\a\backslash\{-\a_0,-\a_2,-\a_3,-\a_4,-\a_5,\a_5,\a_1+\a_2,\a_1+\a_2+\a_5,\a_1+\a_2+\a_3,\a_1+\a_2+\a_3+\a_4,
\a_1+\a_2+\a_3+\a_5,\a_1+\a_2+\a_3+\a_4+\a_5\}$, then $\chi_q(S(\a)\otimes S(\beta))_{\leq4}$ contains a unique dominant monomial and $S(\a)\otimes S(\beta)$ is simple.

For $\beta=-\a_0,-\a_2,-\a_3,-\a_4,-\a_5$, we have done.

If $\beta=\a_5$, $Y^{\beta}=Y_{2,1}Y_{2,3}$, denote $m=Y^{\a}Y^{\beta}$. The only dominant monomials in $\chi_q(S(\a)\otimes S(\beta))_{\leq4}$ are $m$ and $m A_{1,1}^{-1}A_{2,2}^{-1}$.
\[
\begin{array}{cc}
  \varphi_1(m)_{\leq4}=m(1+A_{1,1}^{-1}), &
  \varphi_2(mA_{1,1}^{-1})_{\leq4}=mA_{1,1}^{-1}(1+A_{2,2}^{-1})\leq\chi_q(L(m)).
\end{array}
\]

If $\beta=\a_1+\a_2$, $Y^{\beta}=Y_{1,0}Y_{1,2}Y_{2,5}$, denote $m=Y^{\a}Y^{\beta}$. The only dominant monomials in $\chi_q(S(\a)\otimes S(\beta))_{\leq4}$ are $m$ and $m A_{1,1}^{-1}$.
\[
\begin{array}{c}
  mA_{1,1}^{-1}\leq\varphi_1(m)_{\leq4}=m(1+A_{1,1}^{-1})(1+A_{1,3}^{-1}+A_{1,3}^{-1}A_{1,1}^{-1})\leq\chi_q(L(m)).
\end{array}
\]

If $\beta=\a_1+\a_2+\a_3$, $Y^{\beta}=Y_{1,0}Y_{2,3}Y_{2,5}Y_{3,0}$, denote $m=Y^{\a}Y^{\beta}$. The only dominant monomials in $\chi_q(S(\a)\otimes S(\beta))_{\leq4}$ are $m$, $m A_{1,1}^{-1}A_{2,2}^{-1}$ and $2m A_{1,1}^{-1}A_{2,2}^{-1}A_{3,1}^{-1}$.

Since $\overline{m}_{\{1,2\}}=(Y_{1,0})(Y_{1,0}Y_{2,3}Y_{2,5}),$
\[
\begin{array}{c}
  \varphi_{\{1,2\}}(m)_{\leq4}=m(1+A_{1,1}^{-1}+A_{1,1}^{-1}A_{2,2}^{-1})(1+A_{1,1}^{-1})\leq\chi_q(L(m)).
\end{array}
\]
Besides, $\varphi_3(2mA_{1,1}^{-1})_{\leq4}=2mA_{1,1}^{-1}(1+A_{3,1}^{-1})$,
 $ \varphi_2(2mA_{3,1}^{-1}A_{1,1}^{-1})_{\leq4}=2mA_{3,1}^{-1}A_{1,1}^{-1}(1+A_{2,2}^{-1})\leq\chi_q(L(m))$.

Hence all dominant monomials occur in $\chi_q(L(m))_{\leq4}.$

If $\beta=\a_1+\a_2+\a_5$, $Y^{\beta}=Y_{1,0}Y_{1,2}$, denote $m=Y^{\a}Y^{\beta}$. The only dominant monomials in $\chi_q(S(\a)\otimes S(\beta))_{\leq4}$ are $m$ and $m A_{1,1}^{-1}$.
\[
\begin{array}{c}
  mA_{1,1}^{-1}\leq\varphi_1(m)_{\leq4}=m(1+A_{1,1}^{-1})(1+A_{1,3}^{-1}+A_{1,3}^{-1}A_{1,1}^{-1})\leq\chi_q(L(m)).
\end{array}
\]

If $\beta=\a_1+\a_2+\a_3+\a_4$, $Y^{\beta}=Y_{1,0}Y_{2,3}Y_{2,5}$, denote $m=Y^{\a}Y^{\beta}$. The only dominant monomials in $\chi_q(S(\a)\otimes S(\beta))_{\leq4}$ are $m$ and $m A_{1,1}^{-1}A_{2,2}^{-1}$.
\[
\begin{array}{c}
  \varphi_{\{1,2\}}(m)_{\leq4}=m(1+A_{1,1}^{-1}+A_{1,1}^{-1}A_{2,2}^{-1})(1+A_{1,1}^{-1})\leq\chi_q(L(m)).
\end{array}
\]

If $\beta=\a_1+\a_2+\a_3+\a_5$, $Y^{\beta}=Y_{1,0}Y_{2,3}Y_{3,0}$, denote $m=Y^{\a}Y^{\beta}$. The only dominant monomials in $\chi_q(S(\a)\otimes S(\beta))_{\leq4}$ are $m$, $m A_{1,1}^{-1}A_{2,2}^{-1}$ and $2m A_{1,1}^{-1}A_{2,2}^{-1}A_{3,1}^{-1}$.
\[
\begin{array}{c}
  \varphi_{\{1,2\}}(m)_{\leq4}=m(1+A_{1,1}^{-1}+A_{1,1}^{-1}A_{2,2}^{-1})(1+A_{1,1}^{-1}+A_{2,4}^{-1}+A_{2,4}^{-1}A_{1,1}^{-1}+A_{2,4}^{-1}A_{1,1}^{-1} A_{2,2}^{-1})\leq\chi_q(L(m)).
\end{array}
\]
Besides, $\varphi_3(2mA_{1,1}^{-1})_{\leq4}=2mA_{1,1}^{-1}(1+A_{3,1}^{-1})$, $\varphi_2(2mA_{3,1}^{-1}A_{1,1}^{-1})_{\leq4}=2mA_{3,1}^{-1}A_{1,1}^{-1}(1+A_{2,2}^{-1})(1+A_{2,4}^{-1}+
 A_{2,4}^{-1}A_{2,2}^{-1})$.

Hence all dominant monomials occur in $\chi_q(L(m))_{\leq4}.$

If $\beta=\a_1+\a_2+\a_3+\a_4+\a_5$, $Y^{\beta}=Y_{1,0}Y_{2,3}$, denote $m=Y^{\a}Y^{\beta}$. The only dominant monomials in $\chi_q(S(\a)\otimes S(\beta))_{\leq4}$ are $m$ and $m A_{1,1}^{-1}A_{2,2}^{-1}$.
\[
\begin{array}{c}
 \varphi_{\{1,2\}}(m)_{\leq4}=m(1+A_{1,1}^{-1}+A_{1,1}^{-1}A_{2,2}^{-1})(1+A_{1,1}^{-1}+A_{2,4}^{-1}+A_{2,4}^{-1}A_{1,1}^{-1}+A_{2,4}^{-1}A_{1,1}^{-1} A_{2,2}^{-1})\leq\chi_q(L(m)).
\end{array}
\]

\medskip
(7) $\a=\a_2$, $Y^{\a}=Y_{2,1}Y_{2,5}$. If $\beta\in C_\a\backslash\{-\a_0,-\a_1,-\a_3,-\a_4,-\a_5,\a_0,\a_4,\a_2+\a_5\}$, then $\chi_q(S(\a)\otimes S(\beta))_{\leq4}$ contains a unique dominant monomial and $S(\a)\otimes S(\beta)$ is simple.

For $\beta=-\a_0,-\a_1,-\a_3,-\a_4,-\a_5,\a_0$, we have done.

If $\beta=\a_4$, $Y^{\beta}=Y_{3,2}Y_{3,4}$, denote $m=Y^{\a}Y^{\beta}$. The only dominant monomials in $\chi_q(S(\a)\otimes S(\beta))_{\leq4}$ are $m$ and $mA_{2,2}^{-1}A_{3,3}^{-1}$.
\[
\begin{array}{cc}
  \varphi_2(m)_{\leq4}=m(1+A_{2,2}^{-1}),&
  \varphi_3(mA_{2,2}^{-1})_{\leq4}=mA_{2,2}^{-1}(1+A_{3,3}^{-2})\leq\chi_q(L(m)).
\end{array}
\]

If $\beta=\a_2+\a_5$, $Y^{\beta}=Y_{2,1}$, denote $m=Y^{\a}Y^{\beta}$. The only dominant monomials in $\chi_q(S(\a)\otimes S(\beta))_{\leq4}$ are $m$ and $mA_{1,3}^{-1}A_{3,3}^{-1}A_{2,2}^{-1}A_{2,4}^{-1}$.

Easily check that $\chi_q(L(Y^{\a})\otimes L(Y^{\beta}))_{\leq4}$ contains only two dominant monomials, so $\chi_q(L(Y^{\a}Y^{\beta}))_{\leq4}$ has two dominant monomials at most. If $mA_{1,3}^{-1}A_{3,3}^{-1}A_{2,2}^{-1}A_{2,4}^{-1}$ is not contained in $\chi_q(L(Y^{\a}Y^{\beta}))_{\leq4}$,
then $L(Y^{\a}Y^{\beta})$ is minuscule, and $\chi_q(L(Y^{\a}Y^{\beta}))=FM(Y^{\a}Y^{\beta})$.

$\chi_q(L(m))_{\leq4}=FM(m)_{\leq4}=m(1+2A_{2,2}^{-1}+A_{2,2}^{-2}+2A_{2,2}^{-1}A_{1,3}^{-1}+2A_{2,2}^{-1}A_{3,3}^{-1}+2A_{2,2}^{-2}A_{1,3}^{-1}
+2A_{2,2}^{-2}A_{3,3}^{-1}+A_{2,2}^{-2}A_{3,3}^{-2}
+A_{2,2}^{-2}A_{1,3}^{-2}+2A_{2,2}^{-1}A_{1,3}^{-1}A_{3,3}^{-1}+4A_{2,2}^{-2}A_{1,3}^{-1}A_{3,3}^{-1}+2A_{2,2}^{-2}A_{1,3}^{-1}A_{3,3}^{-2}+2A_{2,2}^{-2}A_{1,3}^{-2}A_{3,3}^{-1}+
A_{2,2}^{-2}A_{1,3}^{-2}A_{3,3}^{-2}+A_{2,2}^{-2}A_{1,3}^{-2}A_{3,3}^{-2}A_{2,4}^{-1}).$

By Proposition~\ref{qchadecom}, we fix $i=1$, then there is a
unique decomposition of $\chi_q(L(m))$ as a finite sum
$ \chi_q(L(m)) = \sum_{\substack{m\in\M_{1,+}\\ \lambda_m\geq0}} \lambda_m \varphi_1(m).$
In other word, any monomial in $\chi_q(L(m))$ must be contained in some $\varphi_1(m')$, where $m'\in \M_{1,+}$ and $m'\leq\chi_q(L(m)$. But, the monomial $A_{2,2}^{-2}A_{1,3}^{-2}A_{3,3}^{-2}A_{2,4}^{-1}$ does not satisfy this property. Contradict.

\medskip
(8) $\a=\a_5$, $Y^{\a}=Y_{2,1}Y_{2,3}$. If $\beta\in C_\a\backslash\{-\a_0,-\a_1,-\a_2,-\a_3,-\a_4,\a_1,\a_3,\a_2+\a_5,\a_1+\a_2+\a_3+\a_5\}$, then $\chi_q(S(\a)\otimes S(\beta))_{\leq4}$ contains a unique dominant monomial and $S(\a)\otimes S(\beta)$ is simple.

For $\beta=-\a_0,-\a_1,-\a_2,-\a_3,-\a_4,\a_1,\a_3$, we have done.

If $\beta=\a_2+\a_5$, $Y^{\beta}=Y_{2,1}$, denote $m=Y^{\a}Y^{\beta}$. The only dominant monomials in $\chi_q(S(\a)\otimes S(\beta))_{\leq4}$ are $m$ and $m A_{2,2}^{-1}$.
\[
\begin{array}{c}
  mA_{2,2}^{-1}\leq\varphi_2(m)_{\leq4}=m(1+A_{2,2}^{-1})(1+A_{2,4}^{-1}+A_{2,4}^{-1}A_{2,2}^{-1})\leq\chi_q(L(m)).
\end{array}
\]

If $\beta=\a_1+\a_2+\a_3+\a_5$, $Y^{\beta}=Y_{1,0}Y_{2,3}Y_{3,0}$, denote $m=Y^{\a}Y^{\beta}$. The only dominant monomials in $\chi_q(S(\a)\otimes S(\beta))_{\leq4}$ are $m$ and $mA_{1,1}^{-1}A_{2,2}^{-1}A_{3,1}^{-1}$.
\[
\varphi_2(mA_{3,1}^{-1}A_{1,1}^{-1})_{\leq4}=mA_{3,1}^{-1}A_{1,1}^{-1}(1+A_{2,2}^{-1})(1+A_{2,4}^{-1}
+A_{2,4}^{-1}A_{2,2}^{-1})^{2}\leq\chi_q(L(m)).
\]

\medskip
(9) $\a=\a_1+\a_2$, $Y^{\a}=Y_{1,0}Y_{1,2}Y_{2,5}$. If $\beta\in C_\a\backslash\{-\a_0,-\a_3,-\a_4,-\a_5,\a_1,\a_2+\a_5,\a_1+\a_2+\a_3,\a_1+\a_2+\a_5,\a_1+\a_2+\a_3+\a_4,\a_0+\a_1+\a_2+\a_5,
\a_1+2\a_2+\a_3+\a_5,\a_0+\a_1+2\a_2+\a_3+\a_5,\a_1+2\a_2+\a_3+\a_4+\a_5,\a_0+2\a_1+2\a_2+\a_3+\a_4+\a_5,
\a_0+\a_1+2\a_2+\a_3+\a_4+\a_5,\a_0+2\a_1+3\a_2+2\a_3+\a_4+\a_5\}$, then $\chi_q(S(\a)\otimes S(\beta))_{\leq4}$ contains a unique dominant monomial and $S(\a)\otimes S(\beta)$ is simple.

For $\beta=-\a_0,-\a_3,-\a_4,-\a_5,\a_1$, we have done.

If $\beta=\a_2+\a_5$, $Y^{\beta}=Y_{2,1}$, denote $m=Y^{\a}Y^{\beta}$. The only dominant monomials in $\chi_q(S(\a)\otimes S(\beta))_{\leq4}$ are $m$ and $m A_{1,3}^{-1}A_{3,3}^{-1}A_{2,2}^{-1}A_{2,4}^{-1}$.

As before, $\chi_q(L(Y^{\a})\otimes L(Y^{\beta}))_{\leq4}$ contains only two dominant monomials, so $\chi_q(L(Y^{\a}Y^{\beta}))_{\leq4}$ has two dominant monomials at most. If $mA_{1,3}^{-1}A_{3,3}^{-1}A_{2,2}^{-1}A_{2,4}^{-1}$ is not contained in $\chi_q(L(Y^{\a}Y^{\beta}))_{\leq4}$,
then $L(Y^{\a}Y^{\beta})$ is minuscule, and $\chi_q(L(Y^{\a}Y^{\beta}))=FM(Y^{\a}Y^{\beta})$.

$FM(m)_{\leq4}=m(1+A_{1,3}^{-1}+A_{1,1}^{-1}A_{1,3}^{-1}+A_{2,2}^{-1}+A_{1,1}^{-1}A_{1,3}^{-1}A_{2,2}^{-1}
+2A_{2,2}^{-1}A_{1,3}^{-1}+A_{2,2}^{-1}A_{1,3}^{-2}+A_{2,2}^{-1}A_{1,3}^{-2}A_{1,1}^{-1}+A_{2,2}^{-1}A_{3,3}^{-1}+
A_{1,1}^{-1}A_{1,3}^{-1}A_{2,2}^{-1}A_{3,3}^{-1}+2A_{1,3}^{-1}A_{2,2}^{-1}A_{3,3}^{-1}+
A_{1,3}^{-2}A_{2,2}^{-1}A_{3,3}^{-1}+A_{1,1}^{-1}A_{1,3}^{-2}A_{2,2}^{-1}A_{3,3}^{-1}+A_{1,3}^{-2}A_{2,2}^{-1}A_{2,4}^{-1}A_{3,3}^{-1}
+A_{1,1}^{-1}A_{1,3}^{-2}A_{2,2}^{-1}A_{2,4}^{-1}A_{3,3}^{-1}).$

By Proposition~\ref{qchadecom}, we fix $i=1$, then there is a
unique decomposition of $\chi_q(L(m))$ as a finite sum
$ \chi_q(L(m)) =  \sum_{\substack{m\in\M_{1,+}\\ \lambda_m\geq0}} \lambda_m \varphi_1(m).$
In other word, any monomial in $\chi_q(L(m))$ must be contained in some $\varphi_1(m')$, where $m'\in \M_{1,+}$ and $m'\leq\chi_q(L(m)$. But, the monomial $A_{1,3}^{-2}A_{2,2}^{-1}A_{2,4}^{-1}A_{3,3}^{-1}$ does not satisfy this property.

If $\beta=\a_1+\a_2+\a_3$, $Y^{\beta}=Y_{1,0}Y_{2,3}Y_{2,5}Y_{3,0}$, denote $m=Y^{\a}Y^{\beta}$. The only dominant monomials in $\chi_q(S(\a)\otimes S(\beta))_{\leq4}$ are $m$, $mA_{1,1}^{-1}$ and $m A_{1,1}^{-1}A_{3,1}^{-1}A_{2,2}^{-1}$.
\[
\begin{array}{cc}
\varphi_1(m)_{\leq4}=m(1+A_{1,1}^{-1})(1+A_{1,3}^{-1}+A_{1,3}^{-1}A_{1,1}^{-1}),&
\varphi_2(mA_{1,1}^{-1}A_{3,1}^{-1})_{\leq4}=mA_{1,1}^{-1}A_{3,1}^{-1}(1+A_{2,2}^{-1}).
\end{array}
\]


If $\beta=\a_1+\a_2+\a_5$, $Y^{\beta}=Y_{1,0}Y_{1,2}$, denote $m=Y^{\a}Y^{\beta}$. The only dominant monomials in $\chi_q(S(\a)\otimes S(\beta))_{\leq4}$ are $m$, $mA_{1,3}^{-1}A_{2,4}^{-1}$ and $mA_{1,1}^{-1}A_{1,3}^{-1}A_{2,4}^{-1}$.
\[
\begin{array}{c}
 \varphi_{\{1,2\}}(m)_{\leq4}=m(1+A_{1,3}^{-1}+A_{1,1}^{-1} A_{1,3}^{-1}+A_{1,3}^{-1} A_{2,4}^{-1}+A_{1,1}^{-1} A_{1,3}^{-1} A_{2,4}^{-1}+A_{1,1}^{-1} A_{1,3}^{-1} A_{2,4}^{-1} A_{2,2}^{-1})(1+A_{1,3}^{-1}+A_{1,1}^{-1}A_{1,3}^{-1}).
\end{array}
\]

If $\beta=\a_1+\a_2+\a_3+\a_4$, $Y^{\beta}=Y_{1,0}Y_{2,3}Y_{2,5}$, denote $m=Y^{\a}Y^{\beta}$. The only dominant monomials in $\chi_q(S(\a)\otimes S(\beta))_{\leq4}$ are $m$ and $m A_{1,1}^{-1}$.
\[
\begin{array}{c}
  mA_{1,1}^{-1}\leq\varphi_1(m)_{\leq4}=m(1+A_{1,1}^{-1})(1+A_{1,3}^{-1}+A_{1,3}^{-1}A_{1,1}^{-1})\leq\chi_q(L(m)).
\end{array}
\]

If $\beta=\a_0+\a_1+\a_2+\a_5$, $Y^{\beta}=Y_{1,2}$, denote $m=Y^{\a}Y^{\beta}$. The only dominant monomials in $\chi_q(S(\a)\otimes S(\beta))_{\leq4}$ are $m$ and $m A_{1,3}^{-1}A_{2,4}^{-1}$.
\[
\begin{array}{c}
 \varphi_{\{1,2\}}(m)_{\leq4}=m(1+A_{1,3}^{-1}+A_{1,3}^{-1}A_{2,4}^{-1})(1+A_{1,3}^{-1}+A_{1,1}^{-1}A_{1,3}^{-1}).
\end{array}
\]

If $\beta=\a_1+2\a_2+\a_3+\a_5$, $Y^{\beta}=Y_{1,0}Y_{1,2}Y_{3,0}Y_{3,2}Y_{2,5}$, denote $m=Y^{\a}Y^{\beta}$. The only dominant monomials in $\chi_q(S(\a)\otimes S(\beta))_{\leq4}$ are $m$, $mA_{1,3}^{-1}A_{3,3}^{-1}A_{2,4}^{-1}$, $mA_{1,1}^{-1}A_{1,3}^{-1}A_{3,3}^{-1}A_{2,4}^{-1}$ and $mA_{1,1}^{-1}A_{1,3}^{-1}A_{3,1}^{-1}A_{3,3}^{-1}A_{2,2}^{-1}A_{2,4}^{-1}.$

$mA_{1,3}^{-1}A_{3,3}^{-1}A_{2,4}^{-1}+mA_{1,1}^{-1}A_{1,3}^{-1}A_{3,3}^{-1}A_{2,4}^{-1}\leq\varphi_{\{1,2\}}(mA_{3,3}^{-1})$,

$mA_{1,1}^{-1}A_{1,3}^{-1}A_{3,1}^{-1}A_{3,3}^{-1}A_{2,2}^{-1}A_{2,4}^{-1}\leq\varphi_{\{1,2\}}(mA_{3,1}^{-1}A_{3,3}^{-1})$.


If $\beta=\a_0+\a_1+2\a_2+\a_3+\a_5$, $Y^{\beta}=Y_{1,2}Y_{2,5}Y_{3,0}Y_{3,2}$, denote $m=Y^{\a}Y^{\beta}$. The only dominant monomials in $\chi_q(S(\a)\otimes S(\beta))_{\leq4}$ are $m$ and $mA_{1,3}^{-1}A_{3,3}^{-1}A_{2,4}^{-1}$.
$$mA_{1,3}^{-1}A_{3,3}^{-1}A_{2,4}^{-1}\leq \varphi_{\{1,2\}}(mA_{3,3}^{-1}).$$
If $\beta=\a_0+2\a_1+2\a_2+\a_3+\a_5$, $Y^{\beta}=Y_{1,0}Y_{1,2}Y_{2,3}Y_{2,5}Y_{3,0}$, denote $m=Y^{\a}Y^{\beta}$. The only dominant monomials in $\chi_q(S(\a)\otimes S(\beta))_{\leq4}$ are $m$, $mA_{1,3}^{-1}A_{2,4}^{-1}$, $m A_{1,1}^{-1}A_{1,3}^{-1}A_{2,4}^{-1}$ and $mA_{1,1}^{-1}A_{1,3}^{-1}A_{2,2}^{-1}A_{2,4}^{-1}A_{3,1}^{-1}$.
\[
\begin{array}{ll}
mA_{1,3}^{-1}A_{2,4}^{-1}+m A_{1,1}^{-1}A_{1,3}^{-1}A_{2,4}^{-1}\leq\varphi_{\{1,2\}}(m)_{\leq4},&
mA_{1,1}^{-1}A_{1,3}^{-1}A_{2,2}^{-1}A_{2,4}^{-1}A_{3,1}^{-1}\leq\varphi_{\{1,2\}}(mA_{3,1}^{-1})_{\leq4}.
\end{array}
\]

If $\beta=\a_1+2\a_2+\a_3+\a_4+\a_5$, $Y^{\beta}=Y_{1,0}Y_{1,2}Y_{2,5}Y_{3,2}$, denote $m=Y^{\a}Y^{\beta}$. The only dominant monomials in $\chi_q(S(\a)\otimes S(\beta))_{\leq4}$ are $m$, $mA_{1,3}^{-1}A_{3,3}^{-1}A_{2,4}^{-1}$ and $mA_{1,1}^{-1}A_{1,3}^{-1}A_{3,3}^{-1}A_{2,4}^{-1}$.
\[
mA_{1,3}^{-1}A_{3,3}^{-1}A_{2,4}^{-1}+mA_{1,1}^{-1}A_{1,3}^{-1}A_{3,3}^{-1}A_{2,4}^{-1}\leq \varphi_{\{1,2\}}(mA_{3,3}^{-1})\leq\chi_q(L(m)).
\]

If $\beta=\a_0+2\a_1+2\a_2+\a_3+\a_4+\a_5$, $Y^{\beta}=Y_{1,0}Y_{1,2}Y_{2,3}Y_{2,5}$, denote $m=Y^{\a}Y^{\beta}$. The only dominant monomials in $\chi_q(S(\a)\otimes S(\beta))_{\leq4}$ are $m$, $mA_{1,3}^{-1}A_{2,4}^{-1}$ and $mA_{1,1}^{-1}A_{1,3}^{-1}A_{2,4}^{-1}$.
\[
mA_{1,3}^{-1}A_{2,4}^{-1}+mA_{1,1}^{-1}A_{1,3}^{-1}A_{2,4}^{-1}\leq \varphi_{\{1,2\}}(m)\leq\chi_q(L(m)).
\]

If $\beta=\a_0+\a_1+2\a_2+\a_3+\a_4+\a_5$, $Y^{\beta}=Y_{1,2}Y_{2,5}Y_{3,2}$, denote $m=Y^{\a}Y^{\beta}$. The only dominant monomials in $\chi_q(S(\a)\otimes S(\beta))_{\leq4}$ are $m$ and $mA_{1,3}^{-1}A_{3,3}^{-1}A_{2,4}^{-1}$.
\[
mA_{1,3}^{-1}A_{3,3}^{-1}A_{2,4}^{-1}\leq \varphi_{\{1,2\}}(mA_{3,3}^{-1})\leq\chi_q(L(m)).
\]

If $\beta=\a_0+2\a_1+3\a_2+2\a_3+\a_4+\a_5$, $Y^{\beta}=Y_{1,0}Y_{1,2}Y_{2,3}Y_{2,5}^{2}Y_{3,0}Y_{3,2}$, denote $m=Y^{\a}Y^{\beta}$. The only dominant monomials in $\chi_q(S(\a)\otimes S(\beta))_{\leq4}$ are $m$, $mA_{1,3}^{-1}A_{3,3}^{-1}A_{2,4}^{-1}$ and $mA_{1,1}^{-1}A_{1,3}^{-1}A_{3,3}^{-1}A_{2,4}^{-1}$.
\[
mA_{1,3}^{-1}A_{3,3}^{-1}A_{2,4}^{-1}+mA_{1,1}^{-1}A_{1,3}^{-1}A_{3,3}^{-1}A_{2,4}^{-1}\leq \varphi_{\{1,2\}}(mA_{3,3}^{-1})\leq\chi_q(L(m)).
\]

\medskip
(10) $\a=\a_2+\a_5$, $Y^{\a}=Y_{2,1}$. If $\beta\in C_\a\backslash\{-\a_0,-\a_1,-\a_3,-\a_4,\a_0,\a_2,\a_4,\a_5,\a_1+\a_2,\a_2+\a_3,\a_2+\a_3+\a_4,\a_0+\a_1+\a_2,
\a_1+2\a_2+\a_3+\a_5,\a_0+\a_1+2\a_2+\a_3+\a_5,\a_1+2\a_2+\a_3+\a_4+\a_5,\a_0+\a_1+2\a_2+\a_3+\a_4+\a_5\}$, then $\chi_q(S(\a)\otimes S(\beta))_{\leq4}$ contains a unique dominant monomial and $S(\a)\otimes S(\beta)$ is simple.

For $\beta=-\a_0,-\a_1,-\a_3,-\a_4,\a_0,\a_2,\a_4,\a_5,\a_1+\a_2,\a_2+\a_3$, we have done.

If $\beta=\a_2+\a_3+\a_4$, $Y^{\beta}=Y_{2,5}Y_{3,2}$, denote $m=Y^{\a}Y^{\beta}$. The only dominant monomials in $\chi_q(S(\a)\otimes S(\beta))_{\leq4}$ are $m$ and $mA_{1,3}^{-1}A_{3,3}^{-1}A_{2,2}^{-1}A_{2,4}^{-1}$.

Since $\chi_q(L(m))_{\leq4}$ has two dominant monomials at most. If $mA_{1,3}^{-1}A_{3,3}^{-1}A_{2,2}^{-1}A_{2,4}^{-1}$ is not contained in $\chi_q(L(m))_{\leq4}$,
$L(m)$ is minuscule, and $\chi_q(L(m))=FM(m)$.
\[
FM(m)_{\leq4}=m(1+A_{3,3}^{-1}+A_{2,2}^{-1}+A_{1,3}^{-1}A_{2,2}^{-1}+2A_{3,3}^{-1}A_{2,2}^{-1}+
A_{3,3}^{-2}A_{2,2}^{-1}+
\]
\[
2A_{1,3}^{-1}A_{2,2}^{-1}A_{3,3}^{-1}+A_{1,3}^{-1}A_{2,2}^{-1}A_{3,3}^{-2}
+A_{1,3}^{-1}A_{2,2}^{-1}A_{3,3}^{-2}A_{2,4}^{-1}.
\]
By Proposition~\ref{qchadecom}, we fix $i=3$, then there is a
unique decomposition of $\chi_q(L(m))$ as a finite sum
\[
 \chi_q(L(m)) =  \sum_{\substack{m\in\M_{3,+}\\ \lambda_m\geq0}} \lambda_m \varphi_3(m).
\]
In other word, any monomial in $\chi_q(L(m)$ must be contained in some $\varphi_3(m')$, where $m'\in \M_{3,+}$ and $m'\leq\chi_q(L(m)$. But, the monomial $A_{1,3}^{-1}A_{2,2}^{-1}A_{3,3}^{-2}A_{2,4}^{-1}$ does not satisfy this property.

If $\beta=\a_0+\a_1+\a_2$, $Y^{\beta}=Y_{1,2}Y_{2,5}$, denote $m=Y^{\a}Y^{\beta}$. The only dominant monomials in $\chi_q(S(\a)\otimes S(\beta))_{\leq4}$ are $m$ and $mA_{1,3}^{-1}A_{3,3}^{-1}A_{2,2}^{-1}A_{2,4}^{-1}$.

This case is similar to the case $\beta=\a_2+\a_3+\a_4$.

If $\beta=\a_1+2\a_2+\a_3+\a_5$, $Y^{\beta}=Y_{1,0}Y_{1,2}Y_{2,5}Y_{3,0}Y_{3,2}$, denote $m=Y^{\a}Y^{\beta}$. The only dominant monomials in $\chi_q(S(\a)\otimes S(\beta))_{\leq4}$ are $m,\ mA_{1,3}^{-1}A_{3,3}^{-1}A_{2,4}^{-1},\ 2mA_{1,1}^{-1}A_{1,3}^{-1}A_{3,1}^{-1}A_{3,3}^{-1}A_{2,2}^{-1}A_{2,4}^{-1},\ mA_{1,3}^{-1}A_{3,1}^{-1}A_{3,3}^{-1}A_{2,2}^{-1}A_{2,4}^{-1},\ \\ mA_{1,1}^{-1}A_{1,3}^{-1}A_{3,3}^{-1}A_{2,2}^{-1}A_{2,4}^{-1}$ and $2mA_{1,3}^{-1}A_{3,3}^{-1}A_{2,2}^{-1}A_{2,4}^{-1}$.
\[
\begin{array}{ll}
\medskip
 mA_{1,3}^{-1}A_{3,3}^{-1}A_{2,4}^{-1}\leq\varphi_{2}(mA_{1,3}^{-1}A_{3,3}^{-1}),&
 2mA_{1,1}^{-1}A_{1,3}^{-1}A_{3,1}^{-1}A_{3,3}^{-1}A_{2,2}^{-1}A_{2,4}^{-1}\leq
 \varphi_{2}(mA_{1,1}^{-1}A_{1,3}^{-1}A_{3,3}^{-1}A_{3,1}^{-1}),\\
 mA_{1,3}^{-1}A_{3,1}^{-1}A_{3,3}^{-1}A_{2,2}^{-1}A_{2,4}^{-1}\leq\varphi_{2}(mA_{1,3}^{-1}A_{3,3}^{-1}A_{3,1}^{-1}),&
 mA_{3,3}^{-1}A_{1,1}^{-1}A_{1,3}^{-1}A_{2,2}^{-1}A_{2,4}^{-1}\leq\varphi_{2}(mA_{3,3}^{-1}A_{1,3}^{-1}A_{1,1}^{-1}).
\end{array}
\]

$2mA_{1,3}^{-1}A_{3,3}^{-1}A_{2,2}^{-1}A_{2,4}^{-1}\leq\varphi_{\{1,2\}}(2mA_{3,3}^{-1}A_{2,2}^{-1})_{\leq4}.$

If $\beta=\a_0+\a_1+2\a_2+\a_3+\a_5$, $Y^{\beta}=Y_{1,2}Y_{2,5}Y_{3,0}Y_{3,2}$, denote $m=Y^{\a}Y^{\beta}$. The only dominant monomials in $\chi_q(S(\a)\otimes S(\beta))_{\leq4}$ are $m$, $mA_{1,3}^{-1}A_{3,3}^{-1}A_{2,4}^{-1}$, $mA_{1,3}^{-1}A_{3,1}^{-1}A_{3,3}^{-1}A_{2,2}^{-1}A_{2,4}^{-1}$ and $2mA_{1,3}^{-1}A_{3,3}^{-1}A_{2,2}^{-1}A_{2,4}^{-1}$.
\[
\begin{array}{ll}
\medskip
 mA_{1,3}^{-1}A_{3,3}^{-1}A_{2,4}^{-1}\leq\varphi_{2}(mA_{1,3}^{-1}A_{3,3}^{-1}),&
 mA_{1,3}^{-1}A_{3,1}^{-1}A_{3,3}^{-1}A_{2,2}^{-1}A_{2,4}^{-1}\leq\varphi_{2}(mA_{1,3}^{-1}A_{3,3}^{-1}A_{3,1}^{-1}),\\
2mA_{1,3}^{-1}A_{3,3}^{-1}A_{2,2}^{-1}A_{2,4}^{-1}\leq\varphi_{\{1,2\}}(2mA_{3,3}^{-1}A_{2,2}^{-1}).&
\end{array}
\]

\medskip
(11) $\a=\a_3+\a_4$, $Y^{\a}=Y_{1,4}$. If $\beta\in C_\a\backslash\{-\a_0,-\a_1,-\a_2,-\a_5,\a_1+\a_2+\a_3,
\a_1+\a_2+\a_3+\a_5,\a_1+2\a_2+2\a_3+\a_4+\a_5,\a_0+2\a_1+2\a_2+2\a_3+\a_4+\a_5,\a_0+\a_1+2\a_2+2\a_3+\a_4+\a_5
\}$, then $\chi_q(S(\a)\otimes S(\beta))_{\leq4}$ contains a unique dominant monomial and $S(\a)\otimes S(\beta)$ is simple.

For $\beta=-\a_0,-\a_1,-\a_2,-\a_5$, we have done.

If $\beta=\a_1+\a_2+\a_3$, $Y^{\beta}=Y_{1,0}Y_{2,3}Y_{2,5}Y_{3,0}$, denote $m=Y^{\a}Y^{\beta}$. The only dominant monomials in $\chi_q(S(\a)\otimes S(\beta))_{\leq4}$ are $m$ and $mA_{1,1}^{-1}A_{2,2}^{-1}A_{3,1}^{-1}$.
\[
\begin{array}{ccc}
  \varphi_{1}(m)_{\leq4}=m(1+A_{1,1}^{-1}), & \varphi_{3}(mA_{1,1}^{-1})_{\leq4}=mA_{1,1}^{-1}(1+A_{3,1}^{-1}),&
 \varphi_{2}(mA_{1,1}^{-1}A_{3,1}^{-1})_{\leq4}=mA_{1,1}^{-1}A_{3,1}^{-1}(1+A_{2,2}^{-1}).
\end{array}
\]

If $\beta=\a_1+\a_2+\a_3+\a_5$, $Y^{\beta}=Y_{1,0}Y_{2,3}Y_{3,0}$, denote $m=Y^{\a}Y^{\beta}$. The only dominant monomials in $\chi_q(S(\a)\otimes S(\beta))_{\leq4}$ are $m$ and $mA_{1,1}^{-1}A_{2,2}^{-1}A_{3,1}^{-1}$.
\[
mA_{1,1}^{-1}A_{2,2}^{-1}A_{3,1}^{-1}\leq\varphi_{2}(mA_{1,1}^{-1}A_{3,1}^{-1})_{\leq4}=mA_{1,1}^{-1}A_{3,1}^{-1}(1+
A_{2,2}^{-1})(1+A_{2,4}^{-1}+A_{2,4}^{-1}A_{2,2}^{-1})\leq\chi_q(L(m)).
\]

If $\beta=\a_1+2\a_2+2\a_3+\a_4+\a_5$, $Y^{\beta}=Y_{1,0}Y_{2,3}Y_{2,5}Y_{3,0}Y_{3,2}$, denote $m=Y^{\a}Y^{\beta}$. The only dominant monomials in $\chi_q(S(\a)\otimes S(\beta))_{\leq4}$ are $m$, $mA_{2,4}^{-1}A_{3,3}^{-1}$ and $mA_{1,1}^{-1}A_{2,2}^{-1}A_{2,4}^{-1}A_{3,1}^{-1}A_{3,3}^{-1}$.
\[
\begin{array}{ll}
mA_{2,4}^{-1}A_{3,3}^{-1}\leq\varphi_{2}(mA_{3,3}^{-1})_{\leq4},&
mA_{1,1}^{-1}A_{2,2}^{-1}A_{2,4}^{-1}A_{3,1}^{-1}A_{3,3}^{-1}\leq
\varphi_{2}(mA_{1,1}^{-1}A_{3,1}^{-1}A_{3,3}^{-1})_{\leq4}.
\end{array}
\]

If $\beta=\a_0+2\a_1+2\a_2+2\a_3+\a_4+\a_5$, $Y^{\beta}=Y_{1,0}Y_{2,3}^{2}Y_{2,5}Y_{3,0}$, denote $m=Y^{\a}Y^{\beta}$. The only dominant monomials in $\chi_q(S(\a)\otimes S(\beta))_{\leq4}$ are $m$, $mA_{2,4}^{-1}$ and $mA_{1,1}^{-1}A_{2,2}^{-1}A_{2,4}^{-1}A_{3,1}^{-1}$.
\[
\begin{array}{ll}
mA_{2,4}^{-1}\leq\varphi_{2}(m)_{\leq4}, &
mA_{1,1}^{-1}A_{2,2}^{-1}A_{2,4}^{-1}A_{3,1}^{-1}\leq\varphi_{2}(mA_{1,1}^{-1}A_{3,1}^{-1})_{\leq4}.
\end{array}
\]

If $\beta=\a_0+\a_1+2\a_2+2\a_3+\a_4+\a_5$, $Y^{\beta}=Y_{2,3}Y_{2,5}Y_{3,0}Y_{3,2}$, denote $m=Y^{\a}Y^{\beta}$. The only dominant monomials in $\chi_q(S(\a)\otimes S(\beta))_{\leq4}$ are $m$ and  $mA_{3,3}^{-1}A_{2,4}^{-1}$.
\[
\begin{array}{cc}
\varphi_{3}(m)_{\leq4}=m(1+A_{3,3}^{-1}+A_{3,3}^{-1}A_{3,1}^{-1}), &
mA_{3,3}^{-1}A_{2,4}^{-1}\leq\varphi_{2}(mA_{3,3}^{-1})_{\leq4}=mA_{3,3}^{-1}(1+A_{2,4}^{-1})\leq\chi_q(L(m)).
\end{array}
\]

\medskip
(12) $\a=\a_1+\a_2+\a_3$, $Y^{\a}=Y_{1,0}Y_{2,3}Y_{2,5}Y_{3,0}$. In this case, $\chi_q(L(Y^{\a}))$ has two dominant monomials, so we need to consider every $\beta\in C_\a.$

For $\beta=-\a_0,-\a_4,-\a_5,\a_1,\a_3,\a_0+\a_1,\a_1+\a_2,\a_2+\a_3,\a_3+\a_4$, we have done.

Up to the symmetries $1\leftrightarrow 3$, $0\leftrightarrow 4$ of $\widetilde{I}$, we reduces to the cases:
$\beta\in \{\a_1+\a_2+\a_3,\a_1+\a_2+\a_5,\a_0+\a_1+\a_2+\a_3,
\a_1+2\a_2+\a_3+\a_5,\a_1+\a_2+\a_3+\a_5,\a_0+\a_1+\a_2+\a_3+\a_5,\a_0+2\a_1+2\a_2+\a_3+\a_5,\a_0+2\a_1+2\a_2+2\a_3+\a_4+\a_5
\}$.

If $\beta=\a_1+\a_2+\a_3$, $Y^{\beta}=Y^{\a}=Y_{1,0}Y_{2,3}Y_{2,5}Y_{3,0}$, denote $m=Y^{2\a}$. The only dominant monomials in $\chi_q(S(\a)\otimes S(\a))_{\leq4}$ are $m$, $2mA_{1,1}^{-1}A_{2,2}^{-1}A_{3,1}^{-1}$ and $mA_{1,1}^{-2}A_{2,2}^{-2}A_{3,1}^{-2}$.

$mA_{1,1}^{-2}A_{2,2}^{-2}A_{3,1}^{-2}\leq\varphi_2(mA_{1,1}^{-2}A_{3,1}^{-2})_{\leq4}$,\quad $2mA_{1,1}^{-1}A_{2,2}^{-1}A_{3,1}^{-1}\leq\varphi_{\{2,3\}}(2mA_{1,1}^{-1})$.

If $\beta=\a_1+\a_2+\a_5$, $Y^{\beta}=Y_{1,0}Y_{1,2}$, denote $m=Y^{\a}Y^{\beta}$. The only dominant monomials in $\chi_q(S(\a)\otimes S(\beta))_{\leq4}$ are $m,\ mA_{1,1}^{-1},\ mA_{1,1}^{-1}A_{2,2}^{-1}A_{3,1}^{-1},\ mA_{1,3}^{-1}A_{2,4}^{-1},\ mA_{1,1}^{-1}A_{1,3}^{-1}A_{2,2}^{-1}A_{2,4}^{-1}$ and $2mA_{1,1}^{-1}A_{1,3}^{-1}A_{2,2}^{-1}A_{2,4}^{-1}A_{3,1}^{-1}$.

$mA_{1,1}^{-1}\leq\varphi_{1}(m)_{\leq4},$ \quad $mA_{1,3}^{-1}A_{2,4}^{-1}\leq\varphi_2(mA_{1,3}^{-1})_{\leq4},$ \quad
$mA_{1,1}^{-1}A_{2,2}^{-1}A_{3,1}^{-1}\leq\varphi_2(mA_{1,1}^{-1}A_{3,1}^{-1})_{\leq4},$

$2mA_{1,1}^{-1}A_{1,3}^{-1}A_{2,2}^{-1}A_{2,4}^{-1}A_{3,1}^{-1}\leq
\varphi_{2}(2mA_{1,1}^{-1}A_{1,3}^{-1}A_{3,1}^{-1})_{\leq4},$\quad
$mA_{1,1}^{-1}A_{1,3}^{-1}A_{2,2}^{-1}A_{2,4}^{-1}\leq\varphi_{\{1,2\}}(m)_{\leq4}.$

If $\beta=\a_0+\a_1+\a_2+\a_3$, $Y^{\beta}=Y_{2,3}Y_{2,5}Y_{3,0}$, denote $m=Y^{\a}Y^{\beta}$. The only dominant monomials in $\chi_q(S(\a)\otimes S(\beta))_{\leq4}$ are $m$ and $mA_{1,1}^{-1}A_{2,2}^{-1}A_{3,1}^{-1}$.

$mA_{1,1}^{-1}A_{2,2}^{-1}A_{3,1}^{-1}\leq\varphi_{\{2,3\}}(mA_{1,1}^{-1})_{\leq4}.$

If $\beta=\a_1+2\a_2+\a_3+\a_5$, $Y^{\beta}=Y_{1,0}Y_{1,2}Y_{2,5}Y_{3,0}Y_{3,2}$, denote $m=Y^{\a}Y^{\beta}$. The only dominant monomials in $\chi_q(S(\a)\otimes S(\beta))_{\leq4}$ are $m,\ mA_{1,1}^{-1},\ mA_{3,1}^{-1},\ mA_{1,1}^{-1}A_{3,1}^{-1},\ mA_{1,3}^{-1}A_{3,3}^{-1}A_{2,4}^{-1},\ mA_{1,1}^{-1}A_{2,2}^{-1}A_{3,1}^{-1}$ and \\ $2mA_{1,1}^{-1}A_{1,3}^{-1}A_{3,1}^{-1}A_{3,3}^{-1}A_{2,2}^{-1}A_{2,4}^{-1}$.
\[
\begin{array}{ll}
\medskip
  mA_{1,1}^{-1}\leq\varphi_1(m)_{\leq4}, &  mA_{3,1}^{-1}\leq\varphi_3(m)_{\leq4},\\
  \medskip
  mA_{1,1}^{-1}A_{3,1}^{-1}\leq\varphi_3(mA_{1,1}^{-1})_{\leq4},&
  mA_{1,3}^{-1}A_{3,3}^{-1}A_{2,4}^{-1}\leq\varphi_2(mA_{1,3}^{-1}A_{3,3}^{-1})_{\leq4},\\
  mA_{1,1}^{-1}A_{2,2}^{-1}A_{3,1}^{-1}\leq\varphi_2(mA_{1,1}^{-1}A_{3,1}^{-1})_{\leq4},&
  2mA_{1,1}^{-1}A_{1,3}^{-1}A_{3,1}^{-1}A_{3,3}^{-1}A_{2,2}^{-1}A_{2,4}^{-1}\leq
\varphi_{\{2,3\}}(2mA_{1,1}^{-1}A_{1,3}^{-1})_{\leq4}.
\end{array}
\]

If $\beta=\a_1+\a_2+\a_3+\a_5$, $Y^{\beta}=Y_{1,0}Y_{2,3}Y_{3,0}$, denote $m=Y^{\a}Y^{\beta}$. The only dominant monomials in $\chi_q(S(\a)\otimes S(\beta))_{\leq4}$ are $m,\ mA_{2,4}^{-1},\ mA_{1,1}^{-1}A_{2,2}^{-1}A_{2,4}^{-1},\  mA_{3,1}^{-1}A_{2,2}^{-1}A_{2,4}^{-1},\ mA_{1,1}^{-2}A_{2,2}^{-2}A_{3,1}^{-2},\  2mA_{1,1}^{-1}A_{2,2}^{-1}A_{3,1}^{-1}$ and $5mA_{1,1}^{-1}A_{3,1}^{-1}A_{2,2}^{-1}A_{2,4}^{-1}$.
\[
\begin{array}{ll}
\medskip
  \varphi_2(m)_{\leq4}=m(1+A_{2,4}^{-1}), &
  \varphi_1(m)_{\leq4}=m(1+A_{1,1}^{-1})^{2},\\
  \varphi_3(mA_{1,1}^{-2})_{\leq4}=mA_{1,1}^{-2}(1+A_{3,1}^{-1})^{2},&
  \varphi_2(mA_{1,1}^{-2}A_{3,1}^{-2})_{\leq4}=mA_{1,1}^{-2}A_{3,1}^{-2}(1+A_{2,2}^{-1})^{2}
  (1+A_{2,4}^{-1}+A_{2,4}^{-1}A_{2,2}^{-1}),
\end{array}
\]
Consider $\varphi_{\{2,3\}}(2mA_{1,1}^{-1})$, $\overline{mA_{1,1}^{-1}}=(Y_{2,1}Y_{2,3}Y_{2,5})(Y_{2,3}Y_{3,0})(Y_{3,0})$, thus
$2mA_{1,1}^{-1}A_{2,2}^{-1}A_{3,1}^{-1}+4mA_{1,1}^{-1}A_{3,1}^{-1}A_{2,2}^{-1}A_{2,4}^{-1}\\
\leq
\varphi_{\{2,3\}}(2mA_{1,1}^{-1})_{\leq4}.$  Besides, $\varphi_3(mA_{1,1}^{-1}A_{2,2}^{-1}A_{2,4}^{-1})_{\leq4}=mA_{1,1}^{-1}A_{2,2}^{-1}A_{2,4}^{-1}(1+A_{3,1}^{-1})$,
so $mA_{1,1}^{-1}A_{3,1}^{-1}A_{2,2}^{-1}A_{2,4}^{-1}$ has coefficient $5$.
Consider $\varphi_{\{1,2\}}(m)$, we have $\overline{m}=(Y_{1,0}Y_{2,3}Y_{2,5})(Y_{1,0}Y_{2,3})$, thus
$mA_{1,1}^{-1}A_{2,2}^{-1}A_{2,4}^{-1}\leq
\varphi_{\{1,2\}}(m)_{\leq4}.$

If $\beta=\a_0+\a_1+\a_2+\a_3+\a_5$, $Y^{\beta}=Y_{2,3}Y_{3,0}$, denote $m=Y^{\a}Y^{\beta}$. The only dominant monomials in $\chi_q(S(\a)\otimes S(\beta))_{\leq4}$ are $m,\ mA_{2,4}^{-1}, \ mA_{3,1}^{-1}A_{2,2}^{-1}A_{2,4}^{-1}, \ mA_{1,1}^{-1}A_{2,2}^{-1}A_{3,1}^{-1}$,  and  $2mA_{1,1}^{-1}A_{3,1}^{-1}A_{2,2}^{-1}A_{2,4}^{-1}$.
\[
\begin{array}{ll}
\medskip
  mA_{2,4}^{-1}\leq\varphi_2(m)_{\leq4}, &
 mA_{3,1}^{-1}A_{2,2}^{-1}A_{2,4}^{-1}\leq\varphi_{\{2,3\}}(m)_{\leq4},\\
  mA_{1,1}^{-1}A_{2,2}^{-1}A_{3,1}^{-1}\leq \varphi_{\{2,3\}}(mA_{1,1}^{-1})_{\leq4},&
 2mA_{1,1}^{-1}A_{3,1}^{-1}A_{2,2}^{-1}A_{2,4}^{-1}\leq \varphi_2(2mA_{1,1}^{-1}A_{3,1}^{-1})_{\leq4}.
\end{array}
\]

If $\beta=\a_0+2\a_1+2\a_2+\a_3+\a_5$, $Y^{\beta}=Y_{1,0}Y_{1,2}Y_{2,3}Y_{2,5}Y_{3,0}$, denote $m=Y^{\a}Y^{\beta}$. The only domi\-nant monomials in $\chi_q(S(\a)\otimes S(\beta))_{\leq4}$ are $m,\ mA_{1,1}^{-1},\ mA_{1,3}^{-1}A_{2,4}^{-1},\ 2mA_{1,1}^{-1}A_{1,3}^{-1}A_{3,1}^{-1}A_{2,2}^{-1}A_{2,4}^{-1},\ mA_{1,1}^{-1}A_{2,2}^{-1}A_{3,1}^{-1}$  and  $mA_{1,1}^{-2}A_{1,3}^{-1}A_{2,2}^{-2}A_{2,4}^{-1}A_{3,1}^{-2}$.
\[
\begin{array}{ll}
\medskip
  mA_{1,1}^{-1}\leq\varphi_1(m)_{\leq4}, &
  mA_{1,3}^{-1}A_{2,4}^{-1}\leq\varphi_{2}(mA_{1,3}^{-1})_{\leq4},\\
  \medskip
  mA_{1,1}^{-1}A_{2,2}^{-1}A_{3,1}^{-1}\leq \varphi_{\{2,3\}}(mA_{1,1}^{-1})_{\leq4},&
 2mA_{1,1}^{-1}A_{1,3}^{-1}A_{3,1}^{-1}A_{2,2}^{-1}A_{2,4}^{-1}\leq \varphi_{\{2,3\}}(2mA_{1,1}^{-1}A_{1,3}^{-1})_{\leq4},\\
 mA_{1,1}^{-2}A_{1,3}^{-1}A_{2,2}^{-2}A_{2,4}^{-1}A_{3,1}^{-2}\leq
 \varphi_{2}(mA_{1,1}^{-2}A_{1,3}^{-1}A_{3,1}^{-2})_{\leq4}.
\end{array}
\]

If $\beta=\a_0+2\a_1+2\a_2+2\a_3+\a_4+\a_5$, $Y^{\beta}=Y_{1,0}Y_{2,3}^{2}Y_{2,5}Y_{3,0}$, denote $m=Y^{\a}Y^{\beta}$. The only dominant monomials in $\chi_q(S(\a)\otimes S(\beta))_{\leq4}$ are $m$, \ $mA_{2,4}^{-1}$,\ $mA_{1,1}^{-1}A_{2,2}^{-1}A_{3,1}^{-1}$,\ $2mA_{1,1}^{-1}A_{3,1}^{-1}A_{2,2}^{-1}A_{2,4}^{-1}$, \  and  $mA_{1,1}^{-2}A_{2,2}^{-2}A_{2,4}^{-1}A_{3,1}^{-2}$.

We can check that all the monomials excepting $mA_{1,1}^{-1}A_{2,2}^{-1}A_{3,1}^{-1}$ of $\chi_q(S(\a)\otimes S(\beta))_{\leq4}$ are contained in $\chi_q(L(m))\leq4$. If
$mA_{1,1}^{-1}A_{2,2}^{-1}A_{3,1}^{-1}$ is not contained in $\chi_q(L(m))_{\leq4}$,
by Proposition~\ref{qchadecom}, we fix $i=1$, then there is a
unique decomposition of $\chi_q(L(m))$ as a finite sum
$ \chi_q(L(m)) = \sum_{\substack{m\in\M_{1,+}\\ \lambda_m\geq0}} \lambda_m \varphi_1(m).$
In other word, any monomial in $\chi_q(L(m))$ must be contained in some $\varphi_1(m')$, where $m'\in \M_{1,+}$ and $m'\leq\chi_q(L(m)$. But, obviously there does not exist such monomial $m'$ such that $mA_{1,1}^{-2}A_{2,2}^{-1}A_{3,1}^{-1}\leq\varphi_1(m')$.


\medskip
(13) $\a=\a_0+\a_1+\a_2$, $Y^{\a}=Y_{1,2}Y_{2,5}$. If $\beta\in C_\a\backslash\{-\a_3,-\a_4,-\a_5,\a_0,\a_2+\a_5,\a_0+\a_1+\a_2+\a_5,
\a_0+\a_1+2\a_2+\a_3+\a_5,\a_0+\a_1+2\a_2+\a_3+\a_4+\a_5\}$, then $\chi_q(S(\a)\otimes S(\beta))_{\leq4}$ contains a unique dominant monomial and $S(\a)\otimes S(\beta)$ is simple.

For $\beta=-\a_3,-\a_4,-\a_5,\a_0,\a_2+\a_5$, we have done.

If $\beta=\a_0+\a_1+\a_2+\a_5$, $Y^{\beta}=Y_{1,2}$, denote $m=Y^{\a}Y^{\beta}$. The only dominant monomials in $\chi_q(S(\a)\otimes S(\beta))_{\leq4}$ are $m$ and $mA_{1,3}^{-1}A_{2,4}^{-1}$.

Consider $\varphi_{\{1,2\}}(m)$,  $\overline{m}=(Y_{1,2}Y_{2,5})(Y_{1,2})$, thus
$mA_{1,3}^{-1}A_{2,4}^{-1}\leq\varphi_{\{1,2\}}(m)_{\leq4}.$

If $\beta=\a_0+\a_1+2\a_2+\a_3+\a_5$, $Y^{\beta}=Y_{1,2}Y_{2,5}Y_{3,0}Y_{3,2}$, denote $m=Y^{\a}Y^{\beta}$. The only dominant monomials in $\chi_q(S(\a)\otimes S(\beta))_{\leq4}$ are $m$ and $mA_{1,3}^{-1}A_{2,4}^{-1}A_{3,3}^{-1}$.

Consider $\varphi_{\{1,2\}}(mA_{3,3}^{-1})$, $\overline{mA_{3,3}^{-1}}=(Y_{2,3}Y_{2,5})(Y_{1,2}Y_{2,5})(Y_{1,2})$, thus
$mA_{3,3}^{-1}A_{1,3}^{-1}A_{2,4}^{-1}\leq\varphi_{\{1,2\}}(mA_{3,3}^{-1})_{\leq4}.$

If $\beta=\a_0+\a_1+2\a_2+\a_3+\a_4+\a_5$, $Y^{\beta}=Y_{1,2}Y_{2,5}Y_{3,2}$, denote $m=Y^{\a}Y^{\beta}$. The only dominant monomials in $\chi_q(S(\a)\otimes S(\beta))_{\leq4}$ are $m$ and $mA_{1,3}^{-1}A_{2,4}^{-1}A_{3,3}^{-1}$.

Consider $\varphi_{\{1,2\}}(mA_{3,3}^{-1})$, $\overline{mA_{3,3}^{-1}}=(Y_{2,3}Y_{2,5})(Y_{1,2}Y_{2,5})(Y_{1,2})$, thus 
$mA_{3,3}^{-1}A_{1,3}^{-1}A_{2,4}^{-1}\leq\varphi_{\{1,2\}}(mA_{3,3}^{-1})_{\leq4}.$

\medskip
(14) $\a=\a_1+\a_2+\a_5$, $Y^{\a}=Y_{1,0}Y_{1,2}$. If $\beta\in C_\a\backslash\{-\a_0,-\a_3,-\a_4,\a_1,\a_1+\a_2,\a_1+\a_2+\a_3,\a_1+\a_2+\a_3+\a_4,\a_1+2\a_2+\a_3+\a_5,
\a_1+\a_2+\a_3+\a_5,\a_1+\a_2+\a_3+\a_4+\a_5,\a_0+2\a_1+2\a_2+\a_3+\a_5,\a_1+2\a_2+\a_3+\a_4+\a_5,
\a_1+2\a_2+2\a_3+\a_4+\a_5,\a_0+2\a_1+2\a_2+2\a_3+\a_4+\a_5,\a_0+2\a_1+2\a_2+\a_3+\a_4+\a_5,
\a_0+2\a_1+3\a_2+2\a_3+\a_4+2\a_5\}$, then $\chi_q(S(\a)\otimes S(\beta))_{\leq4}$ contains a unique dominant monomial and $S(\a)\otimes S(\beta)$ is simple.

For $\beta=-\a_0,-\a_3,-\a_4,\a_1,\a_1+\a_2,\a_1+\a_2+\a_3$, we have done.

If $\beta=\a_1+\a_2+\a_3+\a_4$, $Y^{\beta}=Y_{1,0}Y_{2,3}Y_{2,5}$, denote $m=Y^{\a}Y^{\beta}$. The only dominant monomials in $\chi_q(S(\a)\otimes S(\beta))_{\leq4}$ are $m$, $mA_{1,1}^{-1}$, $mA_{1,3}^{-1}A_{2,4}^{-1}$ and $mA_{1,1}^{-1}A_{1,3}^{-1}A_{2,2}^{-1}A_{2,4}^{-1}$.

Consider $\varphi_{\{1,2\}}(m)$, we have $\overline{m}=(Y_{1,0}Y_{2,3}Y_{2,5})(Y_{1,0}Y_{1,2})$, thus all the $4$ dominant monomials occur in $\varphi_{\{1,2\}}(m)$.

If $\beta=\a_1+2\a_2+\a_3+\a_5$, $Y^{\beta}=Y_{1,0}Y_{1,2}Y_{3,0}Y_{3,2}Y_{2,5}$, denote $m=Y^{\a}Y^{\beta}$. The only dominant monomials in $\chi_q(S(\a)\otimes S(\beta))_{\leq4}$ are $m,\  mA_{1,3}^{-1}A_{2,4}^{-1},\  mA_{1,1}^{-1}A_{1,3}^{-1}A_{2,4}^{-1},\  mA_{1,1}^{-1}A_{1,3}^{-1}A_{2,2}^{-1}A_{2,4}^{-1}A_{3,3}^{-1},\\
2mA_{1,1}^{-1}A_{1,3}^{-1}A_{2,4}^{-1}A_{3,3}^{-1},\
 2mA_{1,3}^{-1}A_{2,4}^{-1}A_{3,3}^{-1}$ and $2mA_{1,1}^{-1}A_{1,3}^{-1}A_{2,2}^{-1}A_{2,4}^{-1}A_{3,1}^{-1}A_{3,3}^{-1}$.
\[
\begin{array}{l}
\medskip
2mA_{1,1}^{-1}A_{1,3}^{-1}A_{2,2}^{-1}A_{2,4}^{-1}A_{3,1}^{-1}A_{3,3}^{-1}\leq \varphi_{2}(2mA_{1,1}^{-1}A_{1,3}^{-1}A_{3,1}^{-1}A_{3,3}^{-1})_{\leq4},\\
\medskip
mA_{1,3}^{-1}A_{2,4}^{-1}+mA_{1,1}^{-1}A_{1,3}^{-1}A_{2,4}^{-1}\leq \varphi_{\{1,2\}}(m)_{\leq4},\\
mA_{1,1}^{-1}A_{1,3}^{-1}A_{2,2}^{-1}A_{2,4}^{-1}A_{3,3}^{-1}+2mA_{1,1}^{-1}A_{1,3}^{-1}A_{2,4}^{-1}A_{3,3}^{-1}+
2mA_{1,3}^{-1}A_{2,4}^{-1}A_{3,3}^{-1}\leq \varphi_{\{1,2\}}(mA_{3,3}^{-1})_{\leq4}.
\end{array}
\]

If $\beta=\a_1+\a_2+\a_3+\a_5$, $Y^{\beta}=Y_{1,0}Y_{2,3}Y_{3,0}$, denote $m=Y^{\a}Y^{\beta}$. The only dominant monomials in $\chi_q(S(\a)\otimes S(\beta))_{\leq4}$ are $m,\ mA_{1,1}^{-1}$ and $mA_{1,1}^{-1}A_{2,2}^{-1}A_{3,1}^{-1}$.
\[
\begin{array}{ll}
mA_{1,1}^{-1}\leq\varphi_1(m)_{\leq4}, &
 mA_{1,1}^{-1}A_{2,2}^{-1}A_{3,1}^{-1}\leq\varphi_2(mA_{1,1}^{-1}A_{3,1}^{-1})_{\leq4}.
\end{array}
\]

If $\beta=\a_1+\a_2+\a_3+\a_4+\a_5$, $Y^{\beta}=Y_{1,0}Y_{2,3}$, denote $m=Y^{\a}Y^{\beta}$. The only dominant monomials in $\chi_q(S(\a)\otimes S(\beta))_{\leq4}$ are $m$ and $mA_{1,1}^{-1}$.
\[
mA_{1,1}^{-1}\leq\varphi_1(m)_{\leq4}=m(1+A_{1,1}^{-1})(1+A_{1,3}^{-1}+A_{1,3}^{-1}A_{1,1}^{-1}).
\]

If $\beta=\a_0+2\a_1+2\a_2+\a_3+\a_5$, $Y^{\beta}=Y_{1,0}Y_{1,2}Y_{2,3}Y_{2,5}Y_{3,0}$, denote $m=Y^{\a}Y^{\beta}$. The only dominant monomials in $\chi_q(S(\a)\otimes S(\beta))_{\leq4}$ are $m$, $mA_{1,1}^{-1}A_{1,3}^{-1}A_{2,4}^{-1}A_{2,2}^{-1}$, $2mA_{1,3}^{-1}A_{2,4}^{-1}$, $2mA_{1,1}^{-1}A_{1,3}^{-1}A_{2,4}^{-1}A_{2,2}^{-1}A_{3,1}^{-1}$ and $2mA_{1,1}^{-1}A_{1,3}^{-1}A_{2,4}^{-1}$.
\[
\begin{array}{l}
\medskip
 2mA_{1,1}^{-1}A_{1,3}^{-1}A_{2,4}^{-1}A_{2,2}^{-1}A_{3,1}^{-1}\leq
 \varphi_{2}(2mA_{1,1}^{-1}A_{1,3}^{-1}A_{3,1}^{-1})_{\leq4},\\
 m+2mA_{1,1}^{-1}A_{1,3}^{-1}+mA_{1,1}^{-1}A_{1,3}^{-1}A_{2,4}^{-1}A_{2,2}^{-1}+2mA_{1,1}^{-1}A_{1,3}^{-1}A_{2,4}^{-1}
\leq\varphi_{\{1,2\}}(m)_{\leq4}.
\end{array}
\]


If $\beta=\a_1+2\a_2+\a_3+\a_4+\a_5$, $Y^{\beta}=Y_{1,0}Y_{1,2}Y_{2,5}Y_{3,2}$, denote $m=Y^{\a}Y^{\beta}$. The only dominant monomials in $\chi_q(S(\a)\otimes S(\beta))_{\leq4}$ are $m,\ mA_{1,1}^{-1}A_{1,3}^{-1}A_{2,4}^{-1},\  mA_{1,3}^{-1}A_{2,4}^{-1},\  2mA_{1,3}^{-1}A_{2,4}^{-1}A_{3,3}^{-1},\
2mA_{1,1}^{-1}A_{1,3}^{-1}A_{2,4}^{-1}A_{3,3}^{-1}$ and $mA_{1,1}^{-1}A_{1,3}^{-1}A_{2,2}^{-1}A_{2,4}^{-1}A_{3,3}^{-1}$.
\[
\begin{array}{l}
\medskip
 m+mA_{1,3}^{-1}A_{2,4}^{-1}+mA_{1,1}^{-1}A_{1,3}^{-1}A_{2,4}^{-1}
\leq\varphi_{\{1,2\}}(m)_{\leq4},\\
2mA_{1,3}^{-1}A_{2,4}^{-1}A_{3,3}^{-1}+2mA_{1,1}^{-1}A_{1,3}^{-1}A_{2,4}^{-1}A_{3,3}^{-1}+
mA_{1,1}^{-1}A_{1,3}^{-1}A_{2,2}^{-1}A_{2,4}^{-1}A_{3,3}^{-1}
\leq\varphi_{\{1,2\}}(mA_{3,3}^{-1})_{\leq4}.
\end{array}
\]

If $\beta=\a_1+2\a_2+2\a_3+\a_4+\a_5$, $Y^{\beta}=Y_{1,0}Y_{2,3}Y_{2,5}Y_{3,0}Y_{3,2}$, denote $m=Y^{\a}Y^{\beta}$. The only dominant monomials in $\chi_q(S(\a)\otimes S(\beta))_{\leq4}$ are $m,\ mA_{1,1}^{-1},\ mA_{1,3}^{-1}A_{2,4}^{-1},\ mA_{3,3}^{-1}A_{2,4}^{-1},\ mA_{1,1}^{-1}A_{1,3}^{-1}A_{2,2}^{-1}A_{2,4}^{-1},\\ mA_{1,1}^{-1}A_{2,2}^{-1}A_{2,4}^{-1}A_{3,1}^{-1}A_{3,3}^{-1},\
mA_{1,1}^{-1}A_{2,4}^{-1}A_{3,3}^{-1},\ 2mA_{1,3}^{-1}A_{2,4}^{-1}A_{3,3}^{-1},\ 2mA_{1,1}^{-1}A_{1,3}^{-1}A_{2,2}^{-1}A_{2,4}^{-1}A_{3,1}^{-1}A_{3,3}^{-1}$ and \\ $mA_{1,1}^{-1}A_{1,3}^{-1}A_{2,2}^{-1}A_{2,4}^{-1}A_{3,3}^{-1}$.
\[
\begin{array}{l}
\medskip
mA_{1,1}^{-1}+mA_{1,3}^{-1}A_{2,4}^{-1}+mA_{1,1}^{-1}A_{1,3}^{-1}A_{2,2}^{-1}A_{2,4}^{-1}\leq
\varphi_{\{1,2\}}(m)_{\leq4},\\
\medskip
mA_{3,3}^{-1}(A_{2,4}^{-1}+A_{1,1}^{-1}A_{2,4}^{-1}+2A_{1,3}^{-1}A_{2,4}^{-1}+
A_{1,1}^{-1}A_{1,3}^{-1}A_{2,2}^{-1}A_{2,4}^{-1})
\leq\varphi_{\{1,2\}}(mA_{3,3}^{-1})_{\leq4},\\
mA_{3,1}^{-1}A_{3,3}^{-1}(A_{1,1}^{-1}A_{2,2}^{-1}A_{2,4}^{-1}+2A_{1,1}^{-1}A_{1,3}^{-1}A_{2,2}^{-1}A_{2,4}^{-1})
\leq\varphi_{\{1,2\}}(mA_{3,1}^{-1}A_{3,3}^{-1})_{\leq4}.
\end{array}
\]

If $\beta=\a_0+2\a_1+2\a_2+2\a_3+\a_4+\a_5$, $Y^{\beta}=Y_{1,0}Y_{2,3}^{2}Y_{2,5}Y_{3,0}$, denote $m=Y^{\a}Y^{\beta}$. The only dominant monomials in $\chi_q(S(\a)\otimes S(\beta))_{\leq4}$ are $m,\ mA_{1,1}^{-1},\ mA_{2,4}^{-1},\ mA_{1,1}^{-1}A_{2,4}^{-1},\ 2mA_{1,3}^{-1}A_{2,4}^{-1},\\ mA_{1,1}^{-1}A_{1,3}^{-1}A_{2,2}^{-1}A_{2,4}^{-1},\
mA_{1,1}^{-1}A_{3,1}^{-1}A_{2,2}^{-1}A_{2,4}^{-1}$ and
$2mA_{1,1}^{-1}A_{1,3}^{-1}A_{2,2}^{-1}A_{2,4}^{-1}A_{3,1}^{-1}$.
\[
\begin{array}{l}
\medskip
mA_{1,1}^{-1}+mA_{2,4}^{-1}+mA_{1,1}^{-1}A_{2,4}^{-1}+2mA_{1,3}^{-1}A_{2,4}^{-1}+ mA_{1,1}^{-1}A_{1,3}^{-1}A_{2,2}^{-1}A_{2,4}^{-1}
\leq\varphi_{\{1,2\}}(m)_{\leq4},\\
mA_{3,1}^{-1}(A_{1,1}^{-1}A_{2,2}^{-1}A_{2,4}^{-1}+2A_{1,1}^{-1}A_{1,3}^{-1}A_{2,2}^{-1}A_{2,4}^{-1})
\leq\varphi_{\{1,2\}}(mA_{3,1}^{-1})_{\leq4}.
\end{array}
\]


If $\beta=\a_0+2\a_1+2\a_2+\a_3+\a_4+\a_5$, $Y^{\beta}=Y_{1,0}Y_{1,2}Y_{2,3}Y_{2,5}$, denote $m=Y^{\a}Y^{\beta}$. The only dominant monomials in $\chi_q(S(\a)\otimes S(\beta))_{\leq4}$ are $m$, $2mA_{1,3}^{-1}A_{2,4}^{-1}$, $2mA_{1,1}^{-1}A_{1,3}^{-1}A_{2,4}^{-1}$ and $mA_{1,1}^{-1}A_{1,3}^{-1}A_{2,2}^{-1}A_{2,4}^{-1}$.
\[
\varphi_{\{1,2\}}(m)_{\leq4}=m(1+A_{1,3}^{-1}+A_{1,1}^{-1} A_{1,3}^{-1}+A_{1,3}^{-1} A_{2,4}^{-1}+A_{1,1}^{-1} A_{1,3}^{-1} A_{2,4}^{-1}+A_{1,1}^{-1} A_{1,3}^{-1} A_{2,4}^{-1} A_{2,2}^{-1})(1+A_{1,3}^{-1}+
\]
\[
A_{1,3}^{-1} A_{1,1}^{-1}+A_{1,3}^{-1} A_{2,4}^{-1}+A_{1,3}^{-1} A_{1,1}^{-1}A_{2,4}^{-1}).
\]


If $\beta=\a_0+2\a_1+3\a_2+2\a_3+\a_4+2\a_5$, $Y^{\beta}=Y_{1,0}Y_{1,2}Y_{2,3}Y_{2,5}Y_{3,0}Y_{3,2}$, denote $m=Y^{\a}Y^{\beta}$. The only dominant monomials in $\chi_q(S(\a)\otimes S(\beta))_{\leq4}$ are $m$, $2mA_{1,3}^{-1}A_{2,4}^{-1}$, $mA_{3,3}^{-1}A_{2,4}^{-1}$, $2mA_{1,1}^{-1}A_{1,3}^{-1}A_{2,4}^{-1}$,
$4mA_{1,3}^{-1}A_{3,3}^{-1}A_{2,4}^{-1}$, $mA_{1,1}^{-1}A_{1,3}^{-1}A_{2,2}^{-1}A_{2,4}^{-1}$, $4mA_{1,1}^{-1}A_{1,3}^{-1}A_{3,3}^{-1}A_{2,4}^{-1}$, $mA_{1,1}^{-1}A_{1,3}^{-1}A_{2,2}^{-1}A_{2,4}^{-1}A_{3,3}^{-1}$ and \\ $2mA_{1,1}^{-1}A_{1,3}^{-1}A_{3,1}^{-1}A_{3,3}^{-1}A_{2,2}^{-1}A_{2,4}^{-1}$.
\[
\begin{array}{l}
2mA_{1,1}^{-1}A_{1,3}^{-1}A_{3,1}^{-1}A_{3,3}^{-1}A_{2,2}^{-1}A_{2,4}^{-1}\leq
\varphi_{2}(2mA_{1,1}^{-1}A_{1,3}^{-1}A_{3,1}^{-1}A_{3,3}^{-1})_{\leq4},\\
2mA_{1,3}^{-1}A_{2,4}^{-1}+2mA_{1,1}^{-1}A_{1,3}^{-1}A_{2,4}^{-1}+mA_{1,1}^{-1}A_{1,3}^{-1}A_{2,2}^{-1}A_{2,4}^{-1}
\leq\varphi_{\{1,2\}}(m)_{\leq4},\\
mA_{3,3}^{-1}(A_{2,4}^{-1}+4A_{1,3}^{-1}A_{2,4}^{-1}+4A_{1,1}^{-1}A_{1,3}^{-1}A_{2,4}^{-1}+
A_{1,1}^{-1}A_{1,3}^{-1}A_{2,2}^{-1}A_{2,4}^{-1})
\leq\varphi_{\{1,2\}}(mA_{3,3}^{-1})_{\leq4}.
\end{array}
\]

\medskip
(15) $\a=\a_1+\a_2+\a_3+\a_4$, $Y^{\a}=Y_{1,0}Y_{2,3}Y_{2,5}$. If $\beta\in C_\a\backslash\{-\a_0, -\a_5, \a_1, \a_1+\a_2, \a_1+\a_2+\a_5, \a_2+\a_3+\a_4+\a_5, \a_1+2\a_2+\a_3+\a_5, \a_1+\a_2+\a_3+\a_4+\a_5, \a_0+2\a_1+2\a_2+\a_3+\a_5, \a_1+2\a_2+\a_3+\a_4+\a_5,
\a_1+2\a_2+2\a_3+\a_4+\a_5, \a_0+2\a_1+2\a_2+2\a_3+\a_4+\a_5,\ a_0+2\a_1+2\a_2+\a_3+\a_4+\a_5, \a_0+\a_1+\a_2+\a_3+\a_4+\a_5, \a_0+\a_1+2\a_2+2\a_3+\a_4+\a_5, \a_0+2\a_1+3\a_2+2\a_3+\a_4+\a_5, \a_0+2\a_1+3\a_2+2\a_3+\a_4+2\a_5\}$, then $\chi_q(S(\a)\otimes S(\beta))_{\leq4}$ contains a unique dominant monomial and $S(\a)\otimes S(\beta)$ is simple.

For $\beta=-\a_0,-\a_5,\a_1, \a_1+\a_2,\a_1+\a_2+\a_5$, we have done.

If $\beta=\a_2+\a_3+\a_4+\a_5$, $Y^{\beta}=Y_{3,2}$, denote $m=Y^{\a}Y^{\beta}$. The only dominant monomials in $\chi_q(S(\a)\otimes S(\beta))_{\leq4}$ are $m$ and $mA_{3,3}^{-1}A_{2,4}^{-1}$.
\[
\begin{array}{cc}
  \varphi_{3}(m)_{\leq4}=m(1+A_{3,3}^{-1}), &
  \varphi_{2}(mA_{3,3}^{-1})_{\leq4}=mA_{3,3}^{-1}(1+A_{2,4}^{-1}).
\end{array}
\]

If $\beta=\a_1+2\a_2+\a_3+\a_5$, $Y^{\beta}=Y_{1,0}Y_{1,2}Y_{2,5}Y_{3,0}Y_{3,2}$, denote $m=Y^{\a}Y^{\beta}$. The only dominant monomials in $\chi_q(S(\a)\otimes S(\beta))_{\leq4}$ are $m$, $mA_{1,1}^{-1}$,
$mA_{3,3}^{-1}A_{2,4}^{-1}A_{1,3}^{-1}$ and $mA_{1,1}^{-1}A_{1,3}^{-1}A_{3,1}^{-1}A_{3,3}^{-1}A_{2,2}^{-1}A_{2,4}^{-1}$.
\[
\begin{array}{ccc}
  mA_{1,1}^{-1}\leq\varphi_1(m)_{\leq4}, &
  mA_{3,3}^{-1}A_{2,4}^{-1}A_{1,3}^{-1}\leq\varphi_2(mA_{3,3}^{-1}A_{1,3}^{-1})_{\leq4}, &
  mA_{1,1}^{-1}A_{1,3}^{-1}A_{3,1}^{-1}A_{3,3}^{-1}A_{2,2}^{-1}A_{2,4}^{-1}\leq
  \varphi_{\{1,2\}}(mA_{3,3}^{-1}A_{3,1}^{-1})_{\leq4}.
\end{array}
\]

If $\beta=\a_1+\a_2+\a_3+\a_4+\a_5$, $Y^{\beta}=Y_{1,0}Y_{2,3}$, denote $m=Y^{\a}Y^{\beta}$. The only dominant monomials in $\chi_q(S(\a)\otimes S(\beta))_{\leq4}$ are $m$, $mA_{2,4}^{-1}$ and $mA_{1,1}^{-1}A_{2,2}^{-1}A_{2,4}^{-1}$.
\[
\chi_q(S(\a)\otimes S(\beta))_{\leq4}= \varphi_{\{1,2\}}(m)_{\leq4}\leq\chi_q(L(m)).
\]

If $\beta=\a_0+2\a_1+2\a_2+\a_3+\a_5$, $Y^{\beta}=Y_{1,0}Y_{1,2}Y_{2,3}Y_{2,5}Y_{3,0}$, denote $m=Y^{\a}Y^{\beta}$. The only dominant monomials in $\chi_q(S(\a)\otimes S(\beta))_{\leq4}$ are $m$, $mA_{1,1}^{-1}$,
$mA_{2,4}^{-1}A_{1,3}^{-1}$ and $mA_{1,1}^{-1}A_{1,3}^{-1}A_{3,1}^{-1}A_{2,2}^{-1}A_{2,4}^{-1}$.
\[
\begin{array}{ccc}
  mA_{1,1}^{-1}\leq\varphi_1(m)_{\leq4}, &
  mA_{2,4}^{-1}A_{1,3}^{-1}\leq\varphi_2(mA_{1,3}^{-1})_{\leq4}, &
  mA_{1,1}^{-1}A_{1,3}^{-1}A_{3,1}^{-1}A_{2,2}^{-1}A_{2,4}^{-1}\leq
  \varphi_{\{1,2\}}(mA_{3,1}^{-1})_{\leq4}.
\end{array}
\]

If $\beta=\a_1+2\a_2+\a_3+\a_4+\a_5$, $Y^{\beta}=Y_{1,0}Y_{1,2}Y_{2,5}Y_{3,2}$, denote $m=Y^{\a}Y^{\beta}$. The only dominant monomials in $\chi_q(S(\a)\otimes S(\beta))_{\leq4}$ are $m$, $mA_{1,1}^{-1}$ and
$mA_{3,3}^{-1}A_{2,4}^{-1}A_{1,3}^{-1}$.
\[
\begin{array}{cc}
  mA_{1,1}^{-1}\leq\varphi_1(m)_{\leq4}, &
  mA_{3,3}^{-1}A_{2,4}^{-1}A_{1,3}^{-1}\leq\varphi_2(mA_{1,3}^{-1}A_{3,3}^{-1})_{\leq4}.
\end{array}
\]

If $\beta=\a_1+2\a_2+2\a_3+\a_4+\a_5$, $Y^{\beta}=Y_{1,0}Y_{2,3}Y_{2,5}Y_{3,0}Y_{3,2}$, denote $m=Y^{\a}Y^{\beta}$. The only dominant monomials in $\chi_q(S(\a)\otimes S(\beta))_{\leq4}$ are $m$, $mA_{3,3}^{-1}A_{2,4}^{-1}$ and
$mA_{1,1}^{-1}A_{2,2}^{-1}A_{2,4}^{-1}A_{3,1}^{-1}A_{3,3}^{-1}$.
\[
\begin{array}{cc}
  mA_{3,3}^{-1}A_{2,4}^{-1}\leq\varphi_2(mA_{3,3}^{-1})_{\leq4}, &
  mA_{1,1}^{-1}A_{2,2}^{-1}A_{2,4}^{-1}A_{3,1}^{-1}A_{3,3}^{-1}\leq\varphi_2(mA_{3,1}^{-1}A_{3,3}^{-1})_{\leq4}.
\end{array}
\]

If $\beta=\a_0+2\a_1+2\a_2+\a_3+\a_4+\a_5$, $Y^{\beta}=Y_{1,0}Y_{1,2}Y_{2,3}Y_{2,5}$, denote $m=Y^{\a}Y^{\beta}$. The only dominant monomials in $\chi_q(S(\a)\otimes S(\beta))_{\leq4}$ are $m$,
$mA_{1,1}^{-1}$ and $mA_{1,3}^{-1}A_{2,4}^{-1}$.
\[
\begin{array}{cc}
  mA_{1,1}^{-1}\leq\varphi_1(m)_{\leq4}, &
  mA_{1,3}^{-1}A_{2,4}^{-1}\leq\varphi_2(mA_{1,3}^{-1})_{\leq4}.
\end{array}
\]

If $\beta=\a_0+2\a_1+2\a_2+2\a_3+\a_4+\a_5$, $Y^{\beta}=Y_{1,0}Y_{2,3}^{2}Y_{2,5}Y_{3,0}$, denote $m=Y^{\a}Y^{\beta}$. The only dominant monomials in $\chi_q(S(\a)\otimes S(\beta))_{\leq4}$ are $m$,
$mA_{2,4}^{-1}$ and $mA_{1,1}^{-1}A_{3,1}^{-1}A_{2,2}^{-1}A_{2,4}^{-1}$.
\[
\begin{array}{cc}
  mA_{2,4}^{-1}\leq\varphi_2(m)_{\leq4}, &
  mA_{1,1}^{-1}A_{3,1}^{-1}A_{2,2}^{-1}A_{2,4}^{-1}\leq\varphi_{\{1,2\}}(mA_{3,1}^{-1})_{\leq4}.
\end{array}
\]

If $\beta=\a_0+\a_1+\a_2+\a_3+\a_4+\a_5$, $Y^{\beta}=Y_{2,3}$, denote $m=Y^{\a}Y^{\beta}$. The only dominant monomials in $\chi_q(S(\a)\otimes S(\beta))_{\leq4}$ are $m$ and $mA_{2,4}^{-1}$.
\[
\begin{array}{c}
  mA_{2,4}^{-1}\leq\varphi_2(m)_{\leq4}=m(1+A_{2,4}^{-1}).
\end{array}
\]

If $\beta=\a_0+\a_1+2\a_2+2\a_3+\a_4+\a_5$, $Y^{\beta}=Y_{2,3}Y_{2,5}Y_{3,0}Y_{3,2}$, denote $m=Y^{\a}Y^{\beta}$. The only dominant monomials in $\chi_q(S(\a)\otimes S(\beta))_{\leq4}$ are $m$ and $mA_{3,3}^{-1}A_{2,4}^{-1}$.
\[
\begin{array}{cc}
  \varphi_3(m)_{\leq4}=m(1+A_{3,3}^{-1}+A_{3,3}^{-1}A_{3,1}^{-1}),&
  \varphi_2(mA_{3,3}^{-1})_{\leq4}=mA_{3,3}^{-1}(1+A_{2,4}^{-1}).
\end{array}
\]

If $\beta=\a_0+2\a_1+3\a_2+2\a_3+\a_4+\a_5$, $Y^{\beta}=Y_{1,0}Y_{1,2}Y_{2,3}Y_{2,5}^{2}Y_{3,0}Y_{3,2}$, denote $m=Y^{\a}Y^{\beta}$. The only dominant monomials in $\chi_q(S(\a)\otimes S(\beta))_{\leq4}$ are $m$, $mA_{1,1}^{-1}$ and $mA_{1,3}^{-1}A_{3,3}^{-1}A_{2,4}^{-1}$.
\[
\begin{array}{cc}
  \varphi_1(m)_{\leq4}=m(1+A_{1,1}^{-1})(1+A_{1,3}^{-1}+A_{1,3}^{-1}A_{1,1}^{-1}),&
  \varphi_2(mA_{1,3}^{-1}A_{3,3}^{-1})_{\leq4}=mmA_{1,3}^{-1}A_{3,3}^{-1}(1+A_{2,4}^{-1}).
\end{array}
\]

If $\beta=\a_0+2\a_1+3\a_2+2\a_3+\a_4+2\a_5$, $Y^{\beta}=Y_{1,0}Y_{1,2}Y_{2,3}Y_{2,5}Y_{3,0}Y_{3,2}$, denote $m=Y^{\a}Y^{\beta}$. The only dominant monomials in $\chi_q(S(\a)\otimes S(\beta))_{\leq4}$ are $m$, $mA_{1,1}^{-1}$, $mA_{1,3}^{-1}A_{2,4}^{-1}$, $mA_{3,3}^{-1}A_{2,4}^{-1}$,  $2mA_{1,3}^{-1}A_{3,3}^{-1}A_{2,4}^{-1}$, $mA_{1,1}^{-1}A_{3,3}^{-1}A_{2,4}^{-1}$, $mA_{1,3}^{-1}A_{3,3}^{-1}A_{2,4}^{-2}$, $mA_{1,1}^{-1}A_{1,3}^{-1}A_{2,2}^{-1}A_{2,4}^{-1}A_{3,1}^{-1}A_{3,3}^{-1}$ and $mA_{1,1}^{-1}A_{1,3}^{-1}A_{2,2}^{-1}A_{2,4}^{-2}A_{3,1}^{-1}A_{3,3}^{-1}$.
\[
\begin{array}{l}
\bigskip
mA_{3,3}^{-1}(A_{2,4}^{-1}+2A_{1,3}^{-1}A_{2,4}^{-1}+A_{1,1}^{-1}A_{2,4}^{-1}+A_{1,3}^{-1}A_{2,4}^{-2})
\leq\varphi_{\{1,2\}}(mA_{3,3}^{-1})_{\leq4},\\
mA_{3,1}^{-1}A_{3,3}^{-1}(A_{1,1}^{-1}A_{1,3}^{-1}A_{2,2}^{-1}A_{2,4}^{-1}+A_{1,1}^{-1}A_{1,3}^{-1}A_{2,2}^{-1}A_{2,4}^{-2})
\leq\varphi_{\{1,2\}}(mA_{3,1}^{-1}A_{3,3}^{-1})_{\leq4}.
\end{array}
\]
\[
\begin{array}{cc}
  \varphi_1(m)_{\leq4}=m(1+A_{1,1}^{-1})(1+A_{1,3}^{-1}+A_{1,3}^{-1}A_{1,1}^{-1}), &
  \varphi_2(mA_{1,3}^{-1})_{\leq4}=mA_{1,3}^{-1}(1+A_{2,4}^{-1}).
\end{array}
\]

\medskip
(16) $\a=\a_0+\a_1+\a_2+\a_5$, $Y^{\a}=Y_{1,2}$. If $\beta\in C_\a\backslash\{-\a_3, -\a_4, \a_0, \a_1+\a_2,
\a_0+\a_1+\a_2, \a_0+\a_1+\a_2+\a_3, \a_1+2\a_2+\a_3+\a_5, \a_0+2\a_1+2\a_2+\a_3+\a_5, \a_0+\a_1+2\a_2+\a_3+\a_5, \a_1+2\a_2+\a_3+\a_4+\a_5, \a_0+\a_1+\a_2+\a_3+\a_4, \a_0+2\a_1+2\a_2+\a_3+\a_4+\a_5, \a_0+\a_1+2\a_2+2\a_3+\a_4+\a_5,  \a_0+\a_1+2\a_2+\a_3+\a_4+\a_5, \a_0+2\a_1+3\a_2+2\a_3+\a_4+\a_5, \a_0+2\a_1+3\a_2+2\a_3+\a_4+2\a_5\}$, then $\chi_q(S(\a)\otimes S(\beta))_{\leq4}$ contains a unique dominant monomial and $S(\a)\otimes S(\beta)$ is simple.

For $\beta=-\a_3, -\a_4, \a_0, \a_1+\a_2,\a_0+\a_1+\a_2,\a_0+\a_1+\a_2+\a_3$, we have done.

If $\beta=\a_1+2\a_2+\a_3+\a_5$, $Y^{\beta}=Y_{1,0}Y_{1,2}Y_{3,0}Y_{3,2}Y_{2,5}$, denote $m=Y^{\a}Y^{\beta}$. The only dominant monomials in $\chi_q(S(\a)\otimes S(\beta))_{\leq4}$ are $m$, $mA_{1,3}^{-1}A_{2,4}^{-1}$,  $2mA_{1,3}^{-1}A_{3,3}^{-1}A_{2,4}^{-1}$, $mA_{1,1}^{-1}A_{1,3}^{-1}A_{3,3}^{-1}A_{2,4}^{-1}$ and $mA_{1,1}^{-1}A_{1,3}^{-1}A_{3,1}^{-1}A_{3,3}^{-1}A_{2,2}^{-1}A_{2,4}^{-1}$.
\[
\begin{array}{ll}
\medskip
  mA_{1,3}^{-1}A_{2,4}^{-1}\leq\varphi_{\{1,2\}}(m)_{\leq4},&
  2mA_{1,3}^{-1}A_{3,3}^{-1}A_{2,4}^{-1}\leq\varphi_2(2mA_{1,3}^{-1}A_{3,3}^{-1})_{\leq4},\\
  mA_{1,1}^{-1}A_{1,3}^{-1}A_{3,3}^{-1}A_{2,4}^{-1}\leq\varphi_2(mA_{1,1}^{-1}A_{1,3}^{-1}A_{3,3}^{-1})_{\leq4},&
  mA_{1,1}^{-1}A_{1,3}^{-1}A_{3,1}^{-1}A_{3,3}^{-1}A_{2,2}^{-1}A_{2,4}^{-1}\leq
  \varphi_2(mA_{1,1}^{-1}A_{1,3}^{-1}A_{3,1}^{-1}A_{3,3}^{-1})_{\leq4}.
\end{array}
\]

If $\beta=\a_0+2\a_1+2\a_2+\a_3+\a_5$, $Y^{\beta}=Y_{1,0}Y_{1,2}Y_{2,3}Y_{2,5}Y_{3,0}$, denote $m=Y^{\a}Y^{\beta}$. The only dominant monomials in $\chi_q(S(\a)\otimes S(\beta))_{\leq4}$ are $m$, $2mA_{1,3}^{-1}A_{2,4}^{-1}$,  $m A_{1,1}^{-1}A_{1,3}^{-1}A_{2,4}^{-1}$ and $mA_{1,1}^{-1}A_{1,3}^{-1}A_{2,2}^{-1}A_{2,4}^{-1}A_{3,1}^{-1}$.
\[
\begin{array}{cc}
  2mA_{1,3}^{-1}A_{2,4}^{-1}\leq\varphi_{2}(2mA_{1,3}^{-1})_{\leq4},&
  m A_{1,1}^{-1}A_{1,3}^{-1}(A_{2,4}^{-1}+A_{2,2}^{-1}A_{2,4}^{-1}A_{3,1}^{-1})\leq
  \varphi_{\{1,2\}}(mA_{1,1}^{-1}A_{1,3}^{-1})_{\leq4}.
\end{array}
\]

If $\beta=\a_0+\a_1+2\a_2+\a_3+\a_5$, $Y^{\beta}=Y_{1,2}Y_{2,5}Y_{3,0}Y_{3,2}$, denote $m=Y^{\a}Y^{\beta}$. The only dominant monomials in $\chi_q(S(\a)\otimes S(\beta))_{\leq4}$ are $m$, $mA_{1,3}^{-1}A_{2,4}^{-1}$ and   $2m A_{3,3}^{-1}A_{1,3}^{-1}A_{2,4}^{-1}$.
\[
\begin{array}{cc}
  mA_{1,3}^{-1}A_{2,4}^{-1}\leq\varphi_{\{1,2\}}(m)_{\leq4},&
  2m A_{3,3}^{-1}A_{1,3}^{-1}A_{2,4}^{-1}\leq
  \varphi_{2}(2mA_{1,3}^{-1}A_{3,3}^{-1})_{\leq4}.
\end{array}
\]

If $\beta=\a_1+2\a_2+\a_3+\a_4+\a_5$, $Y^{\beta}=Y_{1,0}Y_{1,2}Y_{2,5}Y_{3,2}$, denote $m=Y^{\a}Y^{\beta}$. The only dominant monomials in $\chi_q(S(\a)\otimes S(\beta))_{\leq4}$ are $m$, $mA_{1,3}^{-1}A_{2,4}^{-1}$,
$2mA_{3,3}^{-1}A_{1,3}^{-1}A_{2,4}^{-1}$ and $mA_{1,1}^{-1}A_{1,3}^{-1}A_{2,4}^{-1}A_{3,3}^{-1}$.
\[
\begin{array}{lll}
mA_{1,3}^{-1}A_{2,4}^{-1}\leq\varphi_{\{1,2\}}(m),&
2m A_{3,3}^{-1}A_{1,3}^{-1}A_{2,4}^{-1}\leq\varphi_{2}(2mA_{1,3}^{-1}A_{3,3}^{-1}),&
mA_{1,1}^{-1}A_{1,3}^{-1}A_{2,4}^{-1}A_{3,3}^{-1}\leq\varphi_{2}(mA_{1,1}^{-1}A_{1,3}^{-1}A_{3,3}^{-1}).
\end{array}
\]

If $\beta=\a_0+\a_1+\a_2+\a_3+\a_4$, $Y^{\beta}=Y_{2,3}Y_{2,5}$, denote $m=Y^{\a}Y^{\beta}$. The only dominant monomials in $\chi_q(S(\a)\otimes S(\beta))_{\leq4}$ are $m$ and $mA_{1,3}^{-1}A_{2,4}^{-1}$.
\[
mA_{1,3}^{-1}A_{2,4}^{-1} \leq\varphi_{2}(mA_{1,3}^{-1})_{\leq4}.
\]

If $\beta=\a_0+2\a_1+2\a_2+\a_3+\a_4+\a_5$, $Y^{\beta}=Y_{1,0}Y_{1,2}Y_{2,3}Y_{2,5}$, denote $m=Y^{\a}Y^{\beta}$. The only dominant monomials in $\chi_q(S(\a)\otimes S(\beta))_{\leq4}$ are $m$, $2mA_{1,3}^{-1}A_{2,4}^{-1}$ and $mA_{1,1}^{-1}A_{1,3}^{-1}A_{2,4}^{-1}$.
\[
\begin{array}{ll}
2mA_{1,3}^{-1}A_{2,4}^{-1}\leq\varphi_{2}(2mA_{1,3}^{-1}),&
mA_{1,1}^{-1}A_{1,3}^{-1}A_{2,4}^{-1}\leq\varphi_{2}(mA_{1,1}^{-1}A_{1,3}^{-1}).
\end{array}
\]

If $\beta=\a_0+\a_1+2\a_2+2\a_3+\a_4+\a_5$, $Y^{\beta}=Y_{2,3}Y_{2,5}Y_{3,0}Y_{3,2}$, denote $m=Y^{\a}Y^{\beta}$. The only dominant monomials in $\chi_q(S(\a)\otimes S(\beta))_{\leq4}$ are $m$, $mA_{1,3}^{-1}A_{2,4}^{-1}$, $mA_{3,3}^{-1}A_{2,4}^{-1}$ and $2mA_{1,3}^{-1}A_{3,3}^{-1}A_{2,4}^{-1}$.
\[
\begin{array}{lll}
mA_{1,3}^{-1}A_{2,4}^{-1}\leq\varphi_{2}(mA_{1,3}^{-1}), &
mA_{3,3}^{-1}A_{2,4}^{-1}\leq\varphi_{2}(mA_{3,3}^{-1}),&
2mA_{1,3}^{-1}A_{3,3}^{-1}A_{2,4}^{-1}\leq\varphi_{\{1,2\}}(mA_{1,3}^{-1}).
\end{array}
\]

If $\beta=\a_0+2\a_1+3\a_2+2\a_3+\a_4+2\a_5$, $Y^{\beta}=Y_{1,0}Y_{1,2}Y_{2,3}Y_{2,5}Y_{3,0}Y_{3,2}$, denote $m=Y^{\a}Y^{\beta}$. The only dominant monomials in $\chi_q(S(\a)\otimes S(\beta))_{\leq4}$ are $m$, $2mA_{1,3}^{-1}A_{2,4}^{-1}$, $mA_{3,3}^{-1}A_{2,4}^{-1}$, $4mA_{1,3}^{-1}A_{3,3}^{-1}A_{2,4}^{-1}$, $mA_{1,1}^{-1}A_{1,3}^{-1}A_{2,4}^{-1}$, $2mA_{1,1}^{-1}A_{1,3}^{-1}A_{3,3}^{-1}A_{2,4}^{-1}$ and $mA_{1,1}^{-1}A_{1,3}^{-1}A_{3,1}^{-1}A_{3,3}^{-1}A_{2,2}^{-1}A_{2,4}^{-1}$.
\[
\begin{array}{ll}
\medskip
mA_{1,1}^{-1}A_{1,3}^{-1}(A_{3,1}^{-1}A_{3,3}^{-1}A_{2,2}^{-1}A_{2,4}^{-1}+A_{2,4}^{-1})\leq
  \varphi_{\{2,3\}}(mA_{1,1}^{-1}A_{1,3}^{-1}),&
2mA_{1,3}^{-1}A_{2,4}^{-1}\leq\varphi_{2}(2mA_{1,3}^{-1}), \\
mA_{3,3}^{-1}(4A_{1,3}^{-1}A_{2,4}^{-1}+2A_{1,1}^{-1}A_{1,3}^{-1}A_{2,4}^{-1}+A_{2,4}^{-1})\leq
  \varphi_{\{1,2\}}(mA_{3,3}^{-1}).
\end{array}
\]

If $\beta=\a_0+\a_1+2\a_2+\a_3+\a_4+\a_5$, $Y^{\beta}=Y_{1,2}Y_{2,5}Y_{3,2}$, denote $m=Y^{\a}Y^{\beta}$. The only dominant monomials in $\chi_q(S(\a)\otimes S(\beta))_{\leq4}$ are $m$, $mA_{1,3}^{-1}A_{2,4}^{-1}$ and $2mA_{1,3}^{-1}A_{3,3}^{-1}A_{2,4}^{-1}$.
\[
\begin{array}{ll}
mA_{1,3}^{-1}A_{2,4}^{-1}\leq\varphi_{\{1,2\}}(m), &
2mA_{1,3}^{-1}A_{3,3}^{-1}A_{2,4}^{-1}\leq\varphi_{2}(2mA_{1,3}^{-1}A_{3,3}^{-1}).
\end{array}
\]

If $\beta=\a_0+2\a_1+3\a_2+2\a_3+\a_4+\a_5$, $Y^{\beta}=Y_{1,0}Y_{1,2}Y_{2,3}Y_{2,5}^{2}Y_{3,0}Y_{3,2}$, denote $m=Y^{\a}Y^{\beta}$. The only dominant monomials in $\chi_q(S(\a)\otimes S(\beta))_{\leq4}$ are $m$, $mA_{1,3}^{-1}A_{2,4}^{-1}$, $2mA_{1,3}^{-1}A_{3,3}^{-1}A_{2,4}^{-1}$, $mA_{1,1}^{-1}A_{1,3}^{-1}A_{3,3}^{-1}A_{2,4}^{-1}$ and $mA_{1,3}^{-2}A_{3,3}^{-1}A_{2,4}^{-2}$.
\[
\begin{array}{ll}
mA_{1,3}^{-1}A_{2,4}^{-1}\leq\varphi_{\{1,2\}}(m), &
mA_{3,3}^{-1}(2A_{1,3}^{-1}A_{2,4}^{-1}+A_{1,1}^{-1}A_{1,3}^{-1}A_{2,4}^{-1}+A_{1,3}^{-2}A_{2,4}^{-2})\leq
  \varphi_{\{1,2\}}(mA_{3,3}^{-1}).
\end{array}
\]

\medskip
(17) $\a=\a_1+\a_2+\a_3+\a_5$, $Y^{\a}=Y_{1,0}Y_{2,3}Y_{3,0}$. In this case, $\chi_q(L(Y^{\a}))$ has two dominant monomials, so we need to consider every $\beta\in C_\a.$

For $\beta=-\a_0, -\a_4, \a_1, \a_3, \a_5, \a_0+\a_1,  \a_3+\a_4,  \a_1+\a_2+\a_3, \a_2+\a_3+\a_5, \a_1+\a_2+\a_5$, we have done.

Up to the symmetries $1\leftrightarrow 3$, $0\leftrightarrow 4$ of $\widetilde{I}$, we reduces to the cases:
$\beta\in \{\a_1+\a_2+\a_3+\a_5,\a_0+\a_1+\a_2+\a_3+\a_5\}$.

If $\beta=\a_1+\a_2+\a_3+\a_5$, $Y^{\beta}=Y^{\a}$, denote $m=Y^{\a}Y^{\beta}$. The only dominant monomials in $\chi_q(S(\a)\otimes S(\beta))_{\leq4}$ are $m$, $2mA_{1,1}^{-1}A_{3,1}^{-1}A_{2,2}^{-1}$ and $mA_{1,1}^{-2}A_{3,1}^{-2}A_{2,2}^{-2}$.
\[
\begin{array}{ll}
mA_{1,1}^{-2}A_{3,1}^{-2}A_{2,2}^{-2}\leq\varphi_{2}(mA_{1,1}^{-2}A_{3,1}^{-2}),&
2mA_{1,1}^{-1}A_{3,1}^{-1}A_{2,2}^{-1}\leq\varphi_{\{2,3\}}(2mA_{1,1}^{-1}).
\end{array}
\]

If $\beta=\a_0+\a_1+\a_2+\a_3+\a_5$, $Y^{\beta}=Y_{2,3}Y_{3,0}$, denote $m=Y^{\a}Y^{\beta}$. The only dominant monomials in $\chi_q(S(\a)\otimes S(\beta))_{\leq4}$ are $m$ and $mA_{1,1}^{-1}A_{3,1}^{-1}A_{2,2}^{-1}$.
\[
mA_{1,1}^{-1}A_{3,1}^{-1}A_{2,2}^{-1}\leq\varphi_{\{2,3\}}(mA_{1,1}^{-1}).
\]

\medskip
(18) $\a=\a_1+2\a_2+\a_3+\a_5$, $Y^{\a}=Y_{1,0}Y_{1,2}Y_{2,5}Y_{3,0}Y_{3,2}$. In this case, $\chi_q(L(Y^{\a}))$ has $3$ dominant monomials, so we need to consider every $\beta\in C_\a.$

For $\beta=-\a_0, -\a_4, \a_1+\a_2, \a_2+\a_5, \a_2+\a_3, \a_1+\a_2+\a_3, \a_1+\a_2+\a_5, \a_2+\a_3+\a_5,
\a_1+\a_2+\a_3+\a_4, \a_0+\a_1+\a_2+\a_3, \a_0+\a_1+\a_2+\a_5,\a_2+\a_3+\a_4+\a_5$, we have done.

Up to the symmetries $1\leftrightarrow 3$, $0\leftrightarrow 4$ of $\widetilde{I}$, we reduce to the cases:
$\beta\in \{\a_1+2\a_2+\a_3+\a_5, \a_0+2\a_1+2\a_2+\a_3+\a_5, \a_0+\a_1+2\a_2+\a_3+\a_5, \a_0+2\a_1+2\a_2+\a_3+\a_4+\a_5, \a_0+2\a_1+3\a_2+2\a_3+\a_4+\a_5, \a_0+2\a_1+3\a_2+2\a_3+\a_4+2\a_5\}$.

If $\beta=\a_1+2\a_2+\a_3+\a_5$, $Y^{\beta}=Y^{\a}$, denote $m=Y^{\a}Y^{\beta}$. The only dominant monomials in $\chi_q(S(\a)\otimes S(\beta))_{\leq4}$ are
\smallskip
$m,\  2mA_{1,3}^{-1}A_{3,3}^{-1}A_{2,4}^{-1},\ 2mA_{1,1}^{-1}A_{1,3}^{-1}A_{3,1}^{-1}A_{3,3}^{-1}A_{2,4}^{-1},\ 2mA_{1,1}^{-1}A_{1,3}^{-1}A_{3,3}^{-1}A_{2,4}^{-1},\\
2mA_{1,3}^{-1}A_{3,1}^{-1}A_{3,3}^{-1}A_{2,4}^{-1},\
mA_{1,3}^{-2}A_{3,3}^{-2}A_{2,4}^{-2}, 2mA_{1,1}^{-1}A_{1,3}^{-1}A_{3,1}^{-1}A_{3,3}^{-1}A_{2,2}^{-1}A_{2,4}^{-1}, \ 2mA_{1,1}^{-1}A_{1,3}^{-2}A_{3,1}^{-1}A_{3,3}^{-2}A_{2,2}^{-1}A_{2,4}^{-2}$ and $mA_{1,1}^{-2}A_{1,3}^{-2}A_{3,1}^{-2}A_{3,3}^{-2}A_{2,2}^{-2}A_{2,4}^{-2}$.
\[
\begin{array}{ll}
\bigskip
mA_{1,3}^{-2}A_{3,3}^{-2}A_{2,4}^{-2}\leq\varphi_2(mA_{1,3}^{-2}A_{3,3}^{-2}),&
mA_{1,1}^{-2}A_{1,3}^{-2}A_{3,1}^{-2}A_{3,3}^{-2}A_{2,2}^{-2}A_{2,4}^{-2}\leq
  \varphi_2(A_{1,1}^{-2}A_{1,3}^{-2}A_{3,1}^{-2}A_{3,3}^{-2}),\\
2mA_{1,1}^{-1}A_{1,3}^{-2}A_{3,1}^{-1}A_{3,3}^{-2}A_{2,2}^{-1}A_{2,4}^{-2}\leq
\varphi_{\{2,3\}}(2mA_{1,1}^{-1}A_{1,3}^{-2}),&
2mA_{1,3}^{-1}(A_{3,3}^{-1}A_{2,4}^{-1}+A_{3,1}^{-1}A_{3,3}^{-1}A_{2,4}^{-1})\leq
\varphi_{\{2,3\}}(2mA_{1,3}^{-1}),
\end{array}
\]

$2mA_{1,1}^{-1}A_{1,3}^{-1}(A_{3,3}^{-1}A_{2,4}^{-1}+A_{3,1}^{-1}A_{3,3}^{-1}A_{2,4}^{-1}+
A_{3,1}^{-1}A_{3,3}^{-1}A_{2,2}^{-1}A_{2,4}^{-1})\leq\varphi_{\{2,3\}}(2mA_{1,1}^{-1}A_{1,3}^{-1}).$

If $\beta=\a_0+2\a_1+2\a_2+\a_3+\a_5$, $Y^{\beta}=Y_{1,0}Y_{1,2}Y_{2,3}Y_{2,5}Y_{3,0}$, denote $m=Y^{\a}Y^{\beta}$. The only dominant monomials in $\chi_q(S(\a)\otimes S(\beta))_{\leq4}$ are
$m,\ mA_{3,1}^{-1},\  mA_{1,3}^{-1}A_{2,4}^{-1},\ mA_{1,1}^{-1}A_{1,3}^{-1}A_{2,4}^{-1},\ mA_{1,3}^{-1}A_{2,4}^{-1}A_{3,1}^{-1},\\
mA_{1,1}^{-1}A_{1,3}^{-1}A_{2,4}^{-1}A_{3,1}^{-1},\
 2mA_{1,3}^{-1}A_{2,4}^{-1}A_{3,3}^{-1},\ mA_{1,3}^{-2}A_{2,4}^{-2}A_{3,3}^{-1},\ 2mA_{1,1}^{-1}A_{1,3}^{-1}A_{2,4}^{-1}A_{3,3}^{-1},\ mA_{1,1}^{-1}A_{1,3}^{-1}A_{2,2}^{-1}A_{2,4}^{-1}A_{3,1}^{-1},\\ 2mA_{1,1}^{-1}A_{1,3}^{-1}A_{2,2}^{-1}A_{2,4}^{-1}A_{3,1}^{-1}A_{3,3}^{-1},  2mA_{1,1}^{-1}A_{1,3}^{-2}A_{2,2}^{-1}A_{2,4}^{-2}A_{3,1}^{-1}A_{3,3}^{-1}$ and
$mA_{1,1}^{-2}A_{1,3}^{-2}A_{2,2}^{-2}A_{2,4}^{-2}A_{3,1}^{-2}A_{3,3}^{-1}$.

We have
\[
\begin{array}{l}
\medskip
  (mA_{1,3}^{-1}A_{2,4}^{-1}+mA_{1,1}^{-1}A_{1,3}^{-1}A_{2,4}^{-1})\leq\varphi_{\{1,2\}}(m)_{\leq4},\\
\medskip
  m A_{3,1}^{-2}A_{3,3}^{-1}A_{1,1}^{-2}A_{1,3}^{-2}A_{2,2}^{-2}A_{2,4}^{-2}\leq
   \varphi_{\{1,2\}}(m A_{3,1}^{-2}A_{3,3}^{-1})_{\leq4},\\
\medskip
  mA_{3,3}^{-1}(2A_{1,3}^{-1}A_{2,4}^{-1}+A_{1,1}^{-1}A_{1,3}^{-1}A_{2,4}^{-1}+A_{1,3}^{-2}A_{2,4}^{-2})
     \leq\varphi_{\{1,2\}}(mA_{3,3}^{-1})_{\leq4},\\
\medskip
  2mA_{3,1}^{-1}A_{3,3}^{-1}(A_{1,1}^{-1}A_{1,3}^{-1}A_{2,2}^{-1}A_{2,4}^{-1}+
  A_{1,1}^{-1}A_{1,3}^{-2}A_{2,2}^{-1}A_{2,4}^{-2})\leq\varphi_{\{1,2\}}(2mA_{3,1}^{-1}A_{3,3}^{-1})_{\leq4},\\
  mA_{3,1}^{-1}(1+A_{1,1}^{-1}A_{1,3}^{-1}A_{2,2}^{-1}A_{2,4}^{-1}+A_{1,1}^{-1}A_{1,3}^{-1}A_{2,4}^{-1}+
     A_{1,3}^{-1}A_{2,4}^{-1})\leq\varphi_{\{1,2\}}(mA_{3,1}^{-1})_{\leq4}.
\end{array}
\]

If $\beta=\a_0+\a_1+2\a_2+\a_3+\a_5$, $Y^{\beta}=Y_{1,2}Y_{2,5}Y_{3,0}Y_{3,2}$, denote $m=Y^{\a}Y^{\beta}$. The only dominant monomials in $\chi_q(S(\a)\otimes S(\beta))_{\leq4}$ are
\medskip
$m,\ 2mA_{1,3}^{-1}A_{2,4}^{-1}A_{3,3}^{-1},\  mA_{1,3}^{-2}A_{2,4}^{-2}A_{3,3}^{-2},\ mA_{1,1}^{-1}A_{1,3}^{-1}A_{2,4}^{-1}A_{3,3}^{-1},\\
2m A_{1,3}^{-1}A_{3,1}^{-1}A_{3,3}^{-1}A_{2,4}^{-1},\  mA_{1,1}^{-1}A_{1,3}^{-1}A_{3,1}^{-1}A_{3,3}^{-1}A_{2,4}^{-1},\  mA_{1,1}^{-1}A_{1,3}^{-1}A_{2,2}^{-1}A_{2,4}^{-1}A_{3,1}^{-1}A_{3,3}^{-1}$ and $mA_{1,1}^{-1}A_{1,3}^{-2}A_{2,2}^{-1}A_{2,4}^{-2}A_{3,1}^{-1}A_{3,3}^{-2}$.

We have

\noindent\medskip
  $mA_{1,3}^{-2}A_{2,4}^{-2}A_{3,3}^{-2}\leq\varphi_{\{2,3\}}(mA_{1,3}^{-2})_{\leq4},$ \quad
  $mA_{1,1}^{-1}A_{1,3}^{-2}A_{2,2}^{-1}A_{2,4}^{-2}A_{3,1}^{-1}A_{3,3}^{-2}\leq
  \varphi_{\{2,3\}}(mA_{1,1}^{-1}A_{1,3}^{-2})_{\leq4},$ \\
  \medskip
$2mA_{1,3}^{-1}(A_{2,4}^{-1}A_{3,3}^{-1}+A_{3,1}^{-1}A_{3,3}^{-1}A_{2,4}^{-1})\leq
  \varphi_{\{2,3\}}( 2mA_{1,3}^{-1})_{\leq4},$\\
  $mA_{1,1}^{-1}A_{1,3}^{-1}(A_{2,4}^{-1}A_{3,3}^{-1}+A_{3,1}^{-1}A_{3,3}^{-1}A_{2,4}^{-1}+A_{2,2}^{-1}A_{2,4}^{-1}A_{3,1}^{-1}A_{3,3}^{-1})
  \leq\varphi_{\{2,3\}}(mA_{1,1}^{-1}A_{1,3}^{-1})_{\leq4}.$

If $\beta=\a_0+2\a_1+2\a_2+\a_3+\a_4+\a_5$, $Y^{\beta}=Y_{1,0}Y_{1,2}Y_{2,3}Y_{2,5}$, denote $m=Y^{\a}Y^{\beta}$. The only dominant monomials in $\chi_q(S(\a)\otimes S(\beta))_{\leq4}$ are $m$, $mA_{1,3}^{-1}A_{2,4}^{-1}$,  $2mA_{1,3}^{-1}A_{2,4}^{-1}A_{3,3}^{-1}$, $mA_{1,3}^{-2}A_{2,4}^{-2}A_{3,3}^{-1}$, $mA_{1,1}^{-1}A_{1,3}^{-1}A_{2,2}^{-1}A_{2,4}^{-1}A_{3,1}^{-1}A_{3,3}^{-1}$, $mA_{1,1}^{-1}A_{1,3}^{-1}A_{2,4}^{-1}$,
$2m A_{1,1}^{-1}A_{1,3}^{-1}A_{2,4}^{-1}A_{3,3}^{-1}$ and $mA_{1,1}^{-1}A_{1,3}^{-2}A_{2,2}^{-1}A_{2,4}^{-2}A_{3,1}^{-1}A_{3,3}^{-1}$.

We have
\[
\begin{array}{l}
  (mA_{1,3}^{-1}A_{2,4}^{-1}+mA_{1,1}^{-1}A_{1,3}^{-1}A_{2,4}^{-1})\leq\varphi_{\{1,2\}}(m)_{\leq4},\\
  mA_{3,3}^{-1}(2A_{1,3}^{-1}A_{2,4}^{-1}+2A_{1,1}^{-1}A_{1,3}^{-1}A_{2,4}^{-1}+A_{1,3}^{-2}A_{2,4}^{-2})
     \leq\varphi_{\{1,2\}}(mA_{3,3}^{-1})_{\leq4},\\
  mA_{3,1}^{-1}A_{3,3}^{-1}(A_{1,1}^{-1}A_{1,3}^{-1}A_{2,2}^{-1}A_{2,4}^{-1}+A_{1,1}^{-1}A_{1,3}^{-2}A_{2,2}^{-1}A_{2,4}^{-2})
   \leq\varphi_{\{1,2\}}(mA_{3,1}^{-1}A_{3,3}^{-1})_{\leq4}.
\end{array}
\]

If $\beta=\a_0+2\a_1+3\a_2+2\a_3+\a_4+2\a_5$, $Y^{\beta}=Y_{1,0}Y_{1,2}Y_{2,3}Y_{2,5}Y_{3,0}Y_{3,2}$, denote $m=Y^{\a}Y^{\beta}$. The only dominant monomials in $\chi_q(S(\a)\otimes S(\beta))_{\leq4}$ are
$m,\ mA_{1,3}^{-1}A_{2,4}^{-1},  mA_{3,3}^{-1}A_{2,4}^{-1}, 5mA_{1,3}^{-1}A_{2,4}^{-1}A_{3,3}^{-1}, \\ 2mA_{1,3}^{-2}A_{2,4}^{-2}A_{3,3}^{-1},2mA_{1,3}^{-1}A_{2,4}^{-2}A_{3,3}^{-2}, 3mA_{1,3}^{-2}A_{2,4}^{-1}A_{3,3}^{-2}, mA_{1,3}^{-1}A_{2,4}^{-2}A_{3,3}^{-1},  mA_{2,4}^{-1}A_{3,1}^{-1}A_{3,3}^{-1},m A_{1,1}^{-1}A_{1,3}^{-1}A_{2,4}^{-1},\\
2mA_{1,1}^{-1}A_{1,3}^{-1}A_{2,2}^{-1}A_{2,4}^{-1}A_{3,1}^{-1}A_{3,3}^{-1},
m A_{1,1}^{-1}A_{1,3}^{-1}A_{3,3}^{-1}A_{2,4}^{-2},
 mA_{1,3}^{-1}A_{3,1}^{-1}A_{3,3}^{-1}A_{2,4}^{-2},  5mA_{1,1}^{-1}A_{1,3}^{-1}A_{3,3}^{-1}A_{2,4}^{-1},  5mA_{1,3}^{-1}A_{3,1}^{-1}A_{3,3}^{-1}A_{2,4}^{-1},\\
4mA_{1,1}^{-1}A_{1,3}^{-2}A_{2,2}^{-1}A_{2,4}^{-2}A_{3,1}^{-1}A_{3,3}^{-2},
2mA_{1,1}^{-1}A_{1,3}^{-2}A_{2,2}^{-1}A_{2,4}^{-2}A_{3,1}^{-1}A_{3,3}^{-1},
 2mA_{1,3}^{-2}A_{3,1}^{-1}A_{3,3}^{-1}A_{2,4}^{-2},
 5mA_{1,1}^{-1}A_{1,3}^{-1}A_{2,4}^{-1}A_{3,1}^{-1}A_{3,3}^{-1},\\
mA_{1,1}^{-2}A_{1,3}^{-2}A_{2,2}^{-2}A_{2,4}^{-2}A_{3,1}^{-2}A_{3,3}^{-2},
 2mA_{1,1}^{-1}A_{1,3}^{-1}A_{3,3}^{-2}A_{2,4}^{-2},
 mA_{1,1}^{-1}A_{1,3}^{-1}A_{2,4}^{-2}A_{3,1}^{-1}A_{3,3}^{-1},  2mA_{1,1}^{-1}A_{1,3}^{-1}A_{2,2}^{-1}A_{2,4}^{-2}A_{3,1}^{-1}A_{3,3}^{-2}$ and
$mA_{1,1}^{-1}A_{1,3}^{-1}A_{2,2}^{-1}A_{2,4}^{-2}A_{3,1}^{-1}A_{3,3}^{-1}.$

Up to the symmetries $1\leftrightarrow 3$ of $I$, we reduce to the monomials:

\noindent\medskip$mA_{1,3}^{-1}A_{2,4}^{-1},  5mA_{1,3}^{-1}A_{2,4}^{-1}A_{3,3}^{-1}, 2mA_{1,3}^{-1}A_{2,4}^{-2}A_{3,3}^{-2},
3mA_{1,3}^{-2}A_{2,4}^{-1}A_{3,3}^{-2},
2mA_{1,1}^{-1}A_{1,3}^{-1}A_{2,2}^{-1}A_{2,4}^{-1}A_{3,1}^{-1}A_{3,3}^{-1},
 mA_{1,3}^{-1}A_{2,4}^{-2}A_{3,3}^{-1},\\
\medskip
m A_{1,1}^{-1}A_{1,3}^{-1}A_{2,4}^{-1},\
m A_{1,1}^{-1}A_{1,3}^{-1}A_{3,3}^{-1}A_{2,4}^{-2},\
5mA_{1,1}^{-1}A_{1,3}^{-1}A_{3,3}^{-1}A_{2,4}^{-1},\
4mA_{1,1}^{-1}A_{1,3}^{-2}A_{2,2}^{-1}A_{2,4}^{-2}A_{3,1}^{-1}A_{3,3}^{-2},\
2m A_{1,1}^{-1}A_{1,3}^{-1}A_{3,3}^{-2}A_{2,4}^{-2},\ \\
\medskip
2mA_{1,1}^{-1}A_{1,3}^{-1}A_{2,2}^{-1}A_{2,4}^{-2}A_{3,1}^{-1}A_{3,3}^{-2},
5mA_{1,1}^{-1}A_{1,3}^{-1}A_{2,4}^{-1}A_{3,1}^{-1}A_{3,3}^{-1},
mA_{1,1}^{-2}A_{1,3}^{-2}A_{2,2}^{-2}A_{2,4}^{-2}A_{3,1}^{-2}A_{3,3}^{-2},
mA_{1,1}^{-1}A_{1,3}^{-1}A_{2,4}^{-2}A_{3,1}^{-1}A_{3,3}^{-1}, \\
mA_{1,1}^{-1}A_{1,3}^{-1}A_{2,2}^{-1}A_{2,4}^{-2}A_{3,1}^{-1}A_{3,3}^{-1}.$

After some calculation, we have

\noindent\medskip
  $mA_{1,3}^{-1}A_{2,4}^{-1}+m A_{1,1}^{-1}A_{1,3}^{-1}A_{2,4}^{-1}\leq\varphi_{\{1,2\}}(m)_{\leq4},$
  $4mA_{1,1}^{-1}A_{1,3}^{-2}A_{2,2}^{-1}A_{2,4}^{-2}A_{3,1}^{-1}A_{3,3}^{-2}\leq
  \varphi_{\{2,3\}}(2mA_{1,1}^{-1}A_{1,3}^{-2})_{\leq4},$ \\
  \medskip
  $mA_{1,3}^{-1}A_{2,4}^{-1}(A_{3,3}^{-1}+A_{2,4}^{-1}A_{3,3}^{-1})\leq
  \varphi_{\{2,3\}}(mA_{1,3}^{-1}A_{2,4}^{-1})_{\leq4},$ $mA_{1,1}^{-2}A_{1,3}^{-2}A_{2,2}^{-2}A_{2,4}^{-2}A_{3,1}^{-2}A_{3,3}^{-2}\leq
  \varphi_{\{2,3\}}(mA_{1,1}^{-2}A_{1,3}^{-2})_{\leq4},$\\
  \medskip
  $3mA_{1,3}^{-2}A_{2,4}^{-1}A_{3,3}^{-2}\leq\varphi_{\{2,3\}}(mA_{1,3}^{-2})_{\leq4},$ \quad
  $2mA_{1,3}^{-1}(2A_{2,4}^{-2}A_{3,3}^{-1}+A_{2,4}^{-2}A_{3,3}^{-2})\leq
  \varphi_{\{2,3\}}(2mA_{1,3}^{-1})_{\leq4},$\\
  \medskip
 $mA_{1,1}^{-1}A_{1,3}^{-1}A_{2,4}^{-1}(A_{3,3}^{-1}+A_{2,4}^{-1}A_{3,3}^{-1}+A_{2,4}^{-1}A_{3,1}^{-1}A_{3,3}^{-1}+
  A_{2,2}^{-1}A_{2,4}^{-1}A_{3,1}^{-1}A_{3,3}^{-1})\leq
  \varphi_{\{2,3\}}( mA_{1,1}^{-1}A_{1,3}^{-1}A_{2,4}^{-1})_{\leq4},$ \\ $2mA_{1,1}^{-1}A_{1,3}^{-1}(2A_{2,4}^{-2}A_{3,3}^{-1}+A_{2,4}^{-2}A_{3,3}^{-2}+2A_{2,4}^{-1}A_{3,1}^{-1}A_{3,3}^{-1}
  +A_{2,2}^{-1}A_{2,4}^{-1}A_{3,1}^{-1}A_{3,3}^{-1}+A_{2,2}^{-1}A_{2,4}^{-2}A_{3,1}^{-1}A_{3,3}^{-2})\leq\\
  \varphi_{\{2,3\}}(2mA_{1,1}^{-1}A_{1,3}^{-1})_{\leq4}.$

If $\beta=\a_0+2\a_1+3\a_2+2\a_3+\a_4+\a_5$, $Y^{\beta}=Y_{1,0}Y_{1,2}Y_{2,3}Y_{2,5}^{2}Y_{3,0}Y_{3,2}$, denote
\[
\begin{array}{lll}
\medskip
  m=Y^{\a}Y^{\beta}, & m_1=mA_{1,1}^{-1}A_{1,3}^{-1}A_{2,2}^{-1}A_{2,4}^{-1}A_{3,1}^{-1}A_{3,3}^{-1}, &
  m_2=mA_{1,1}^{-1}A_{1,3}^{-2}A_{2,2}^{-1}A_{2,4}^{-1}A_{3,1}^{-1}A_{3,3}^{-1},\\
  m_3=mA_{1,1}^{-1}A_{1,3}^{-1}A_{2,2}^{-1}A_{2,4}^{-1}A_{3,1}^{-1}A_{3,3}^{-2},&
  m_4=mA_{1,1}^{-1}A_{1,3}^{-2}A_{2,2}^{-1}A_{2,4}^{-1}A_{3,1}^{-1}A_{3,3}^{-2},&
  m_5=mA_{1,1}^{-1}A_{1,3}^{-2}A_{2,2}^{-1}A_{2,4}^{-2}A_{3,1}^{-1}A_{3,3}^{-2}.
\end{array}
\]

We can check that all the monomials of $\chi_q(S(\a)\otimes S(\beta))_{\leq4}$ excluding $m_i(1\leq i\leq 5)$ are contained in $\chi_q(L(m))_{\leq4}$. And denote the set of monomials occur in $\chi_q(L(m))_{\leq4}$ by $\mathcal{L}$.

If all $m_i$ are not contained in $\chi_q(L(m))_{\leq4}$. Consider the monomial $A_{1,1}^{-1}A_{1,3}^{-1}m_1$, by Proposition 2.14, we fix $i=1$, then there is a
unique decomposition of $\chi_q(L(m))$ as a finite sum
\[
 \chi_q(L(m)) = \sum_{\substack{m\in\M_{1,+}\\ \lambda_m\geq0}} \lambda_m \varphi_1(m).
\]
But, obviously, for any $m'\in \mathcal{L}$, $A_{1,1}^{-1}A_{1,3}^{-1}m_1$ does not occur in $\varphi_1(m')$.
And only for $m_1$, we have
\[
A_{1,1}^{-1}A_{1,3}^{-1}m_1\leq\varphi_1(m_1)_{\leq4}=m_1(1+A_{1,3}^{-1}+A_{1,1}^{-1}A_{1,3}^{-1}).
\]
So, $m_1+m_2$ is contained in $\chi_q(L(m))_{\leq4}$.

Consider the monomial $A_{3,1}^{-1}A_{3,3}^{-1}m_1$, fix $i=3$, we can similarly check that $m_3$ is contained in $\chi_q(L(m))_{\leq4}$.

Consider the monomial $A_{1,1}^{-1}A_{1,3}^{-1}A_{3,3}^{-1}A_{2,4}^{-1}m_1$, fix $i=1$, only for $m_5$, $A_{1,1}^{-1}A_{1,3}^{-1}A_{3,3}^{-1}A_{2,4}^{-1}m_1$ occur in $\varphi_1(m_5)$.

Since $m_4$ is not dominant, we have proved that all dominant monomials of $\chi_q(S(\a)\otimes S(\beta))_{\leq4}$ are contained in $\chi_q(L(m))_{\leq4}$.

\medskip
(19) $\a=\a_1+\a_2+\a_3+\a_4+\a_5$, $Y^{\a}=Y_{1,0}Y_{2,3}$. If $\beta\in C_\a\backslash\{-\a_0, \a_1, \a_1+\a_2+\a_5, \a_1+\a_2+\a_3+\a_4, \a_1+\a_2+\a_3+\a_5, \a_1+2\a_2+2\a_3+\a_4+\a_5, \a_0+2\a_1+2\a_2+2\a_3+\a_4+\a_5\}$, then $\chi_q(S(\a)\otimes S(\beta))_{\leq4}$ contains a unique dominant monomial and $S(\a)\otimes S(\beta)$ is simple.

For $\beta=-\a_0, \a_1, \a_1+\a_2+\a_5, \a_1+\a_2+\a_3+\a_4, \a_1+\a_2+\a_3+\a_5$, we have done.

If $\beta=\a_1+2\a_2+2\a_3+\a_4+\a_5$, $Y^{\beta}=Y_{1,0}Y_{2,3}Y_{2,5}Y_{3,0}Y_{3,2}$, denote $m=Y^{\a}Y^{\beta}$. The only dominant monomials in $\chi_q(S(\a)\otimes S(\beta))_{\leq4}$ are $m,\ mA_{2,4}^{-1},\ 2mA_{3,3}^{-1}A_{2,4}^{-1},\ mA_{1,1}^{-1}A_{2,2}^{-1}A_{2,4}^{-1},\ mA_{1,1}^{-1}A_{2,2}^{-1}A_{2,4}^{-1}A_{3,3}^{-1},$ and  $2mA_{1,1}^{-1}A_{2,2}^{-1}A_{2,4}^{-1}A_{3,1}^{-1}A_{3,3}^{-1}.$
\[
\begin{array}{ll}
\medskip
  mA_{2,4}^{-1}+mA_{1,1}^{-1}A_{2,2}^{-1}A_{2,4}^{-1}\leq\varphi_{\{1,2\}}(m)_{\leq4}, &
  2mA_{3,3}^{-1}A_{2,4}^{-1}\leq\varphi_{2}(mA_{3,3}^{-1})_{\leq4}, \\
  2mA_{1,1}^{-1}A_{2,2}^{-1}A_{2,4}^{-1}A_{3,1}^{-1}A_{3,3}^{-1}\leq
  \varphi_{2}(2mA_{1,1}^{-1}A_{3,1}^{-1}A_{3,3}^{-1})_{\leq4}, &
  mA_{1,1}^{-1}A_{2,2}^{-1}A_{2,4}^{-1}A_{3,3}^{-1}\leq
  \varphi_{3}(mA_{1,1}^{-1}A_{2,2}^{-1}A_{2,4}^{-1})_{\leq4}.
\end{array}
\]

If $\beta=\a_0+2\a_1+2\a_2+2\a_3+\a_4+\a_5$, $Y^{\beta}=Y_{1,0}Y_{2,3}^{2}Y_{2,5}Y_{3,0}$, denote $m=Y^{\a}Y^{\beta}$. The only dominant monomials in $\chi_q(S(\a)\otimes S(\beta))_{\leq4}$ are $m,\ 2mA_{2,4}^{-1},\  mA_{1,1}^{-1}A_{2,2}^{-1}A_{2,4}^{-1},$  and $2mA_{1,1}^{-1}A_{2,2}^{-1}A_{2,4}^{-1}A_{3,1}^{-1}.$
\[
\begin{array}{ll}
  2mA_{2,4}^{-1}+mA_{1,1}^{-1}A_{2,2}^{-1}A_{2,4}^{-1}\leq\varphi_{\{1,2\}}(m)_{\leq4}, &
  2mA_{1,1}^{-1}A_{2,2}^{-1}A_{2,4}^{-1}A_{3,1}^{-1}\leq
  \varphi_{2}(2mA_{1,1}^{-1}A_{3,1}^{-1})_{\leq4}.
\end{array}
\]

\medskip
(20) $\a=\a_0+2\a_1+2\a_2+\a_3+\a_5$, $Y^{\a}=Y_{1,0}Y_{1,2}Y_{2,3}Y_{2,5}Y_{3,0}$. In this case, $\chi_q(L(Y^{\a}))$ has $3$ dominant monomials, so we need to consider every $\beta\in C_\a.$

For $\beta=-\a_4, \a_0+\a_1,\ \a_1+\a_2,\ \a_1+\a_2+\a_3,\ \a_2+\a_3+\a_5,\ \a_1+\a_2+\a_5,\ \a_0+\a_1+\a_2+\a_3,\a_0+\a_1+\a_2+\a_5,\ \a_1+\a_2+\a_3+\a_4,\ \a_1+2\a_2+\a_3+\a_5,\ \a_0+\a_1+\a_2+\a_3+\a_5$, we have done.

If $\beta=\a_0+2\a_1+2\a_2+\a_3+\a_5$, $Y^{\beta}=Y^{\a}$, denote $m=Y^{\a}Y^{\beta}$. The only dominant monomials in $\chi_q(S(\a)\otimes S(\beta))_{\leq4}$ are
\smallskip
$m,\ 2mA_{1,3}^{-1}A_{2,4}^{-1},\ mA_{1,3}^{-2}A_{2,4}^{-2},\  2mA_{1,1}^{-1}A_{1,3}^{-1}A_{2,2}^{-1}A_{2,4}^{-1}A_{3,1}^{-1},\ 2mA_{1,1}^{-1}A_{1,3}^{-2}A_{2,2}^{-1}A_{2,4}^{-2}A_{3,1}^{-1},\ \\ mA_{1,1}^{-2}A_{1,3}^{-2}A_{2,2}^{-2}A_{2,4}^{-2}A_{3,1}^{-2},$ and $2mA_{1,1}^{-1}A_{1,3}^{-1}A_{2,4}^{-1}.$

We have
\[
2mA_{1,3}^{-1}A_{2,4}^{-1}+mA_{1,3}^{-2}A_{2,4}^{-2}+2mA_{1,1}^{-1}A_{1,3}^{-1}A_{2,4}^{-1}+2mA_{1,1}^{-1}A_{1,3}^{-1}
+2mA_{1,1}^{-1}A_{1,3}^{-2}+mA_{1,1}^{-2}A_{1,3}^{-2}\leq\varphi_{\{1,2\}}(m)_{\leq4}.
\]
\[
\begin{array}{ll}
\medskip
  2mA_{1,1}^{-1}A_{1,3}^{-1}A_{2,2}^{-1}A_{2,4}^{-1}A_{3,1}^{-1}\leq
  \varphi_{\{2,3\}}(2mA_{1,1}^{-1}A_{1,3}^{-1})_{\leq4}, &
   2mA_{1,1}^{-1}A_{1,3}^{-2}A_{2,2}^{-1}A_{2,4}^{-2}A_{3,1}^{-1}\leq
    \varphi_{\{2,3\}}( 2mA_{1,1}^{-1}A_{1,3}^{-2})_{\leq4}, \\
  mA_{1,1}^{-2}A_{1,3}^{-2}A_{2,2}^{-2}A_{2,4}^{-2}A_{3,1}^{-2}\leq
    \varphi_{\{2,3\}}( mA_{1,1}^{-2}A_{1,3}^{-2})_{\leq4}.
\end{array}
\]

If $\beta=\a_1+2\a_2+2\a_3+\a_4+\a_5$, $Y^{\beta}=Y_{1,0}Y_{2,3}Y_{2,5}Y_{3,0}Y_{3,2}$, denote $m=Y^{\a}Y^{\beta}$. The only dominant monomials in $\chi_q(S(\a)\otimes S(\beta))_{\leq4}$ are
$m,\ mA_{1,1}^{-1},\ mA_{3,1}^{-1},\ mA_{1,1}^{-1}A_{3,1}^{-1},\  mA_{1,3}^{-1}A_{2,4}^{-1},\ mA_{3,3}^{-1}A_{2,4}^{-1},\ \\
mA_{1,1}^{-1}A_{3,3}^{-1}A_{2,4}^{-1}, mA_{1,3}^{-1}A_{3,1}^{-1}A_{2,4}^{-1},\
2mA_{1,3}^{-1}A_{3,3}^{-1}A_{2,4}^{-1},\ mA_{1,1}^{-1}A_{1,3}^{-1}A_{2,2}^{-1}A_{2,4}^{-1}A_{3,1}^{-1},\
mA_{1,1}^{-1}A_{2,2}^{-1}A_{2,4}^{-1}A_{3,1}^{-1}A_{3,3}^{-1},\ \\ 2mA_{1,1}^{-1}A_{1,3}^{-1}A_{2,2}^{-1}A_{2,4}^{-1}A_{3,1}^{-1}A_{3,3}^{-1},
mA_{1,3}^{-1}A_{3,3}^{-1}A_{2,4}^{-2},\
2mA_{1,1}^{-1}A_{1,3}^{-1}A_{2,2}^{-1}A_{2,4}^{-2}A_{3,1}^{-1}A_{3,3}^{-1},$ and $mA_{1,1}^{-2}A_{1,3}^{-1}A_{2,2}^{-2}A_{2,4}^{-2}A_{3,1}^{-2}A_{3,3}^{-1}.$

Up to the symmetries $1\leftrightarrow 3$ of $I$, we reduce to the monomials:

\noindent\medskip
$mA_{1,1}^{-1},\ mA_{1,1}^{-1}A_{3,1}^{-1},\ mA_{1,3}^{-1}A_{2,4}^{-1},\   mA_{1,1}^{-1}A_{3,3}^{-1}A_{2,4}^{-1},\ 2mA_{1,3}^{-1}A_{3,3}^{-1}A_{2,4}^{-1},\
mA_{1,3}^{-1}A_{3,3}^{-1}A_{2,4}^{-2},\
mA_{1,1}^{-1}A_{2,2}^{-1}A_{2,4}^{-1}A_{3,1}^{-1}A_{3,3}^{-1},\ \\
2mA_{1,1}^{-1}A_{1,3}^{-1}A_{2,2}^{-1}A_{2,4}^{-1}A_{3,1}^{-1}A_{3,3}^{-1},\ 2mA_{1,1}^{-1}A_{1,3}^{-1}A_{2,2}^{-1}A_{2,4}^{-2}A_{3,1}^{-1}A_{3,3}^{-1},\ \and\ \ mA_{1,1}^{-2}A_{1,3}^{-1}A_{2,2}^{-2}A_{2,4}^{-2}A_{3,1}^{-2}A_{3,3}^{-1}.$

We have
\[
\begin{array}{l}
\medskip
mA_{1,3}^{-1}(A_{2,4}^{-1}+2A_{3,3}^{-1}A_{2,4}^{-1}+A_{3,3}^{-1}A_{2,4}^{-2})\leq
   \varphi_{\{2,3\}}( mA_{1,3}^{-1})_{\leq4}, \\
\medskip
mA_{1,1}^{-1}(1+A_{3,1}^{-1}+A_{3,3}^{-1}A_{2,4}^{-1}+A_{2,2}^{-1}A_{2,4}^{-1}A_{3,1}^{-1}A_{3,3}^{-1})\leq
   \varphi_{\{2,3\}}( mA_{1,1}^{-1})_{\leq4}, \\
   2mA_{1,1}^{-1}A_{1,3}^{-1}(A_{2,2}^{-1}A_{2,4}^{-1}A_{3,1}^{-1}A_{3,3}^{-1}+A_{2,2}^{-1}A_{2,4}^{-2}A_{3,1}^{-1}A_{3,3}^{-1})
   \leq\varphi_{\{2,3\}}( 2mA_{1,1}^{-1}A_{1,3}^{-1})_{\leq4},
\end{array}
\]

Lastly, we have $mA_{1,1}^{-2}A_{1,3}^{-1}A_{2,2}^{-2}A_{2,4}^{-2}A_{3,1}^{-2}A_{3,3}^{-1}\leq\varphi_{2}( mA_{1,1}^{-2}A_{1,3}^{-1}A_{3,1}^{-2}A_{3,3}^{-1})_{\leq4}$.

If $\beta=\a_0+\a_1+\a_2+\a_3+\a_4+\a_5$, $Y^{\beta}=Y_{2,3}$, denote $m=Y^{\a}Y^{\beta}$. The only dominant monomials in $\chi_q(S(\a)\otimes S(\beta))_{\leq4}$ are $m,\ mA_{2,4}^{-1},\ 2mA_{1,3}^{-1}A_{2,4}^{-1},\
\and\ \ mA_{1,1}^{-1}A_{1,3}^{-1}A_{2,2}^{-1}A_{2,4}^{-1}A_{3,1}^{-1}$.
\[
\begin{array}{ccc}
 mA_{2,4}^{-1}\leq\varphi_{2}(m)_{\leq4}, &
 2mA_{1,3}^{-1}A_{2,4}^{-1}\leq\varphi_{2}(mA_{1,3}^{-1})_{\leq4}, &
 mA_{1,1}^{-1}A_{1,3}^{-1}A_{2,2}^{-1}A_{2,4}^{-1}A_{3,1}^{-1}\leq\varphi_{2}(mA_{1,1}^{-1}A_{1,3}^{-1}A_{3,1}^{-1})_{\leq4}.
\end{array}
\]

If $\beta=\a_0+2\a_1+2\a_2+2\a_3+\a_4+\a_5$, $Y^{\beta}=Y_{1,0}Y_{2,3}^{2}Y_{2,5}Y_{3,0}$, denote $m=Y^{\a}Y^{\beta}$. The only dominant monomials in $\chi_q(S(\a)\otimes S(\beta))_{\leq4}$ are
$m,\ mA_{1,1}^{-1},\  mA_{2,4}^{-1},\ mA_{1,1}^{-1}A_{2,4}^{-1},\ 2mA_{1,3}^{-1}A_{2,4}^{-1},\ mA_{1,3}^{-1}A_{2,4}^{-2},\ \\
mA_{1,1}^{-1}A_{2,2}^{-1}A_{2,4}^{-1}A_{3,1}^{-1},\  2mA_{1,1}^{-1}A_{1,3}^{-1}A_{2,2}^{-1}A_{2,4}^{-1}A_{3,1}^{-1},\ 2mA_{1,1}^{-1}A_{1,3}^{-1}A_{2,2}^{-1}A_{2,4}^{-2}A_{3,1}^{-1},$
and $ mA_{1,1}^{-2}A_{1,3}^{-1}A_{2,2}^{-2}A_{2,4}^{-2}A_{3,1}^{-2}$.
\[
\begin{array}{ll}
  mA_{1,1}^{-1}A_{2,2}^{-1}A_{2,4}^{-1}A_{3,1}^{-1}\leq\varphi_{\{2,3\}}( mA_{1,1}^{-1})_{\leq4}, &
  mA_{1,1}^{-2}A_{1,3}^{-1}A_{2,2}^{-2}A_{2,4}^{-2}A_{3,1}^{-2}\leq\varphi_2(mA_{1,1}^{-2}A_{1,3}^{-1}A_{3,1}^{-2})_{\leq4},
\end{array}
\]
\[
\begin{array}{l}
\medskip
  2mA_{3,1}^{-1}(A_{1,1}^{-1}A_{1,3}^{-1}A_{2,2}^{-1}A_{2,4}^{-1}+A_{1,1}^{-1}A_{1,3}^{-1}A_{2,2}^{-1}A_{2,4}^{-2})
  \leq\varphi_{\{1,2\}}(2mA_{3,1}^{-1})_{\leq4}, \\
  (mA_{1,1}^{-1}+mA_{2,4}^{-1}+mA_{1,1}^{-1}A_{2,4}^{-1}+2mA_{1,3}^{-1}A_{2,4}^{-1}+mA_{1,3}^{-1}A_{2,4}^{-2})
  \leq\varphi_{\{1,2\}}(m)_{\leq4}.
\end{array}
\]

If $\beta=\a_0+2\a_1+2\a_2+\a_3+\a_4+\a_5$, $Y^{\beta}=Y_{1,0}Y_{1,2}Y_{2,3}Y_{2,5}$, denote $m=Y^{\a}Y^{\beta}$. The only dominant monomials in $\chi_q(S(\a)\otimes S(\beta))_{\leq4}$ are $m,\ 2mA_{1,3}^{-1}A_{2,4}^{-1},\ mA_{1,3}^{-2}A_{2,4}^{-2},\ 2mA_{1,1}^{-1}A_{1,3}^{-1}A_{2,4}^{-1},\ \\ mA_{1,1}^{-1}A_{1,3}^{-1}A_{2,2}^{-1}A_{2,4}^{-1}A_{3,1}^{-1}$
and $ mA_{1,1}^{-1}A_{1,3}^{-2}A_{2,2}^{-1}A_{2,4}^{-2}A_{3,1}^{-1}$.
\[
\begin{array}{l}
\medskip
  (2mA_{1,3}^{-1}A_{2,4}^{-1}+mA_{1,3}^{-2}A_{2,4}^{-2}+2mA_{1,1}^{-1}A_{1,3}^{-1}A_{2,4}^{-1})
  \leq\varphi_{\{1,2\}}(m)_{\leq4},\\
  mA_{3,1}^{-1}(A_{1,1}^{-1}A_{1,3}^{-1}A_{2,2}^{-1}A_{2,4}^{-1}+A_{1,1}^{-1}A_{1,3}^{-2}A_{2,2}^{-1}A_{2,4}^{-2})
  \leq\varphi_{\{1,2\}}( mA_{3,1}^{-1})_{\leq4}.
\end{array}
\]

If $\beta=\a_0+2\a_1+3\a_2+2\a_3+\a_4+2\a_5$, $Y^{\beta}=Y_{1,0}Y_{1,2}Y_{2,3}Y_{2,5}Y_{3,0}Y_{3,2}$, denote $m=Y^{\a}Y^{\beta}$. The only dominant monomials in $\chi_q(S(\a)\otimes S(\beta))_{\leq4}$ are
\medskip
$m,\ mA_{3,1}^{-1},\  2mA_{1,3}^{-1}A_{2,4}^{-1},\ mA_{1,3}^{-2}A_{2,4}^{-2},\ mA_{3,3}^{-1}A_{2,4}^{-1},\ \\ 4mA_{1,3}^{-1}A_{2,4}^{-1}A_{3,3}^{-1},
2mA_{1,3}^{-1}A_{2,4}^{-2}A_{3,3}^{-1},\ 3mA_{1,3}^{-2}A_{2,4}^{-2}A_{3,3}^{-1},\ 2mA_{1,1}^{-1}A_{1,3}^{-1}A_{2,4}^{-1},\ 2mA_{1,3}^{-1}A_{2,4}^{-1}A_{3,1}^{-1},\ mA_{1,3}^{-2}A_{2,4}^{-2}A_{3,1}^{-1},\ \\
2mA_{1,1}^{-1}A_{1,3}^{-1}A_{2,4}^{-1}A_{3,1}^{-1},
2mA_{1,1}^{-1}A_{1,3}^{-1}A_{2,4}^{-2}A_{3,3}^{-1},\
4mA_{1,1}^{-1}A_{1,3}^{-1}A_{2,4}^{-1}A_{3,3}^{-1},\
4mA_{1,1}^{-1}A_{1,3}^{-2}A_{2,2}^{-1}A_{2,4}^{-2}A_{3,1}^{-1}A_{3,3}^{-1},\ \\
2mA_{1,1}^{-1}A_{1,3}^{-1}A_{2,2}^{-1}A_{2,4}^{-1}A_{3,1}^{-1}A_{3,3}^{-1},
mA_{1,1}^{-1}A_{1,3}^{-2}A_{2,2}^{-1}A_{2,4}^{-2}A_{3,1}^{-1},\
mA_{1,1}^{-1}A_{1,3}^{-1}A_{2,2}^{-1}A_{2,4}^{-1}A_{3,1}^{-1},\
2mA_{1,1}^{-1}A_{1,3}^{-1}A_{2,2}^{-1}A_{2,4}^{-2}A_{3,1}^{-1}A_{3,3}^{-1}$ and
$mA_{1,1}^{-2}A_{1,3}^{-2}A_{2,2}^{-2}A_{2,4}^{-2}A_{3,1}^{-2}A_{3,3}^{-1}.$

We have
\[
\begin{array}{l}
\medskip
  4mA_{1,1}^{-1}A_{1,3}^{-2}A_{2,2}^{-1}A_{2,4}^{-2}A_{3,1}^{-1}A_{3,3}^{-1}\leq
  \varphi_{\{1,2\}}(2mA_{3,1}^{-1}A_{3,3}^{-1})_{\leq4},\\
  \medskip
  mA_{1,1}^{-2}A_{1,3}^{-2}A_{2,2}^{-2}A_{2,4}^{-2}A_{3,1}^{-2}A_{3,3}^{-1}\leq
  \varphi_{2}( mA_{1,1}^{-2}A_{1,3}^{-2}A_{3,1}^{-2}A_{3,3}^{-1})_{\leq4},\\
  \medskip
  (2mA_{1,3}^{-1}A_{2,4}^{-1}+mA_{1,3}^{-2}A_{2,4}^{-2}+2mA_{1,1}^{-1}A_{1,3}^{-1}A_{2,4}^{-1}) \leq\varphi_{\{1,2\}}(m)_{\leq4},\\
  \medskip
  mA_{3,3}^{-1}(A_{2,4}^{-1}+4A_{1,3}^{-1}A_{2,4}^{-1}+2A_{1,3}^{-1}A_{2,4}^{-2}+3A_{1,3}^{-2}A_{2,4}^{-2})
  \leq\varphi_{\{1,2\}}(mA_{3,3}^{-1})_{\leq4},\\
  \medskip
  2mA_{1,1}^{-1}A_{1,3}^{-1}(2A_{3,3}^{-1}A_{2,4}^{-1}+A_{3,3}^{-1}A_{2,4}^{-2}+A_{2,2}^{-1}A_{2,4}^{-1}A_{3,1}^{-1}A_{3,3}^{-1}
  +A_{2,2}^{-1}A_{2,4}^{-2}A_{3,1}^{-1}A_{3,3}^{-1})
   \leq\varphi_{\{2,3\}}( 2mA_{1,1}^{-1}A_{1,3}^{-1})_{\leq4},\\
  mA_{3,1}^{-1}(1+2A_{1,3}^{-1}A_{2,4}^{-1}+A_{1,3}^{-2}A_{2,4}^{-2}+2A_{1,1}^{-1}A_{1,3}^{-1}A_{2,4}^{-1}+
 A_{1,1}^{-1}A_{1,3}^{-1}A_{2,2}^{-1}A_{2,4}^{-1}+A_{1,1}^{-1}A_{1,3}^{-2}A_{2,2}^{-1}A_{2,4}^{-2})
 \leq\varphi_{\{1,2\}}(mA_{3,1}^{-1})_{\leq4}.
\end{array}
\]

\medskip
(21) $\a=\a_0+\a_1+2\a_2+\a_3+\a_5$, $Y^{\a}=Y_{1,2}Y_{2,5}Y_{3,0}Y_{3,2}$. In this case, $\chi_q(L(Y^{\a}))$ has $2$ dominant monomials, so we need to consider every $\beta\in C_\a.$

For $\beta=-\a_4, \a_0,\ \a_1+\a_2,\ \a_2+\a_3,\ \a_2+\a_5,\ \a_0+\a_1+\a_2,\ \a_2+\a_3+\a_5,\  \a_0+\a_1+\a_2+\a_3,\ \a_0+\a_1+\a_2+\a_5,\ \a_2+\a_3+\a_4+\a_5,\  \a_1+2\a_2+\a_3+\a_5$, we have done.

If $\beta=\a_0+\a_1+2\a_2+\a_3+\a_5$, $Y^{\beta}=Y^{\a}$, denote $m=Y^{\a}Y^{\beta}$. The only dominant monomials in $\chi_q(S(\a)\otimes S(\beta))_{\leq4}$ are $m,\ 2mA_{1,3}^{-1}A_{2,4}^{-1}A_{3,3}^{-1},\
mA_{1,3}^{-2}A_{2,4}^{-2}A_{3,3}^{-2},\ \and\ \ 2mA_{1,3}^{-1}A_{2,4}^{-1}A_{3,1}^{-1}A_{3,3}^{-1}.$
\[
\begin{array}{ll}
  mA_{1,3}^{-2}A_{2,4}^{-2}A_{3,3}^{-2}\leq\varphi_{\{2,3\}}(mA_{1,3}^{-2})_{\leq4}, &
  2mA_{1,3}^{-1}(A_{2,4}^{-1}A_{3,3}^{-1}+A_{2,4}^{-1}A_{3,1}^{-1}A_{3,3}^{-1})
  \leq\varphi_{\{2,3\}}(2mA_{1,3}^{-1})_{\leq4}.
\end{array}
\]

If $\beta=\a_1+2\a_2+\a_3+\a_4+\a_5$, $Y^{\beta}=Y_{1,0}Y_{1,2}Y_{2,5}Y_{3,2}$, denote $m=Y^{\a}Y^{\beta}$. The only dominant monomials in $\chi_q(S(\a)\otimes S(\beta))_{\leq4}$ are $m, 2mA_{1,3}^{-1}A_{2,4}^{-1}A_{3,3}^{-1},\
mA_{1,3}^{-2}A_{2,4}^{-2}A_{3,3}^{-2},\ m A_{1,1}^{-1}A_{1,3}^{-1}A_{2,4}^{-1}A_{3,3}^{-1}$ and  $mA_{1,3}^{-1}A_{2,4}^{-1}A_{3,1}^{-1}A_{3,3}^{-1}.$
\[
\begin{array}{ll}
 \medskip
  mA_{1,3}^{-2}A_{2,4}^{-2}A_{3,3}^{-2}\leq\varphi_{2}(mA_{1,3}^{-2}A_{3,3}^{-2})_{\leq4}, &
  2mA_{1,3}^{-1}A_{2,4}^{-1}A_{3,3}^{-1}\leq\varphi_{\{2,3\}}(2mA_{1,3}^{-1})_{\leq4}, \\
  m A_{1,1}^{-1}A_{1,3}^{-1}A_{2,4}^{-1}A_{3,3}^{-1}\leq\varphi_{\{2,3\}}(m A_{1,1}^{-1}A_{1,3}^{-1})_{\leq4},&
  mA_{1,3}^{-1}A_{2,4}^{-1}A_{3,1}^{-1}A_{3,3}^{-1}\leq\varphi_{\{1,2\}}(m A_{3,1}^{-1}A_{3,3}^{-1})_{\leq4}.
\end{array}
\]

If $\beta=\a_0+\a_1+\a_2+\a_3+\a_4$, $Y^{\beta}=Y_{2,3}Y_{2,5}$, denote $m=Y^{\a}Y^{\beta}$. The only dominant monomials in $\chi_q(S(\a)\otimes S(\beta))_{\leq4}$ are $m$  and $mA_{1,3}^{-1}A_{2,4}^{-1}A_{3,3}^{-1}.$

We have $mA_{1,3}^{-1}A_{2,4}^{-1}A_{3,3}^{-1}\leq\varphi_{2}(mA_{1,3}^{-1}A_{3,3}^{-1})_{\leq4}.$

If $\beta=\a_0+2\a_1+3\a_2+2\a_3+\a_4+\a_5$, $Y^{\beta}=Y_{1,0}Y_{1,2}Y_{2,3}Y_{2,5}^{2}Y_{3,0}Y_{3,2}$, denote $m=Y^{\a}Y^{\beta}$. The only dominant monomials in $\chi_q(S(\a)\otimes S(\beta))_{\leq4}$ are $m,\ 2mA_{1,3}^{-1}A_{2,4}^{-1}A_{3,3}^{-1},\
mA_{1,3}^{-2}A_{2,4}^{-2}A_{3,3}^{-2},\ m A_{1,1}^{-1}A_{1,3}^{-1}A_{2,4}^{-1}A_{3,3}^{-1},\\ 2mA_{1,3}^{-1}A_{2,4}^{-1}A_{3,1}^{-1}A_{3,3}^{-1}\and\ \
m A_{1,1}^{-1}A_{1,3}^{-1}A_{3,1}^{-1}A_{3,3}^{-1}A_{2,4}^{-1}.$
\[
\begin{array}{ll}
\medskip
  2mA_{1,3}^{-1}(A_{2,4}^{-1}A_{3,3}^{-1}+A_{2,4}^{-1}A_{3,1}^{-1}A_{3,3}^{-1})
  \leq\varphi_{\{2,3\}}(2mA_{1,3}^{-1})_{\leq4}, &
  mA_{1,3}^{-2}A_{2,4}^{-2}A_{3,3}^{-2}\leq\varphi_{\{2,3\}}(mA_{1,3}^{-2})_{\leq4},  \\
  m A_{1,1}^{-1}A_{1,3}^{-1}(A_{2,4}^{-1}A_{3,3}^{-1}+A_{3,1}^{-1}A_{3,3}^{-1}A_{2,4}^{-1})
  \leq\varphi_{\{2,3\}}( m A_{1,1}^{-1}A_{1,3}^{-1})_{\leq4}.
\end{array}
\]

If $\beta=\a_0+\a_1+2\a_2+\a_3+\a_4+\a_5$, $Y^{\beta}=Y_{1,2}Y_{2,5}Y_{3,2}$, denote $m=Y^{\a}Y^{\beta}$. The only dominant monomials in $\chi_q(S(\a)\otimes S(\beta))_{\leq4}$ are $m,\ 2mA_{1,3}^{-1}A_{2,4}^{-1}A_{3,3}^{-1},\
mA_{1,3}^{-2}A_{2,4}^{-2}A_{3,3}^{-2} \and\ \ mA_{1,3}^{-1}A_{2,4}^{-1}A_{3,1}^{-1}A_{3,3}^{-1}.$
\[
\begin{array}{ll}
\medskip
  mA_{1,3}^{-2}A_{2,4}^{-2}A_{3,3}^{-2}\leq\varphi_{2}(mA_{1,3}^{-2}A_{3,3}^{-2})_{\leq4}, &
  2mA_{1,3}^{-1}A_{2,4}^{-1}A_{3,3}^{-1}\leq\varphi_{\{2,3\}}(2mA_{1,3}^{-1})_{\leq4},\\
  mA_{1,3}^{-1}A_{2,4}^{-1}A_{3,1}^{-1}A_{3,3}^{-1}\leq\varphi_{\{1,2\}}(mA_{3,1}^{-1}A_{3,3}^{-1})_{\leq4}.
\end{array}
\]

If $\beta=\a_0+\a_1+2\a_2+2\a_3+\a_4+\a_5$, $Y^{\beta}=Y_{2,3}Y_{2,5}Y_{3,0}Y_{3,2}$, denote $m=Y^{\a}Y^{\beta}$. The only dominant monomials in $\chi_q(S(\a)\otimes S(\beta))_{\leq4}$ are $m,\ mA_{2,4}^{-1}A_{3,3}^{-1},\ 2mA_{1,3}^{-1}A_{2,4}^{-1}A_{3,3}^{-1},\
mA_{1,3}^{-1}A_{2,4}^{-2}A_{3,3}^{-2},\ mA_{2,4}^{-1}A_{3,1}^{-1}A_{3,3}^{-1}$ and $2mA_{1,3}^{-1}A_{2,4}^{-1}A_{3,1}^{-1}A_{3,3}^{-1}.$
\[
\begin{array}{l}
\medskip
  mA_{2,4}^{-1}A_{3,3}^{-1}+mA_{2,4}^{-1}A_{3,1}^{-1}A_{3,3}^{-1}\leq\varphi_{\{2,3\}}(m)_{\leq4}, \\
  mA_{1,3}^{-1}(2A_{2,4}^{-1}A_{3,3}^{-1}+A_{2,4}^{-2}A_{3,3}^{-2}+2A_{2,4}^{-1}A_{3,1}^{-1}A_{3,3}^{-1})
  \leq\varphi_{\{2,3\}}( mA_{1,3}^{-1})_{\leq4}.
\end{array}
\]

\medskip
(22) $\a=\a_0+\a_1+\a_2+\a_3+\a_4$, $Y^{\a}=Y_{2,3}Y_{2,5}$. If $\beta\in C_\a\backslash\{-\a_5,\  \a_0+\a_1+\a_2+\a_5,\ \a_2+\a_3+\a_4+\a_5,\  \a_0+\a_1+2\a_2+\a_3+\a_5,\  \a_1+2\a_2+\a_3+\a_4+\a_5,\  \a_0+2\a_1+2\a_2+\a_3+\a_4+\a_5,\  \a_0+\a_1+2\a_2+2\a_3+\a_4+\a_5,\  \a_0+\a_1+\a_2+\a_3+\a_4+\a_5,\  \a_0+\a_1+2\a_2+\a_3+\a_4+\a_5,\  \a_0+2\a_1+3\a_2+2\a_3+\a_4+\a_5\}$, then $\chi_q(S(\a)\otimes S(\beta))_{\leq4}$ contains a unique dominant monomial and $S(\a)\otimes S(\beta)$ is simple.

For $\beta=-\a_5,\  \a_0+\a_1+\a_2+\a_5,\ \a_2+\a_3+\a_4+\a_5,\  \a_0+\a_1+2\a_2+\a_3+\a_5,\  \a_1+2\a_2+\a_3+\a_4+\a_5$, we have done.

If $\beta=\a_0+2\a_1+2\a_2+\a_3+\a_4+\a_5$, $Y^{\beta}=Y_{1,0}Y_{1,2}Y_{2,3}Y_{2,5}$, denote $m=Y^{\a}Y^{\beta}$. The only dominant monomials in $\chi_q(S(\a)\otimes S(\beta))_{\leq4}$ are $m$ and $mA_{1,3}^{-1}A_{2,4}^{-1}.$
\[
mA_{1,3}^{-1}A_{2,4}^{-1}\leq\varphi_2(mA_{1,3}^{-1})_{\leq4}.
\]

If $\beta=\a_0+\a_1+2\a_2+2\a_3+\a_4+\a_5$, $Y^{\beta}=Y_{3,0}Y_{3,2}Y_{2,3}Y_{2,5}$, denote $m=Y^{\a}Y^{\beta}$. The only dominant monomials in $\chi_q(S(\a)\otimes S(\beta))_{\leq4}$ are $m$ and $mA_{3,3}^{-1}A_{2,4}^{-1}.$
\[
mA_{3,3}^{-1}A_{2,4}^{-1}\leq\varphi_2(mA_{3,3}^{-1})_{\leq4}.
\]

If $\beta=\a_0+\a_1+\a_2+\a_3+\a_4+\a_5$, $Y^{\beta}=Y_{2,3}$, denote $m=Y^{\a}Y^{\beta}$. The only dominant monomials in $\chi_q(S(\a)\otimes S(\beta))_{\leq4}$ are $m$ and $mA_{2,4}^{-1}.$
\[
\varphi_2(m)_{\leq4}=m(1+A_{2,4}^{-1}).
\]

If $\beta=\a_0+\a_1+2\a_2+\a_3+\a_4+\a_5$, $Y^{\beta}=Y_{1,2}Y_{2,5}Y_{3,2}$, denote $m=Y^{\a}Y^{\beta}$. The only dominant monomials in $\chi_q(S(\a)\otimes S(\beta))_{\leq4}$ are $m$ and $mA_{1,3}^{-1}A_{3,3}^{-1}A_{2,4}^{-1}.$
\[
mA_{1,3}^{-1}A_{3,3}^{-1}A_{2,4}^{-1}\leq\varphi_2(mA_{1,3}^{-1}A_{3,3}^{-1})_{\leq4}.
\]

If $\beta=\a_0+2\a_1+3\a_2+2\a_3+\a_4+\a_5$, $Y^{\beta}=Y_{1,0}Y_{1,2}Y_{2,3}Y_{2,5}^{2}Y_{3,0}Y_{3,2}$, denote $m=Y^{\a}Y^{\beta}$. The only dominant monomials in $\chi_q(S(\a)\otimes S(\beta))_{\leq4}$ are $m$ and $mA_{1,3}^{-1}A_{3,3}^{-1}A_{2,4}^{-1}.$
\[
mA_{1,3}^{-1}A_{3,3}^{-1}A_{2,4}^{-1}\leq\varphi_2(mA_{1,3}^{-1}A_{3,3}^{-1})_{\leq4}.
\]

\medskip
(23) $\a=\a_0+2\a_1+2\a_2+\a_3+\a_4+\a_5$, $Y^{\a}=Y_{1,0}Y_{1,2}Y_{2,3}Y_{2,5}$. In this case, $\chi_q(L(Y^{\a}))$ has $2$ dominant monomials, so we need to consider every $\beta\in C_\a.$

For $\beta=\a_4,\ \a_1+\a_2,\ \a_0+\a_1,\ \a_1+\a_2+\a_5,\  \a_1+\a_2+\a_3+\a_4,\ \a_0+\a_1+\a_2+\a_3,\
 \a_0+\a_1+\a_2+\a_5,\ \a_2+\a_3+\a_4+\a_5,\ \a_1+2\a_2+\a_3+\a_5,\ \a_0+2\a_1+2\a_2+\a_3+\a_5,\ \a_1+2\a_2+\a_3+\a_4+\a_5,\ \a_0+\a_1+\a_2+\a_3+\a_4$, we have done.

If $\beta=\a_0+2\a_1+2\a_2+\a_3+\a_4+\a_5$, $Y^{\beta}=Y^{\a}$, denote $m=Y^{\a}Y^{\beta}$. The only dominant monomials in $\chi_q(S(\a)\otimes S(\beta))_{\leq4}$ are $m,\ 2mA_{1,3}^{-1}A_{2,4}^{-1},\ mA_{1,3}^{-2}A_{2,4}^{-2}\  \and\ \ 2m A_{1,1}^{-1}A_{1,3}^{-1}A_{2,4}^{-1}.$
\[
2mA_{1,3}^{-1}A_{2,4}^{-1}+ mA_{1,3}^{-2}A_{2,4}^{-2}+2mA_{1,1}^{-1}A_{1,3}^{-1}A_{2,4}^{-1}
\leq\varphi_{\{1,2\}}(m)_{\leq4}.
\]

If $\beta=\a_0+\a_1+2\a_2+2\a_3+\a_4+\a_5$, $Y^{\beta}=Y_{2,3}Y_{2,5}Y_{3,0}Y_{3,2}$, denote $m=Y^{\a}Y^{\beta}$. The only dominant monomials in $\chi_q(S(\a)\otimes S(\beta))_{\leq4}$ are $m,\ mA_{1,3}^{-1}A_{2,4}^{-1},\ mA_{3,3}^{-1}A_{2,4}^{-1},\ mA_{1,3}^{-1}A_{3,3}^{-1}A_{2,4}^{-2},\  \and\ \
2m A_{1,3}^{-1}A_{3,3}^{-1}A_{2,4}^{-1}.$
\[
\begin{array}{cc}
  mA_{3,3}^{-1}A_{2,4}^{-1}\leq\varphi_{2}(mA_{3,3}^{-1})_{\leq4}, &
  mA_{1,3}^{-1}(A_{2,4}^{-1}+A_{3,3}^{-1}A_{2,4}^{-2}+2A_{3,3}^{-1}A_{2,4}^{-1})
  \leq\varphi_{\{2,3\}}(mA_{1,3}^{-1})_{\leq4}.
\end{array}
\]

If $\beta=\a_0+2\a_1+3\a_2+2\a_3+\a_4+2\a_5$, $Y^{\beta}=Y_{1,0}Y_{1,2}Y_{2,3}Y_{2,5}Y_{3,0}Y_{3,2}$, denote $m=Y^{\a}Y^{\beta}$. The only dominant monomials in $\chi_q(S(\a)\otimes S(\beta))_{\leq4}$ are $m$,\ $2mA_{1,3}^{-1}A_{2,4}^{-1}$,\ $mA_{3,3}^{-1}A_{2,4}^{-1},\ mA_{1,3}^{-2}A_{2,4}^{-2}$,\
$2m A_{1,1}^{-1}A_{1,3}^{-1}A_{2,4}^{-1}$,
$4mA_{1,3}^{-1}A_{3,3}^{-1}A_{2,4}^{-1}$,\
$2mA_{1,3}^{-1}A_{3,3}^{-1}A_{2,4}^{-2},\ 3mA_{1,3}^{-2}A_{3,3}^{-1}A_{2,4}^{-2},$\
$4m A_{1,1}^{-1}A_{1,3}^{-1}A_{3,3}^{-1}A_{2,4}^{-1}$,\\
$m A_{1,1}^{-1}A_{1,3}^{-1}A_{2,2}^{-1}A_{2,4}^{-1}A_{3,1}^{-1}A_{3,3}^{-1}$,\
$2m A_{1,1}^{-1}A_{1,3}^{-1}A_{3,3}^{-1}A_{2,4}^{-2}$,\
$m A_{1,1}^{-1}A_{1,3}^{-1}A_{2,2}^{-1}A_{2,4}^{-2}A_{3,1}^{-1}A_{3,3}^{-1},$ and \\
$2m A_{1,1}^{-1}A_{1,3}^{-2}A_{2,2}^{-1}A_{2,4}^{-2}A_{3,1}^{-1}A_{3,3}^{-1}.$

We have
\[
\begin{array}{l}
\medskip
  2m A_{1,1}^{-1}A_{1,3}^{-1}(A_{3,3}^{-1}A_{2,4}^{-2}+2A_{3,3}^{-1}A_{2,4}^{-1})
  \leq\varphi_{\{2,3\}}(2m A_{1,1}^{-1}A_{1,3}^{-1})_{\leq4}, \\
  \medskip
  2mA_{1,3}^{-1}A_{2,4}^{-1}+mA_{1,3}^{-2}A_{2,4}^{-2}+2m A_{1,1}^{-1}A_{1,3}^{-1}A_{2,4}^{-1}
   \leq\varphi_{\{1,2\}}(m)_{\leq4}, \\
   \medskip
  mA_{3,3}^{-1}(A_{2,4}^{-1}+4A_{1,3}^{-1}A_{2,4}^{-1}+2A_{1,3}^{-1}A_{2,4}^{-2}+3A_{1,3}^{-2}A_{2,4}^{-2})
   \leq\varphi_{\{1,2\}}(mA_{3,3}^{-1})_{\leq4}, \\
   mA_{3,1}^{-1}A_{3,3}^{-1}(A_{1,1}^{-1}A_{1,3}^{-1}A_{2,2}^{-1}A_{2,4}^{-1}+A_{1,1}^{-1}A_{1,3}^{-1}A_{2,2}^{-1}A_{2,4}^{-2}
  +2m A_{1,1}^{-1}A_{1,3}^{-2}A_{2,2}^{-1}A_{2,4}^{-2})
   \leq\varphi_{\{1,2\}}(mA_{3,1}^{-1}A_{3,3}^{-1})_{\leq4}.
\end{array}
\]

If $\beta=\a_0+2\a_1+3\a_2+2\a_3+\a_4+\a_5$, $Y^{\beta}=Y_{1,0}Y_{1,2}Y_{2,3}Y_{2,5}^{2}Y_{3,0}Y_{3,2}$, denote $m=Y^{\a}Y^{\beta}$. The only dominant monomials in $\chi_q(S(\a)\otimes S(\beta))_{\leq4}$ are $m,\ mA_{1,3}^{-1}A_{2,4}^{-1},\ m A_{1,1}^{-1}A_{1,3}^{-1}A_{2,4}^{-1},\ 2mA_{1,3}^{-1}A_{3,3}^{-1}A_{2,4}^{-1},\ \\
mA_{1,3}^{-2}A_{3,3}^{-1}A_{2,4}^{-2}$ and $2m A_{1,1}^{-1}A_{1,3}^{-1}A_{3,3}^{-1}A_{2,4}^{-1}.$

We have
\[
\begin{array}{l}
\medskip
  mA_{1,3}^{-1}A_{2,4}^{-1}+m A_{1,1}^{-1}A_{1,3}^{-1}A_{2,4}^{-1}
  \leq\varphi_{\{1,2\}}(m)_{\leq4}, \\
  mA_{3,3}^{-1}(2A_{1,3}^{-1}A_{2,4}^{-1}+A_{1,3}^{-2}A_{2,4}^{-2}+2A_{1,1}^{-1}A_{1,3}^{-1}A_{2,4}^{-1})
  \leq\varphi_{\{1,2\}}(mA_{3,3}^{-1})_{\leq4}.
\end{array}
\]

If $\beta=\a_0+\a_1+\a_2+\a_3+\a_4+\a_5$, $Y^{\beta}=Y_{2,3}$, denote $m=Y^{\a}Y^{\beta}$. The only dominant monomials in $\chi_q(S(\a)\otimes S(\beta))_{\leq4}$ are $m,\ mA_{2,4}^{-1}$ and $2mA_{1,3}^{-1}A_{2,4}^{-1}.$

We have  $mA_{2,4}^{-1}+ 2mA_{1,3}^{-1}A_{2,4}^{-1}\leq\varphi_{\{1,2\}}(m)_{\leq4}.$

\medskip
(24) $\a=\a_0+2\a_1+2\a_2+2\a_3+\a_4+\a_5$, $Y^{\a}=Y_{1,0}Y_{2,3}^{2}Y_{2,5}Y_{3,0}$. In this case, $\chi_q(L(Y^{\a}))$ has $3$ dominant monomials, so we need to consider every $\beta\in C_\a.$

If $\beta\in \C_\a\backslash \{\a_0+2\a_1+2\a_2+2\a_3+\a_4+\a_5,\ \a_0+\a_1+\a_2+\a_3+\a_4+\a_5\}$, we have done.

If $\beta=\a_0+2\a_1+2\a_2+2\a_3+\a_4+\a_5$, $Y^{\beta}=Y^{\a}$, denote $m=Y^{\a}Y^{\beta}$. The only dominant monomials in $\chi_q(S(\a)\otimes S(\beta))_{\leq4}$ are $m,\ 2mA_{2,4}^{-1},\ mA_{2,4}^{-2},\
2mA_{1,1}^{-1}A_{3,1}^{-1}A_{2,2}^{-1}A_{2,4}^{-1},\  2mA_{1,1}^{-1}A_{3,1}^{-1}A_{2,2}^{-1}A_{2,4}^{-2}$ and
$mA_{1,1}^{-2}A_{3,1}^{-2}A_{2,2}^{-2}A_{2,4}^{-2}.$

We have

\noindent\medskip
  $2mA_{2,4}^{-1}+ mA_{2,4}^{-2}\leq\varphi_{2}(m)_{\leq4},$ \quad
  $mA_{1,1}^{-2}A_{3,1}^{-2}A_{2,2}^{-2}A_{2,4}^{-2}\leq\varphi_{\{2,3\}}(mA_{1,1}^{-2})_{\leq4},$\\
 $2mA_{1,1}^{-1}(A_{3,1}^{-1}A_{2,2}^{-1}A_{2,4}^{-1}+A_{3,1}^{-1}A_{2,2}^{-1}A_{2,4}^{-2})
   \leq\varphi_{\{2,3\}}(2mA_{1,1}^{-1})_{\leq4}.$

If $\beta=\a_0+\a_1+\a_2+\a_3+\a_4+\a_5$, $Y^{\beta}=Y_{2,3}$, denote $m=Y^{\a}Y^{\beta}$. The only dominant monomials in $\chi_q(S(\a)\otimes S(\beta))_{\leq4}$ are $m,\ 2mA_{2,4}^{-1}$ and $mA_{1,1}^{-1}A_{3,1}^{-1}A_{2,2}^{-1}A_{2,4}^{-1}.$

We have   $2mA_{2,4}^{-1}\leq\varphi_{2}(m)_{\leq4}$ and $mA_{1,1}^{-1}A_{3,1}^{-1}A_{2,2}^{-1}A_{2,4}^{-1}\leq\varphi_{2}(mA_{1,1}^{-1}A_{3,1}^{-1})_{\leq4}.$

\medskip
(25) $\a=\a_0+\a_1+\a_2+\a_3+\a_4+\a_5$, $Y^{\a}=Y_{2,3}$. In this case, for $\beta\in \C_\a\backslash \{\a_0+2\a_1+3\a_2+2\a_3+\a_4+2\a_5\}$, we have done.
So we need to consider $\beta=\a_0+2\a_1+3\a_2+2\a_3+\a_4+2\a_5$.

Now, $Y^{\beta}=Y_{1,0}Y_{1,2}Y_{2,3}Y_{2,5}Y_{3,0}Y_{3,2}$, denote $m=Y^{\a}Y^{\beta}$. The only dominant monomials in $\chi_q(S(\a)\otimes S(\beta))$ are
$m, mA_{2,4}^{-1}, 2mA_{1,3}^{-1}A_{2,4}^{-1},
2mA_{3,3}^{-1}A_{2,4}^{-1},\ 3mA_{1,3}^{-1}A_{3,3}^{-1}A_{2,4}^{-1}$ and $mA_{1,1}^{-1}A_{1,3}^{-1}A_{3,1}^{-1}A_{3,3}^{-1}A_{2,2}^{-1}A_{2,4}^{-1}.$

We have

\noindent\medskip
 $ mA_{2,4}^{-1}+2mA_{1,3}^{-1}A_{2,4}^{-1}\leq\varphi_{\{1,2\}}(m)_{\leq4},$ \quad $mA_{3,3}^{-1}(2A_{2,4}^{-1}+3A_{1,3}^{-1}A_{2,4}^{-1})\leq\varphi_{\{1,2\}}(mA_{3,3}^{-1})_{\leq4},$\\
$mA_{1,1}^{-1}A_{1,3}^{-1}A_{3,1}^{-1}A_{3,3}^{-1}A_{2,2}^{-1}A_{2,4}^{-1}
  \leq\varphi_{2}(mA_{1,1}^{-1}A_{1,3}^{-1}A_{3,1}^{-1}A_{3,3}^{-1})_{\leq4}.$

\medskip
(26) $\a=\a_0+2\a_1+3\a_2+2\a_3+\a_4+2\a_5$, $Y^{\a}=Y_{1,0}Y_{1,2}Y_{2,3}Y_{2,5}Y_{3,0}Y_{3,2}$. In this case, for $\beta\in \C_\a\backslash \{\a_0+2\a_1+3\a_2+2\a_3+\a_4+2\a_5\}$, we have done.
So we need to consider $\beta=\a_0+2\a_1+3\a_2+2\a_3+\a_4+2\a_5$.

Since $Y^{\beta}=Y^{\a}$, denote $m=Y^{2\a}$. Up to the symmetries $1\leftrightarrow 3$ in $I$, we reduce to the following dominant monomials of $\chi_q(S(\a)\otimes S(\beta))_{\leq4}$:
$m,\ 2mA_{1,3}^{-1}A_{2,4}^{-1},\ 2mA_{1,1}^{-1}A_{1,3}^{-1}A_{2,4}^{-1},\
8mA_{1,3}^{-1}A_{3,3}^{-1}A_{2,4}^{-1},\ \\
4mA_{1,3}^{-1}A_{3,3}^{-1}A_{2,4}^{-2},
6mA_{1,3}^{-2}A_{3,3}^{-2}A_{2,4}^{-2},\ 6mA_{1,3}^{-1}A_{3,3}^{-2}A_{2,4}^{-2},
mA_{1,3}^{-2}A_{2,4}^{-2},
8mA_{1,1}^{-1}A_{1,3}^{-1}A_{3,3}^{-1}A_{2,4}^{-1},
6mA_{1,1}^{-1}A_{1,3}^{-1}A_{3,3}^{-2}A_{2,4}^{-2}, \\
8mA_{1,1}^{-1}A_{1,3}^{-1}A_{3,1}^{-1}A_{3,3}^{-1}A_{2,4}^{-1},\
2mA_{1,1}^{-1}A_{1,3}^{-1}A_{3,1}^{-1}A_{3,3}^{-1}A_{2,2}^{-1}A_{2,4}^{-1},\
4mA_{1,1}^{-1}A_{1,3}^{-1}A_{3,1}^{-1}A_{3,3}^{-1}A_{2,4}^{-2},\ \\
2mA_{1,1}^{-1}A_{1,3}^{-1}A_{3,1}^{-1}A_{3,3}^{-1}A_{2,2}^{-1}A_{2,4}^{-2},\
4mA_{1,1}^{-1}A_{1,3}^{-1}A_{3,1}^{-1}A_{3,3}^{-2}A_{2,2}^{-1}A_{2,4}^{-2},\
6mA_{1,1}^{-1}A_{1,3}^{-2}A_{3,1}^{-1}A_{3,3}^{-2}A_{2,2}^{-1}A_{2,4}^{-2}.$

We have

\noindent \medskip
$2mA_{1,3}^{-1}A_{2,4}^{-1}+ mA_{1,3}^{-2}A_{2,4}^{-2}+2mA_{1,1}^{-1}A_{1,3}^{-1}A_{2,4}^{-1}
\leq\varphi_{\{1,2\}}(m)_{\leq4},$\\
\medskip
$2mA_{1,3}^{-1}(4A_{3,3}^{-1}A_{2,4}^{-1}+2A_{3,3}^{-1}A_{2,4}^{-2}+3A_{3,3}^{-2}A_{2,4}^{-2})
\leq\varphi_{\{2,3\}}(2mA_{1,3}^{-1})_{\leq4},$\\
\smallskip
$2mA_{1,1}^{-1}A_{1,3}^{-1}(4A_{3,3}^{-1}A_{2,4}^{-1}+2A_{3,3}^{-1}A_{2,4}^{-2}+3A_{3,3}^{-2}A_{2,4}^{-2}+
4A_{3,1}^{-1}A_{3,3}^{-1}A_{2,4}^{-1}+2A_{3,1}^{-1}A_{3,3}^{-1}A_{2,4}^{-2}+
A_{3,1}^{-1}A_{3,3}^{-1}A_{2,2}^{-1}A_{2,4}^{-1}+\\
\medskip
A_{3,1}^{-1}A_{3,3}^{-1}A_{2,2}^{-1}A_{2,4}^{-2}+2A_{3,1}^{-1}A_{3,3}^{-2}A_{2,2}^{-1}A_{2,4}^{-2})
\leq\varphi_{\{2,3\}}(2mA_{1,1}^{-1}A_{1,3}^{-1})_{\leq4},$\\
$6mA_{1,3}^{-2}A_{3,3}^{-2}A_{2,4}^{-2}\leq\varphi_{\{2,3\}}(mA_{1,3}^{-2})_{\leq4}$, \quad
$6mA_{1,1}^{-1}A_{1,3}^{-2}A_{3,1}^{-1}A_{3,3}^{-2}A_{2,2}^{-1}A_{2,4}^{-2}
 \leq\varphi_{\{2,3\}}(2mA_{1,1}^{-1}A_{1,3}^{-2})_{\leq4}.$

\medskip
(27) $\a=\a_0+2\a_1+3\a_2+2\a_3+\a_4+\a_5$, $Y^{\a}=Y_{1,0}Y_{1,2}Y_{2,3}Y_{2,5}^{2}Y_{3,0}Y_{3,2}$. In this case, for $\beta\in \C_\a\backslash \{\a_0+2\a_1+3\a_2+2\a_3+\a_4+2\a_5\}$, we have done.
So we only need to consider $\beta=\a_0+2\a_1+3\a_2+2\a_3+\a_4+\a_5$.

Since $Y^{\beta}=Y^{\a}$, denote $m=Y^{2\a}$. Up to the symmetries $1\leftrightarrow 3$ in $I$, we reduce to the following dominant monomials of $\chi_q(S(\a)\otimes S(\beta))_{\leq4}$:
$m,\ 2mA_{1,3}^{-1}A_{3,3}^{-1}A_{2,4}^{-1},\
mA_{1,3}^{-2}A_{3,3}^{-2}A_{2,4}^{-2},\
2mA_{1,1}^{-1}A_{1,3}^{-1}A_{3,3}^{-1}A_{2,4}^{-1}$ and
$2mA_{1,1}^{-1}A_{1,3}^{-1}A_{3,1}^{-1}A_{3,3}^{-1}A_{2,4}^{-1}.$
\[
\begin{array}{ll}
2mA_{1,3}^{-1}A_{3,3}^{-1}A_{2,4}^{-1}\leq\varphi_{\{2,3\}}(2mA_{1,3}^{-1})_{\leq4},&
mA_{1,3}^{-2}A_{3,3}^{-2}A_{2,4}^{-2}\leq\varphi_{\{2,3\}}(mA_{1,3}^{-2})_{\leq4},\\
2mA_{1,1}^{-1}A_{1,3}^{-1}(A_{3,1}^{-1}A_{3,3}^{-1}A_{2,4}^{-1}+A_{3,3}^{-1}A_{2,4}^{-1})
\leq\varphi_{\{2,3\}}(2mA_{1,1}^{-1}A_{1,3}^{-1})_{\leq4}.
\end{array}
\]

\medskip
(28) $\a=\a_0+\a_1+2\a_2+\a_3+\a_4+\a_5$, $Y^{\a}=Y_{1,2}Y_{2,5}Y_{3,2}$. In this case, for $\beta\in \C_\a\backslash \{\a_0+\a_1+2\a_2+\a_3+\a_4+\a_5\}$, we have done.
So we only need to consider $\beta=\a_0+\a_1+2\a_2+\a_3+\a_4+\a_5$.

Since $Y^{\beta}=Y^{\a}$, denote $m=Y^{2\a}$. The dominant monomials of $\chi_q(S(\a)\otimes S(\beta))_{\leq4}$ are $m,\ 2mA_{1,3}^{-1}A_{3,3}^{-1}A_{2,4}^{-1}$ and $mA_{1,3}^{-2}A_{3,3}^{-2}A_{2,4}^{-2}$.
\[
\begin{array}{ll}
2mA_{1,3}^{-1}A_{3,3}^{-1}A_{2,4}^{-1}\leq\varphi_{\{2,3\}}(2mA_{1,3}^{-1})_{\leq4}, &
mA_{1,3}^{-2}A_{3,3}^{-2}A_{2,4}^{-2}\leq\varphi_{2}(mA_{1,3}^{-2}A_{3,3}^{-2})_{\leq4}.
\end{array}
\]

This finishes the proof of Theorem~\ref{monoicat}.
\end{proof}

\begin{corollary}
For every exchange relation of cluster algebra $\A$, there exist an exact sequence of $\CC_2$ corresponding to it.
\end{corollary}

\begin{example}
For exchange relation
\[
x[-\a_0]x[\a_0]=x[-\a_1]+x_7,
\]
the corresponding exact sequence is
\[
0 \rightarrow L(Y_{1,4}Y_{2,1}) \rightarrow L(Y_{1,0})\otimes L(Y_{1,2}Y_{1,4}) \rightarrow L(Y_{1,0}Y_{1,2}Y_{1,4}) \rightarrow 0.
\]

For exchange relation
\[
\begin{array}{c}
x[-\a_2]x[\a_0+\a_1+\a_2]=x[\a_0+\a_1]x_6+x[\a_0]x[-\a_3]x[-\a_5],
\end{array}
\]
the corresponding exact sequence is
\[
\begin{array}{c}
0 \rightarrow L(Y_{3,4})\otimes L(Y_{2,1}Y_{2,3}Y_{2,5}) \rightarrow L(Y_{1,4}Y_{2,1}Y_{3,4})\otimes L(Y_{1,2}Y_{2,5}) \rightarrow L(Y_{1,2}Y_{1,4})\otimes L(Y_{2,5})\otimes L(Y_{2,1}Y_{3,4}) \rightarrow 0.
\end{array}
\]
Further more,
\[
\begin{array}{c}
  L(Y_{3,4})\otimes L(Y_{2,1}Y_{2,3}Y_{2,5})\cong L(Y_{2,1}Y_{2,3}Y_{2,5}Y_{3,4}), \\
  \\
  L(Y_{1,2}Y_{1,4})\otimes L(Y_{2,5})\otimes L(Y_{2,1}Y_{3,4})\cong L(Y_{1,2}Y_{1,4}Y_{2,1}Y_{2,5}Y_{3,4}).
\end{array}
\]
\end{example}

\begin{appendix}
\section{Explicit Formulas for Cluster Variables}
The following result is obtained by computer.
We list the sequences of mutations and explicit formulas for cluster variables:

\noindent $x\big[-\alpha_{0}\big]=x_0,$ \qquad  $x\big[-\alpha_{1}\big]=x_1,$ \qquad $ x\big[-\alpha_{2}\big]=x_2,$\\
$x\big[-\alpha_{3}\big]=x_3,$ \qquad  $x\big[-\alpha_{4}\big]=x_4,$  \qquad  $ x\big[-\alpha_{5}\big]=x_5,$\\
$\mu_{0}: \ x\big[\alpha_{0}\big]=(x_1 + x_7)/x_0,$ \qquad  \qquad $\mu_4: \ x\big[\alpha_{4}\big]=(x_3 + x_8)/x_4,$\\
$\mu_1: \ x\big[\alpha_{1}\big]=(x_0 x_2 + x_7)/x_1,$ \qquad \quad $\mu_3: x\big[\alpha_{3}\big]=(x_2 x_4 + x_8)/x_3,$\\
$\mu_2: \ x\big[\alpha_{2}\big]=(x_1 x_3 x_5 + x_6)/x_2,$ \qquad $\mu_5: \ x\big[\alpha_{5}\big]=(x_2 + x_6)/x_5,$\\
$\mu_{1}\mu_{0}:
x\big[\alpha_{0}+\alpha_{1}\big]=\frac{1}{x_0x_1}(x_1 + x_0x_2 + x_7),$\\
$\mu_3\mu_4:
x\big[\alpha_3+\alpha_{4}\big]=\frac{1}{x_3 x_4}(x_3 + x_2 x_4 + x_8),$\\
$\mu_2\mu_1:
x\big[\alpha_{1}+\alpha_{2}\big]=\frac{1}{x_1 x_2}(x_0 x_2 x_6 + x_1 x_3 x_5 x_7 + x_6 x_7),$\\
$\mu_3\mu_2:
x\big[\alpha_{2}+\alpha_{3}\big]=\frac{1}{x_2 x_3}(x_2 x_4 x_6 + x_1 x_3 x_5 x_8 + x_6 x_8),$\\
$\mu_5\mu_2:
x\big[\alpha_{2}+\alpha_{5}\big]=\frac{1}{x_2 x_5}(x_2 + x_1 x_3 x_5 + x_6),$\\
$\mu_2\mu_1\mu_0:
x\big[\alpha_0+\alpha_1+\alpha_{2}\big]=\frac{1}{x_0 x_1 x_2}(x_1^2 x_3 x_5 + x_1 x_6 + x_0 x_2 x_6 + x_1 x_3 x_5 x_7 + x_6 x_7),$\\
$\mu_2\mu_3\mu_4:
x\big[\alpha_2+\alpha_3+\alpha_{4}\big]=\frac{1}{x_2 x_3 x_4}(x_1 x_3 x_5 (x_3 + x_8) + x_6 (x_3 + x_2 x_4 + x_8)),$\\
$\mu_3\mu_2\mu_1:
x\big[\alpha_{1}+\alpha_{2}+\alpha_{3}\big]=\frac{1}{x_1 x_2 x_3}(x_0 x_2 x_6 (x_2 x_4 + x_8) + x_7 (x_2 x_4 x_6 + (x_1 x_3 x_5 + x_6) x_8)),$\\
$\mu_5\mu_2\mu_3:
x\big[\alpha_{2}+\alpha_{3}+\alpha_{5}\big]=\frac{1}{x_2 x_3 x_5}(x_2^2 x_4 + x_2 x_4 x_6 + (x_2 + x_1 x_3 x_5 + x_6) x_8),$\\
$\mu_5\mu_2\mu_1:
x\big[\alpha_{1}+\alpha_{2}+\alpha_{5}\big]=\frac{1}{x_1 x_2 x_5}(x_0 x_2 (x_2 + x_6) + (x_2 + x_1 x_3 x_5 + x_6) x_7),$\\
$\mu_4\mu_3\mu_2\mu_1:\\
x\big[\alpha_{1}+\alpha_{2}+\alpha_{3}+\alpha_{4}\big]=\frac{1}{x_1 x_2 x_3 x_4}(x_0 x_2 x_6 (x_3 + x_2 x_4 + x_8) + x_7 (x_1 x_3 x_5 (x_3 + x_8) + x_6 (x_3 + x_2 x_4 + x_8))),$\\
$\mu_3\mu_2\mu_1\mu_0:\\
x\big[\alpha_{0}+\alpha_{1}+\alpha_{2}+\alpha_{3}\big]=\frac{1}{x_0 x_1 x_2 x_3}(x_1^2 x_3 x_5 x_8 + x_6 (x_0 x_2 + x_7) (x_2 x_4 + x_8) + x_1 (x_2 x_4 x_6 + (x_6 + x_3 x_5 x_7) x_8)),$\\
$\mu_5\mu_2\mu_3\mu_4:\\
x\big[\alpha_2+\alpha_3+\alpha_{4}+\alpha_5\big]=\frac{1}{x_2 x_3 x_4 x_5}(x_2^2 x_4 + (x_1 x_3 x_5 + x_6) (x_3 + x_8) + x_2 (x_3 + x_4 x_6 + x_8)),$\\
$\mu_5\mu_2\mu_1\mu_0: \\
x\big[\alpha_{0}+\alpha_{1}+\alpha_{2}+\alpha_{5}\big]=\frac{1}{x_0 x_1 x_2 x_5}(x_1^2 x_3 x_5 + (x_2 + x_6) (x_0 x_2 + x_7) + x_1 (x_2 + x_6 + x_3 x_5 x_7)),$\\
$\mu_3\mu_2\mu_1\mu_5:\\
x\big[\alpha_1+\alpha_2+\alpha_{3}+\alpha_5\big]=\frac{1}{x_1 x_2 x_3 x_5}(x_0 x_2 (x_2 + x_6) (x_2 x_4 + x_8) + x_7 (x_2^2 x_4 + (x_1 x_3 x_5 + x_6) x_8 + x_2 (x_4 x_6 + x_8))),$\\
$\mu_1\mu_5\mu_4\mu_3\mu_2\mu_1:\\
x\big[\alpha_{1}+2\alpha_{2}+\alpha_{3}+\alpha_{5}\big]=\frac{1}{x_1 x_2^2 x_3 x_5}((x_2 + x_1 x_3 x_5
+ x_6) x_7 (x_2 x_4 x_6 + (x_1 x_3 x_5 + x_6) x_8) + x_0 x_2 x_6 (x_2^2 x_4 +(x_1 x_3 x_5+ x_6) x_8 + x_2(x_4 x_6 + x_8))),$\\
$\mu_4\mu_3\mu_2\mu_1\mu_0:\\
x\big[\alpha_{0}+\alpha_{1}+\alpha_{2}+\alpha_{3}+\alpha_{4}\big]=\frac{1}{x_0 x_1 x_2 x_3 x_4}(x_1^2 x_3 x_5 (x_3 + x_8) +  x_6 (x_0 x_2 + x_7) (x_3 + x_2 x_4 + x_8) +x_1(x_3^2 x_5 x_7 + x_6(x_2 x_4 +  x_8) + x_3 (x_6 + x_5 x_7 x_8))),$\\
$\mu_4\mu_3\mu_2\mu_1\mu_5:\\
x\big[\alpha_1+\alpha_2+\alpha_{3}+\alpha_4+\alpha_5\big]=\frac{1}{x_1 x_2 x_3 x_4 x_5}(x_0 x_2 (x_2 + x_6) (x_3 + x_2 x_4 + x_8) + x_7 (x_2^2 x_4 + (x_1 x_3 x_5 + x_6)(x_3 + x_8) + x_2 (x_3+ x_4 x_6 + x_8))),$\\
$\mu_0\mu_1\mu_4\mu_3\mu_2\mu_1\mu_5:\\
x\big[\alpha_0+\alpha_1+\alpha_2+\alpha_{3}+\alpha_5\big]=\frac{1}{x_0 x_1 x_2 x_3 x_5}(x_1^2 x_3 x_5 x_8 + (x_2 + x_6) (x_0 x_2 + x_7) (x_2 x_4 + x_8) + x_1 (x_2^2 x_4 +(x_6 + x_3 x_5 x_7) x_8 + x_2 (x_4 x_6 + x_8))),$\\
$\mu_0\mu_1\mu_5\mu_4\mu_3\mu_2\mu_1:\\
x\big[\alpha_{0}+2\alpha_{1}+2\alpha_{2}+\alpha_{3}+\alpha_{5}\big]=\frac{1}{x_0 x_1^2 x_2^2 x_3 x_5}(x_0^2 x_2^2 x_6 (x_2 + x_6) (x_2 x_4 + x_8) + (x_2 + x_1 x_3 x_5 + x_6) x_7(x_1+x_7) (x_2 x_4 x_6+
(x_1 x_3 x_5 + x_6) x_8) + x_0 x_2 (x_1^2 x_3 x_5 x_6 x_8 + 2 x_6 (x_2 + x_6) x_7 (x_2 x_4 + x_8) +
 x_1 (x_2^2 x_4 x_6 + x_6 (x_6 +2 x_3 x_5 x_7) x_8 +x_2 (x_6 + x_3 x_5 x_7) (x_4 x_6 + x_8)))),$\\
$\mu_2\mu_0\mu_1\mu_5\mu_4\mu_3\mu_2\mu_1:\\
x\big[\alpha_{1}+2\alpha_{2}+2\alpha_{3}+\alpha_{4}+\alpha_{5}\big]=\frac{1}{x_1 x_2^2 x_3^2 x_4 x_5}(x_0 x_2 x_6 (x_3 + x_2 x_4 + x_8) (x_2^2 x_4 + (x_1 x_3 x_5 + x_6) x_8 + x_2(x_4 x_6+ x_8))+
x_7 (x_2^3 x_4^2 x_6 + (x_1 x_3 x_5 + x_6)^2 x_8 (x_3 + x_8) + x_2 (x_1 x_3 x_5 + x_6) (x_3 (x_4 x_6 + x_8) + x_8(2 x_4 x_6 + x_8)) +x_2^2 x_4 (x_6 (x_4 x_6 + 2 x_8) + x_3 (x_6 + x_1 x_5 x_8)))),$\\
$\mu_4\mu_2\mu_4\mu_0\mu_3\mu_5\mu_1\mu_2:\\
x\big[\alpha_0+\alpha_1+2\alpha_2+\alpha_{3}+\alpha_5\big]=\frac{1}{x_0 x_1 x_2^2 x_3 x_5}(x_1^3 x_3^2 x_5^2 x_8 + x_6 (x_2 + x_6) (x_0 x_2 + x_7) (x_2 x_4 + x_8) + x_1^2 x_3 x_5(x_2 x_4 x_6+x_2 x_8
+2 x_6 x_8 + x_3 x_5 x_7 x_8) + x_1 (x_2^2 x_4 x_6 + x_6 (x_6 + 2 x_3 x_5 x_7) x_8 + x_2 (x_4 x_6 (x_6 + x_3 x_5 x_7) +(x_6 + x_0 x_3 x_5 x_6 + x_3 x_5 x_7) x_8))),$\\
$\mu_0\mu_2\mu_4\mu_0\mu_3\mu_5\mu_1\mu_2:\\
x\big[\alpha_1+2\alpha_2+\alpha_{3}+\alpha_4+\alpha_5\big]=\frac{1}{x_1 x_2^2 x_3 x_4 x_5}((x_2 + x_1 x_3 x_5 + x_6) x_7 (x_1 x_3 x_5 (x_3 + x_8) + x_6 (x_3 + x_2 x_4 + x_8))+ x_0 x_2 x_6
(x_2^2 x_4 + (x_1 x_3 x_5 + x_6) (x_3 + x_8) + x_2 (x_3 + x_4 x_6 + x_8))),$\\
$\mu_1\mu_2\mu_3\mu_2\mu_0\mu_1\mu_5\mu_4\mu_3\mu_2\mu_1:\\
x\big[\alpha_{0}+\alpha_{1}+\alpha_{2}+\alpha_{3}+\alpha_{4}+\alpha_{5}\big]=\frac{1}{x_0 x_1 x_2 x_3 x_4 x_5}(x_1^2 x_3 x_5 (x_3 + x_8) + (x_2 + x_6) (x_0 x_2 + x_7) (x_3 + x_2 x_4 + x_8) +
 x_1 (x_2^2 x_4 + (x_6 + x_3 x_5 x_7) (x_3 + x_8) + x_2 (x_3 + x_4 x_6 + x_8))),$\\
$\mu_2\mu_4\mu_0\mu_3\mu_5\mu_1\mu_2:\\
x\big[\alpha_0+\alpha_1+2\alpha_2+\alpha_{3}+\alpha_4+\alpha_5\big]=\frac{1}{x_0 x_1 x_2^2 x_3 x_4 x_5}(x_1^3 x_3^2 x_5^2 (x_3 + x_8) + x_6 (x_2 + x_6) (x_0 x_2 + x_7) (x_3 + x_2 x_4 +x_8)+
x_1^2 x_3 x_5 ((2 x_6 + x_3 x_5 x_7) (x_3 + x_8) + x_2 (x_3 + x_4 x_6 + x_8)) + x_1 (x_2^2 x_4 x_6 + x_6 (x_6 + 2 x_3 x_5 x_7) (x_3 + x_8) +x_2 (x_3^2 x_5 (x_0 x_6 +
x_7) + x_6 (x_4 x_6 + x_8) + x_3 (x_6 + x_4 x_5 x_6 x_7 + x_0 x_5 x_6 x_8 + x_5 x_7 x_8)))),$\\
$\mu_2\mu_3\mu_2\mu_0\mu_1\mu_5\mu_4\mu_3\mu_2\mu_1:\\
x\big[\alpha_{0}+2\alpha_{1}+2\alpha_{2}+\alpha_{3}+\alpha_{4}+\alpha_{5}\big]=\frac{1}{x_0 x_1^2 x_2^2 x_3 x_4 x_5}(x_0^2 x_2^2 x_6 (x_2 + x_6) (x_3 + x_2 x_4 + x_8) + (x_2 + x_1 x_3 x_5+ x_6)
x_7 (x_1 +
 x_7) (x_1 x_3 x_5 (x_3 + x_8) + x_6 (x_3 + x_2 x_4 + x_8)) + x_0 x_2 (x_1^2 x_3 x_5 x_6 (x_3 + x_8) + 2 x_6 (x_2 + x_6) x_7 (x_3 + x_2 x_4+ x_8)+ x_1
(x_2^2 x_4 x_6 + x_6 (x_6 + 2 x_3 x_5 x_7) (x_3 + x_8) + x_2 (x_6 + x_3 x_5 x_7) (x_3 + x_4 x_6 + x_8)))),$\\
$\mu_2\mu_1\mu_2\mu_4\mu_3\mu_5\mu_0\mu_1\mu_2\mu_3:\\
x\big[\alpha_0+\alpha_1+2\alpha_2+2\alpha_{3}+\alpha_4+\alpha_5\big]=\frac{1}{x_0 x_1 x_2^2 x_3^2 x_4 x_5}(x_1^3 x_3^2 x_5^2 x_8 (x_3 + x_8) +  x_6 (x_2 + x_6) (x_0 x_2 + x_7) (x_2 x_4 + x_8)(x_3 +
x_2 x_4 + x_8) + x_1^2 x_3 x_5 (x_2^2 x_4 x_8 + (2 x_6 + x_3 x_5 x_7) x_8 (x_3 + x_8) + x_2 (x_3 (x_4 x_6 + x_8) + x_8 (2 x_4 x_6 + x_8))) + x_1(x_2^3 x_4^2 x_6 +
x_6 (x_6 + 2 x_3 x_5 x_7) x_8 (x_3 + x_8) + x_2^2 x_4 (x_6 (x_4 x_6 + 2 x_8) + x_3 (x_6 + x_0 x_5 x_6 x_8 + x_5 x_7 x_8)) + x_2 (x_6 x_8(2 x_4 x_6 +x_8) +
x_3^2 x_5 (x_4 x_6 x_7 + (x_0 x_6 + x_7) x_8) + x_3 (x_8 (x_6 + x_0 x_5 x_6 x_8 + x_5 x_7 x_8) + x_4 x_6 (x_6 + 2 x_5 x_7 x_8))))),$\\
$\mu_3\mu_5\mu_0\mu_1\mu_4\mu_3\mu_2\mu_1\mu_5:\\
x\big[\alpha_0+2\alpha_1+2\alpha_2+2\alpha_{3}+\alpha_4+\alpha_5\big]=\frac{1}{x_0 x_1^2 x_2^2 x_3^2 x_4 x_5} (x_0^2 x_2^2 x_6 (x_2 + x_6) (x_2 x_4 + x_8) (x_3 + x_2 x_4 + x_8) + x_7 (x_1 +
x_7)(x_2^3 x_4^2 x_6+ (x_1 x_3 x_5 + x_6)^2 x_8 (x_3 + x_8) + x_2 (x_1 x_3 x_5 + x_6) (x_3 (x_4 x_6 + x_8) + x_8 (2 x_4 x_6 + x_8)) + x_2^2 x_4 (x_6 (x_4 x_6 +
2 x_8) + x_3(x_6 + x_1 x_5 x_8)))+ x_0 x_2 (x_1^2 x_3 x_5 x_6 x_8 (x_3 + x_2 x_4 + x_8) + 2 x_6 (x_2 + x_6) x_7 (x_2 x_4 + x_8) (x_3 + x_2 x_4 + x_8)+
x_1(x_2^3 x_4^2 x_6 +x_6 (x_6 + 2 x_3 x_5 x_7) x_8 (x_3 + x_8) + x_2 (x_6 + x_3 x_5 x_7) (x_3 (x_4 x_6 + x_8) + x_8 (2 x_4 x_6 + x_8)) + x_2^2 x_4 (x_6(x_4 x_6 +
2 x_8) +x_3 (x_6 + x_5 x_7 x_8))))),$\\
$\mu_5\mu_2\mu_1\mu_0\mu_4\mu_2\mu_4\mu_0\mu_3\mu_5\mu_1\mu_2:\\
x\big[\alpha_0+2\alpha_1+3\alpha_2+2\alpha_{3}+\alpha_4+\alpha_5\big]=\frac{1}{x_0 x_1^2 x_2^3 x_3^2 x_4 x_5} (x_0^2 x_2^2 x_6^2 (x_3 + x_2 x_4 + x_8) (x_2^2 x_4 + (x_1 x_3 x_5 + x_6) x_8 + x_2(x_4 x_6 +
x_8))+ (x_2 + x_1 x_3 x_5 + x_6) x_7 (x_1 + x_7) (x_1^2 x_3^2 x_5^2 x_8 (x_3 + x_8) + x_6^2 (x_2 x_4 + x_8) (x_3 + x_2 x_4 + x_8) +x_1 x_3 x_5 x_6 (2 x_8 (x_3 +
x_8) + x_2 x_4 (x_3 + 2 x_8)))+ x_0 x_2 x_6 (x_1^3 x_3^2 x_5^2 x_8 (x_3 + x_8) + 2 x_6 (x_2 + x_6) x_7 (x_2 x_4 + x_8) (x_3 +x_2 x_4 + x_8) +x_1^2 x_3 x_5
(x_2^2 x_4 x_8 +2 (x_6 + x_3 x_5 x_7) x_8 (x_3 + x_8)+ x_2 (x_8 (2 x_4 x_6 + x_8) + x_3 (x_4 x_6 + x_8 + x_4 x_5 x_7 x_8)))
+x_1 (x_2^3 x_4^2 x_6 + x_6(x_6 +
4 x_3 x_5 x_7) x_8(x_3 + x_8) + x_2^2 x_4 (x_6 + x_3 x_5 x_7) (x_3 + x_4 x_6 + 2 x_8)+ x_2 (2 x_3^2 x_5 x_7 (x_4 x_6 + x_8)+x_6 x_8 (2 x_4 x_6 + x_8) + x_3
 (x_8(x_6 + 2 x_5 x_7 x_8) +x_4 x_6 (x_6 + 5 x_5 x_7 x_8)))))),$\\
$\mu_3\mu_2\mu_0\mu_1\mu_5\mu_4\mu_3\mu_2\mu_1:\\
 x\big[\alpha_{0}+2\alpha_{1}+3\alpha_{2}+2\alpha_{3}+\alpha_{4}+2\alpha_{5}\big]=\frac{1}{x_0 x_1^2 x_2^3 x_3^2 x_4 x_5^2}(x_0^2 x_2^2 x_6 (x_2 + x_6) (x_3 + x_2 x_4 + x_8) (x_2^2 x_4 + (x_1 x_3 x_5 + x_6)x_8+
 x_2 (x_4 x_6 + x_8)) + (x_2 + x_1 x_3 x_5 + x_6) x_7 (x_1 + x_7) (x_2^3 x_4^2 x_6 + (x_1 x_3 x_5 + x_6)^2 x_8 (x_3 + x_8) + x_2 (x_1 x_3 x_5 +x_6)(x_3(x_4 x_6 +
x_8)+ x_8 (2 x_4 x_6 + x_8)) + x_2^2 x_4 (x_6 (x_4 x_6 + 2 x_8) + x_3 (x_6 + x_1 x_5 x_8))) + x_0 x_2 (x_1^3 x_3^2 x_5^2 x_6 x_8 (x_3 + x_8) +2 x_6(x_2 +x_6)^2x_7(x_2 x_4 + x_8)(x_3 + x_2 x_4 + x_8) + x_1^2 x_3 x_5 (2 x_6 (x_6 + x_3 x_5 x_7) x_8 (x_3 + x_8) + x_2^2 x_4 x_6 (x_3 + 2 x_8) + x_2(x_3^2 x_5 x_7 x_8 +2 x_6 x_8
(x_4 x_6 + x_8)+ x_3 (x_4 x_6(x_6+ x_5 x_7 x_8) + x_8 (2 x_6 + x_5 x_7 x_8)))) +x_1 (x_2 + x_6) (x_2^3 x_4^2 x_6 + x_6 (x_6+ 4 x_3 x_5 x_7)x_8 (x_3 + x_8) + x_2^2 x_4 (x_6 (x_4 x_6 + 2 x_8)+ x_3 (x_6 +x_4 x_5 x_6 x_7 + x_5 x_7 x_8)) + x_2 (x_3^2 x_5 x_7 (2 x_4 x_6 + x_8) + x_6 x_8 (2 x_4 x_6 + x_8) +x_3 (x_8 (x_6 + x_5 x_7 x_8) + x_4 x_6 (x_6 + 5 x_5 x_7 x_8)))))).$

\section{Compatible Roots}\label{compatible}

For $\a\in \Phi_{\geq-1}$, denote $C_\a$ be the set of almost positive roots which are compatible with $\a$.
We list them in the following.

\noindent $C_{-\a_{i}}=\{-\a_{i}\} \cup \{\beta\in \Phi_{\geq-1} \mid [\beta : \a_i]=0\}, \ (i\in \widetilde{I}),$\\
$C_{\a_0}=\{-\a_{i} \mid i\in \widetilde{I},i\neq0\} \cup \{\a_0, \ \a_2,\  \a_3,\  \a_4,\  \a_5,\  \a_2+\a_3,\  \a_2+\a_5,\  \a_3+\a_4,\  \a_0+\a_1,\  \a_2+\a_3+\a_4,\  \a_0+
\a_1+\a_2,\a_2+\a_3+\a_5,\ \a_0+\a_1+\a_2+\a_3,\  \a_0+\a_1+\a_2+\a_5,\  \a_2+\a_3+\a_4+\a_5,\  \a_0+\a_1+\a_2+\a_3+
\a_5,\  \a_0+\a_1+2\a_2+\a_3+\a_5,\ \a_0+\a_1+\a_2+\a_3+\a_4,\  \a_0+\a_1+\a_2+\a_3+\a_4+\a_5,\  \a_0+\a_1+2\a_2+2\a_3+
\a_4+\a_5,\  \a_0+\a_1+2\a_2+\a_3+\a_4+\a_5 \}$,\\
$C_{\a_1}=\{-\a_{i} \mid i\in \widetilde{I},i\neq1\} \cup \{\a_{1},\  \a_3,\  \a_4,\  \a_5,\  \a_0+\a_1,\  \a_1+\a_2,\  \a_3+\a_4,\  \a_1+\a_2+\a_3,\  \a_1+\a_2+\a_5,\  \a_1+
\a_2+\a_3+\a_4,\ \a_1+\a_2+\a_3+\a_5,\ \a_1+2\a_2+\a_3+\a_4+\a_5 \}$,\\
$C_{\a_2}=\{-\a_{i} \mid i\in \widetilde{I},i\neq2\} \cup \{\a_0,\ \a_2,\ \a_4,\ \a_2+\a_3,\  \a_2+\a_5,\ \a_2+\a_3+\a_4,\  \a_0+\a_1+\a_2\}$,\\
$C_{\a_5}=\{-\a_{i} \mid i\in \widetilde{I},i\neq5\} \cup \{\a_0,\ \a_1,\ \a_3,\ \a_4,\ \a_5,\ \a_0+\a_1,\  \a_2+\a_5,\ \a_3+\a_4,\ \a_1+\a_2+\a_5,\  \a_2+\a_3+
\a_5,\  \a_0+\a_1+\a_2+\a_5,\ \a_2+\a_3+\a_4+\a_5,\  \a_1+\a_2+\a_3+\a_5,\ \a_0+\a_1+\a_2+\a_3+\a_5,\ \a_1+\a_2+\a_3+
\a_4+\a_5,\ \a_0+\a_1+\a_2+\a_3+\a_4+\a_5\}$,\\
$C_{\a_1+\a_2}=\{-\a_{i} \mid i\in \widetilde{I},i\neq1,2\} \cup \{\a_0,\ \a_1,\ \a_2,\ \a_4,\  \a_0+\a_1,\  \a_2+\a_5,\ \a_2+\a_3,\ \a_1+\a_2,\ \a_1+\a_2+\a_3,\  \a_2+
\a_3+\a_4,\ \a_0+\a_1+\a_2,\ \a_1+\a_2+\a_5,\  \a_1+\a_2+\a_3+\a_4,\ \a_0+\a_1+\a_2+\a_3,\ \a_0+\a_1+\a_2+\a_5,\ \a_1+2\a_2+
a_3+\a_5,\ \a_0+\a_1+2\a_2+\a_3+\a_5,\ \a_0+2\a_1+2\a_2+\a_3+\a_5,\ \a_1+2\a_2+\a_3+\a_4+\a_5,\ \a_0+\a_1+\a_2+
\a_3+\a_4,\ \a_0+2\a_1+2\a_2+\a_3+\a_4+\a_5,\ \a_0+\a_1+2\a_2+\a_3+\a_4+\a_5,\ \a_0+2\a_1+3\a_2+2\a_3+\a_4+\a_5 \}$,\\
$C_{\a_2+\a_5}=\{-\a_{i} \mid i\in \widetilde{I},i\neq2,5\} \cup \{\a_0,\ \a_2,\ \a_4,\ \a_5,\  \a_2+\a_3,\ \a_1+\a_2,\ \a_2+\a_5,\  \a_1+\a_2+\a_5,\  \a_2+\a_3+
\a_4,\ \a_0+\a_1+\a_2,\ \a_2+\a_3+\a_5,\ \a_0+\a_1+\a_2+\a_5,\  \a_2+\a_3+\a_4+\a_5,\ \a_1+2\a_2+\a_3+\a_5,\ \a_0+\a_1+
2\a_2+\a_3+\a_5,\  \a_1+2\a_2+\a_3+\a_4+\a_5,\ \a_0+\a_1+2\a_2+\a_3+\a_4+\a_5\}$,\\
$C_{\a_3+\a_4}=\{-\a_{i} \mid i\in \widetilde{I},i\neq3,4\} \cup \{\a_0,\ \a_1,\ \a_3,\ \a_4,\ \a_5,\  \a_0+\a_1,\  \a_3+\a_4,\ \a_2+\a_3,\  \a_1+\a_2+\a_3,\  \a_2+\a_3+
\a_4,\ \a_2+\a_3+\a_5,\ \a_0+\a_1+\a_2+\a_3,\ \a_1+\a_2+\a_3+\a_4,\ \a_2+\a_3+\a_4+\a_5,\ \a_1+\a_2+\a_3+\a_5,\ \a_0+\a_1+
\a_2+\a_3+\a_5,\ \a_1+2\a_2+2\a_3+\a_4+\a_5,\ \a_0+\a_1+\a_2+\a_3+\a_4,\ \a_1+\a_2+\a_3+\a_4+\a_5,\
\a_0+\a_1+\a_2+
\a_3+\a_4+\a_5,\ \a_0+2\a_1+2\a_2+2\a_3+\a_4+\a_5,\ \a_0+\a_1+2\a_2+2\a_3+\a_4+\a_5\}$,\\
$C_{\a_1+\a_2+\a_3}=\{-\a_0,\ -\a_4,\ -\a_5\} \cup \{ \a_1,\ \a_3,\  \a_0+\a_1,\  \a_1+\a_2,\ \a_3+\a_4,\ \a_2+\a_3,\  \a_1+\a_2+\a_3,\  \a_1+\a_2+
\a_5,\ \a_2+\a_3+\a_5,\ \a_0+\a_1+\a_2+\a_3,\ \a_1+\a_2+\a_3+\a_4,\ \a_1+2\a_2+\a_3+\a_5,\ \a_1+\a_2+\a_3+\a_5,\ \a_0+
\a_1+\a_2+\a_3+\a_5,\ \a_1+\a_2+\a_3+\a_4+\a_5,\ \a_0+2\a_1+2\a_2+\a_3+\a_5,\ \a_1+2\a_2+2\a_3+\a_4+\a_5,\ \a_0+
2\a_1+2\a_2+2\a_3+\a_4+\a_5\}$,\\
$C_{\a_0+\a_1+\a_2}=\{-\a_3,\ -\a_4,\ -\a_5\} \cup \{ \a_0,\ \a_2,\  \a_4,\ \a_0+\a_1,\  \a_1+\a_2,\ \a_2+\a_5,\ \a_2+\a_3,\ \a_0+\a_1+\a_2,\  \a_2+
\a_3+\a_4,\  \a_0+\a_1+\a_2+\a_5,\ \a_0+\a_1+2\a_2+\a_3+\a_5,\
\a_0+\a_1+\a_2+\a_3+\a_4,\ \a_0+\a_1+2\a_2+\a_3+\a_4+
\a_5\}$,\\
$C_{\a_1+\a_2+\a_5}=\{-\a_0,\ -\a_3,\ -\a_4\} \cup \{ \a_1,\ \a_4,\  \a_5,\ \a_0+\a_1,\  \a_1+\a_2,\ \a_2+\a_5,\  \a_1+\a_2+\a_3,\ \a_2+\a_3+
\a_5,\ \a_1+\a_2+\a_5,\ \a_0+\a_1+\a_2+\a_5,\ \a_1+\a_2+\a_3+\a_4,\ \a_2+\a_3+\a_4+\a_5,\ \a_1+2\a_2+\a_3+\a_5,\ \a_1+
\a_2+\a_3+\a_5,\ \a_0+\a_1+\a_2+\a_3+\a_5,\ \a_1+\a_2+\a_3+\a_4+\a_5,\ \a_0+2\a_1+2\a_2+\a_3+\a_5,\ \a_1+2\a_2+\a_3+
\a_4+\a_5,\ \a_1+2\a_2+2\a_3+\a_4+\a_5,\ \a_0+\a_1+\a_2+\a_3+\a_4+\a_5,\ \a_0+2\a_1+2\a_2+2\a_3+\a_4+\a_5,\
\a_0+
2\a_1+2\a_2+\a_3+\a_4+\a_5,\ \a_0+2\a_1+3\a_2+2\a_3+\a_4+2\a_5\}$,\\
$C_{\a_1+\a_2+\a_3+\a_4}=\{-\a_0,\ -\a_5\} \cup \{ \a_1,\ \a_4,\  \a_0+\a_1,\  \a_1+\a_2,\ \a_2+\a_3,\  \a_3+\a_4,\  \a_1+\a_2+\a_3,\ \a_2+\a_3+
\a_4,\ \a_1+\a_2+\a_5,\ \a_1+\a_2+\a_3+\a_4,\ \a_0+\a_1+\a_2+\a_3,\ \a_2+\a_3+\a_4+\a_5,\ \a_1+2\a_2+\a_3+\a_5,\ \a_1+\a_2+
\a_3+\a_4+\a_5,\ \a_0+2\a_1+2\a_2+\a_3+\a_5,\ \a_1+2\a_2+\a_3+\a_4+\a_5,\ \a_1+2\a_2+2\a_3+\a_4+\a_5,\ \a_0+\a_1+
\a_2+\a_3+\a_4,\
\a_0+\a_1+\a_2+\a_3+\a_4+\a_5,\ \a_0+2\a_1+2\a_2+2\a_3+\a_4+\a_5,\ \a_0+2\a_1+2\a_2+\a_3+\a_4+
\a_5,\ \a_0+\a_1+2\a_2+2\a_3+\a_4+\a_5,\ \a_0+2\a_1+3\a_2+2\a_3+\a_4+\a_5,\ \a_0+2\a_1+3\a_2+2\a_3+\a_4+2\a_5\}$,\\
$C_{\a_0+\a_1+\a_2+\a_5}=\{-\a_3,\ -\a_4\} \cup \{ \a_0,\ \a_4,\ \a_5,\  \a_0+\a_1,\  \a_1+\a_2,\ \a_2+\a_5,\  \a_0+\a_1+\a_2,\ \a_2+\a_3+\a_5,\ \a_1+
\a_2+\a_5,\ \a_0+\a_1+\a_2+\a_5,\ \a_0+\a_1+\a_2+\a_3,\ \a_2+\a_3+\a_4+\a_5,\ \a_1+2\a_2+\a_3+\a_5,\ \a_0+\a_1+\a_2+
\a_3+\a_5,\ \a_0+2\a_1+2\a_2+\a_3+\a_5,\ \a_0+\a_1+2\a_2+\a_3+\a_5,\ \a_1+2\a_2+\a_3+\a_4+\a_5,\ \a_0+\a_1+\a_2
+\a_3+\a_4,\
\a_0+\a_1+\a_2+\a_3+\a_4+\a_5,\ \a_0+\a_1+2\a_2+\a_3+\a_4+\a_5,\ \a_0+2\a_1+2\a_2+\a_3+\a_4+\a_5,\
\a_0+\a_1+2\a_2+2\a_3+\a_4+\a_5,\ \a_0+2\a_1+3\a_2+2\a_3+\a_4+\a_5,\ \a_0+2\a_1+3\a_2+2\a_3+\a_4+2\a_5\}$,\\
$C_{\a_1+\a_2+\a_3+\a_5}=\{-\a_0,\ -\a_4\} \cup \{ \a_1,\ \a_3,\ \a_5,\  \a_0+\a_1,\  \a_3+\a_4,\  \a_1+\a_2+\a_3,\ \a_2+\a_3+\a_5,\ \a_1+\a_2+\a_5,\ \a_1+\a_2+\a_3+\a_5,\ \a_0+\a_1+\a_2+\a_3+\a_5,\ \a_1+\a_2+\a_3+\a_4+\a_5\}$,\\
$C_{\a_1+2\a_2+\a_3+\a_5}=\{-\a_0,\ -\a_4\} \cup \{\a_1+\a_2,\ \a_2+\a_5,\ \a_2+\a_3,\ \a_1+\a_2+\a_3,\ \a_1+2\a_2+\a_3+\a_5,\ \a_1+\a_2+\a_3+\a_4,\ \a_1+\a_2+\a_5,\ \a_2+\a_3+\a_5,\ \a_0+\a_1+\a_2+\a_3,\ \a_0+\a_1+\a_2+\a_5,\ \a_2+\a_3+\a_4+\a_5,\ \a_0+2\a_1+2\a_2+\a_3+\a_5,\ \a_0+\a_1+2\a_2+\a_3+\a_5,\ \a_1+2\a_2+\a_3+\a_4+\a_5,\ \a_1+2\a_2+2\a_3+\a_4+\a_5,\  \a_0+2\a_1+2\a_2+\a_3+\a_4+\a_5,\ \a_0+\a_1+2\a_2+2\a_3+\a_4+\a_5,\ \a_0+2\a_1+3\a_2+2\a_3+\a_4+\a_5,\ \a_0+2\a_1+3\a_2+2\a_3+\a_4+2\a_5\}$,\\
$C_{\a_1+\a_2+\a_3+\a_4+\a_5}=\{-\a_0\} \cup \{\a_1+\a_2,\ \a_2+\a_5,\ \a_2+\a_3,\ \a_1+\a_2+\a_3,\ \a_2+\a_3+\a_5,\ \a_1+\a_2+\a_5,\ \a_1+\a_2+\a_3+\a_4,\ \a_2+\a_3+\a_4+\a_5,\ \a_1+\a_2+\a_3+\a_5,\ \a_1+\a_2+\a_3+\a_4+\a_5,\ \a_0+\a_1+\a_2+\a_3+\a_5,\ \a_0+\a_1+\a_2+\a_3+\a_4+\a_5,\ \a_1+2\a_2+2\a_3+\a_4+\a_5,\ \a_0+2\a_1+2\a_2+2\a_3+\a_4+\a_5\},\\$
$C_{\a_0+2\a_1+2\a_2+\a_3+\a_5}=\{-\a_4\} \cup \{\a_0+\a_1,\ \a_1+\a_2,\ \a_1+\a_2+\a_3,\ \a_2+\a_3+\a_5,\ \a_1+\a_2+\a_5,\ \a_0+\a_1+\a_2+\a_3,\ \a_0+\a_1+\a_2+\a_5,\ \a_1+\a_2+\a_3+\a_4,\ \a_1+2\a_2+\a_3+\a_5,\ \a_0+\a_1+\a_2+\a_3+\a_5,\ \a_0+2\a_1+2\a_2+\a_3+\a_5,\ \a_1+2\a_2+2\a_3+\a_4+\a_5,\ \a_0+\a_1+\a_2+\a_3+\a_4+\a_5,\  \a_0+2\a_1+2\a_2+2\a_3+\a_4+\a_5,\ \a_0+2\a_1+2\a_2+\a_3+\a_4+\a_5,\  \a_0+2\a_1+3\a_2+2\a_3+\a_4+2\a_5\}$,\\
$C_{\a_0+\a_1+2\a_2+\a_3+\a_5}=\{-\a_4\} \cup \{\a_0,\ \a_1+\a_2,\ \a_2+\a_3,\ \a_2+\a_5,\ \a_0+\a_1+\a_2,\ \a_2+\a_3+\a_5,\  \a_0+\a_1+\a_2+\a_3,\ \a_0+\a_1+\a_2+\a_5,\ \a_2+\a_3+\a_4+\a_5,\  \a_1+2\a_2+\a_3+\a_5,\ \a_0+\a_1+2\a_2+\a_3+\a_5,\ \a_1+2\a_2+\a_3+\a_4+\a_5,\ \a_0+\a_1+\a_2+\a_3+\a_4,\ \a_0+\a_1+2\a_2+\a_3+\a_4+\a_5,\  \a_0+\a_1+2\a_2+2\a_3+\a_4+\a_5,\ \a_0+2\a_1+3\a_2+2\a_3+\a_4+\a_5\}$,\\
$C_{\a_0+\a_1+\a_2+\a_3+\a_4}=\{-\a_5\} \cup \{\a_0,\ \a_4,\ \a_1+\a_2,\ \a_2+\a_3,\ \a_3+\a_4,\ \a_0+\a_1,\ \a_0+\a_1+\a_2,\  \a_1+\a_2+\a_3+\a_4,\ \a_0+\a_1+\a_2+\a_3,\ \a_0+\a_1+\a_2+\a_5,\ \a_2+\a_3+\a_4+\a_5,\ \a_0+\a_1+\a_2+\a_3+\a_4,\ \a_0+\a_1+2\a_2+\a_3+\a_5,\ \a_2+\a_3+\a_4,\ \a_1+2\a_2+\a_3+\a_4+\a_5,\ \a_0+\a_1+\a_2+\a_3+\a_4+\a_5,\ \a_0+2\a_1+2\a_2+\a_3+\a_4+\a_5,\ \a_0+\a_1+2\a_2+2\a_3+\a_4+\a_5,\ \a_0+\a_1+2\a_2+\a_3+\a_4+\a_5,\ \a_0+2\a_1+3\a_2+2\a_3+\a_4+\a_5\}$,\\
$C_{\a_0+2\a_1+2\a_2+\a_3+\a_4+\a_5}=\{\a_4,\ \a_1+\a_2,\ \a_0+\a_1,\ \a_1+\a_2+\a_5,\  \a_1+\a_2+\a_3+\a_4,\ \a_0+\a_1+\a_2+\a_3,\  \a_0+\a_1+\a_2+\a_5,\ \a_2+\a_3+\a_4+\a_5,\ \a_1+2\a_2+\a_3+\a_5,\ \a_0+2\a_1+2\a_2+\a_3+\a_5,\ \a_1+2\a_2+\a_3+\a_4+\a_5,\ \a_0+\a_1+\a_2+\a_3+\a_4+\a_5,\ \a_0+\a_1+2\a_2+2\a_3+\a_4+\a_5,\ \a_0+2\a_1+2\a_2+\a_3+\a_4+\a_5,\ \a_0+
\a_1+\a_2+\a_3+\a_4,\ \a_0+2\a_1+3\a_2+2\a_3+\a_4+2\a_5,\ \a_0+2\a_1+3\a_2+2\a_3+\a_4+\a_5\}$,\\
$C_{\a_0+2\a_1+2\a_2+2\a_3+\a_4+\a_5}=\{\a_3+\a_4,\ \a_0+\a_1,\ \a_1+\a_2+\a_3,\  \a_1+\a_2+\a_5,\  \a_2+\a_3+\a_5,\ \a_1+\a_2+\a_3+\a_4,\ \a_0+\a_1+\a_2+\a_3,\ \a_1+\a_2+\a_3+\a_4+\a_5,\ \a_0+\a_1+\a_2+\a_3+\a_5,\ \a_0+2\a_1+2\a_2+\a_3+\a_5,\ \a_1+
2\a_2+2\a_3+\a_4+\a_5,\  \a_0+\a_1+\a_2+\a_3+\a_4+\a_5,\ \a_0+2\a_1+2\a_2+2\a_3+\a_4+\a_5\}$,\\
$C_{\a_0+\a_1+\a_2+\a_3+\a_4+\a_5}=\{\a_0,\ \a_4,\ \a_5,\ \a_3+\a_4,\ \a_0+\a_1,\ \a_1+\a_2+\a_5,\  \a_2+\a_3+\a_5,\ \a_1+\a_2+\a_3+\a_4,\ \a_0+\a_1+\a_2+\a_3,\ \a_0+\a_1+\a_2+\a_5,\ \a_2+\a_3+\a_4+\a_5,\ \a_1+\a_2+\a_3+\a_4+\a_5,\ \a_0+\a_1+\a_2+\a_3+\a_5,\ \a_0+2\a_1+2\a_2+\a_3+\a_5,\ \a_1+2\a_2+2\a_3+\a_4+\a_5,\ \a_0+\a_1+\a_2+\a_3+\a_4,\  \a_0+2\a_1+2\a_2+\a_3+\a_4+\a_5,\ \a_0+\a_1+\a_2+\a_3+\a_4+\a_5,\ \a_0+2\a_1+2\a_2+2\a_3+\a_4+\a_5,\  \a_0+\a_1+2\a_2+2\a_3+\a_4+\a_5,\
\a_0+2\a_1+3\a_2+2\a_3+\a_4+2\a_5\}$,\\
$C_{\a_0+2\a_1+3\a_2+2\a_3+\a_4+2\a_5}=\{\a_1+\a_2+\a_5,\  \a_2+\a_3+\a_5,\ \a_1+\a_2+\a_3+\a_4,\ \a_0+\a_1+\a_2+\a_3,\ \a_0+\a_1+\a_2+\a_5,\ \a_1+2\a_2+\a_3+\a_5,\ \a_0+2\a_1+2\a_2+\a_3+\a_5,\ \a_1+2\a_2+2\a_3+\a_4+\a_5,\ \a_0+2\a_1+2\a_2+\a_3+\a_4+\a_5,\ \a_2+\a_3+\a_4+\a_5,\ \a_0+\a_1+2\a_2+2\a_3+\a_4+\a_5,\ \a_0+\a_1+\a_2+\a_3+\a_4+\a_5,\ \a_0+2\a_1+3\a_2+2\a_3+\a_4+2\a_5\}$,\\
$C_{\a_0+2\a_1+3\a_2+2\a_3+\a_4+\a_5}=\{\a_1+\a_2,\ \a_2+\a_3,\  \a_2+\a_3+\a_4+\a_5,\ \a_1+\a_2+\a_3+\a_4,\ \a_0+\a_1+\a_2+\a_3,\ \a_0+\a_1+\a_2+\a_5,\ \a_1+2\a_2+\a_3+\a_5,\ \a_0+\a_1+2\a_2+\a_3+\a_5,\ \a_1+2\a_2+\a_3+\a_4+\a_5,\ \a_0+\a_1+\a_2+\a_3+\a_4,\ \a_0+\a_1+2\a_2+2\a_3+\a_4+\a_5,\ \a_0+2\a_1+2\a_2+\a_3+\a_4+\a_5,\ \a_0+2\a_1+3\a_2+2\a_3+\a_4+\a_5\}$,\\
$C_{\a_0+\a_1+2\a_2+\a_3+\a_4+\a_5}=\{\a_0,\ \a_4,\ \a_1+\a_2,\ \a_2+\a_3,\  \a_2+\a_5,\ \a_2+\a_3+\a_4,\ \a_0+\a_1+\a_2,\ \a_2+\a_3+\a_4+\a_5,\ \a_0+\a_1+\a_2+\a_5,\ \a_0+\a_1+2\a_2+\a_3+\a_5,\ \a_1+2\a_2+\a_3+\a_4+\a_5,\ \a_0+\a_1+\a_2+\a_3+\a_4,\ \a_0+\a_1+2\a_2+\a_3+\a_4+\a_5\}$.

\end{appendix}

\end{document}